\documentclass[reqno, 12pt]{amsart}

\RequirePackage[OT1]{fontenc}
\RequirePackage{amsthm,amsmath,natbib, amssymb, mathrsfs, graphicx, listings}
\usepackage{hyperref}
\usepackage{setspace} 
\usepackage{lscape}
\usepackage{boxedminipage}
\usepackage{booktabs,caption,fixltx2e}
\usepackage[flushleft]{threeparttable}

\newtheoremstyle{theorem}
{10pt} 
{10pt} 
{\sl} 
{\parindent} 
{\bf} 
{. } 
{ } 
{} 
\newtheorem{prop}{Proposition}
\newtheorem{lem}{Lemma}
\newtheorem{thm}{Theorem}

\newtheorem{cor}{Corollary}
\newtheorem{defn}{Definition}
\newtheorem{assp}{Assumption}
\theoremstyle{remark}

\newtheorem{rk}{Remark}
\newtheorem*{ex}{Example}

\def\d{{\rm d}}

\newcommand{\inv} { {-1} }

\def\E{\mathbb{E}}

\def\P{\mathbb{P}}

\def\Pter{\mathrm{P}}
\def\Eter{\mathrm{E}}
\def\supp{\mathrm{supp}}
\def\[{[\![}
\def\]{]\!]}
\def\ind{\mathbf{l}}
\def\T{\mathbf{\Theta}}
\def\d{\mathrm{d}}
\def\e{\mathrm{e}}
\def\R{\mathbf{R}}
\def\N{\mathbf{N}}
\def\Omegabf{\mathbf{\Omega}}
\def\Sbf{\mathbf{S}}

\def\D{\mathbf{D}}

\def\diag{\mathrm{diag}}

\def\hP{  \widehat{ \P \circ {\theta^\bullet_T}^{-1} }  }
\def\hF{  \hat{F}_{\theta^\bullet_T} }

\def\tP{\P \circ {\theta^\bullet_T}^{-1}}               
\title[Econometric inference and multiple use of the same data]{On econometric inference \\and multiple use of the same data}

\author{Benjamin Holcblat}
\address{Department of Finance\\
BI Norwegian Business School\\
Oslo, Norway 0484}
\email{Benjamin.Holcblat@bi.no}

\author{Steffen Gr{\o}nneberg}
\address{Department of Economics\\
BI Norwegian Business School\\
Oslo, Norway 0484}
\email{steffeng@gmail.com}

\usepackage{amsmath}
\usepackage{amsfonts}
\usepackage{graphicx}
\usepackage{amssymb}

\begin{document}

\begin{abstract}

In fields that are mainly nonexperimental, such as economics and finance, it is inescapable to compute test statistics  and confidence regions  that are not probabilistically independent from previously examined data.  The Bayesian and Neyman-Pearson inference theories are known to be inadequate  for such a practice. We show that these inadequacies also hold  m.a.e.  (modulo approximation error). We develop a general econometric  theory, called  the neoclassical  inference theory, that is immune to this inadequacy m.a.e.   The neoclassical inference theory  appears to nest model calibration, and most econometric practices, whether they are labelled Bayesian or \`a la Neyman-Pearson. We derive a general, but simple adjustment to make standard errors account for the approximation error.







\bigskip

{\noindent \textit{Keywords}:     Hypothesis testing;   Confidence region; Estimation; Model calibration.}

\medskip

{\noindent \textit{JEL classification}: C1.}

\end{abstract}

\date{\today}

\maketitle  

\doublespace

\section{Introduction }\label{sec1}

\indent By definition, in nonexperimental fields, new data cannot be generated.  Consequently, it is inescapable to compute  test statistics and confidence regions  that are not probabilistically independent from previously examined data.  By the Skorohod's representation (1976), this practice is equivalent to using twice the same data, so that we call it  \textit{multiple use of the same data}.\footnote{The Skorohod's representation (1976) states that, for any two Borel random variables $ Y$ and $Z $, there exist a Borel random variable $U  $ independent from $Y$, and a Borel function $h(.,.)$ such that   $Z=h(Y,U)$. Thus, if $Y$ and $Z $ are not independent, using first $Y$, and then $Z$ is equivalent to using  first $Y$, and then reusing $Y$ with $U$. } 
   The main objective of this paper is to develop a general econometric theory, called  the \textit{neoclassical  inference theory},\footnote{There are two reasons for this name. Firstly, it is a classical theory, in the \textit{statistical}  sense of the term, i.e.,  in this theory, the  unknown parameter $\theta_0$ is not treated as a random variable, but as a constant. Secondly, it is neoclassical in the \textit{historical} sense of the term: it    seems to formalize underlying principles of work  by classical authors (e.g., Laplace, 1812/1820, livre II, chap. 3; Fisher, 1925/1973, part V).} that is adequate for multiple use of the same data m.a.e. (modulo approximation error). The Bayesian and Neyman-Pearson inference theories  are not adequate for such a practice, even m.a.e. 
  Thus, if we set aside approximation errors, for which we provide an adjustment,  this paper    elucidates why   econometric inference is possible in fields that are mainly nonexperimental, such as economics and finance. 


\subsection{Key idea}

In Bayesian and Neyman-Pearson theories, test statistics  and confidence regions are not independent from   data as the former are functions of the latter even  m.a.e.  Then, if the same realized data are re-used to compute new test statistics  or confidence regions, distributions  conditional on   the   previously observed statistics should be considered (e.g., Lehmann and Romano, 1959/2005, sec. 10.1 for Neyman-Pearson theory; Savage, 1954/1972, sec. 3.5 for Bayesian theory). We show that this conditioning, which is typically ignored in practice,  is a challenge for Bayesian and Neyman-Pearson theories. The neoclassical inference theory circumvents this conditioning. The key idea    is to use realized data to approximate  the  distribution of  random variables, called \textit{generic  proxies}, that have the same unconditional  distribution as data-based statistics, but are probabilistically independent from the  realized data.   


\begin{ex} \label{ex} 
Let  $X_{1:T}:=\left(X_{t} \right)_{t=1}^T $ be data that are  assumed to be $T$ i.i.d. (independent and identically distributed) random variables following a Gaussian distribution with mean $\theta_{0}$ and standard deviation $s$, denoted $\mathcal{N}(\theta_0,s) $.    The realized data are $X_{1:T}(\omega):=\left(X_{t}(\omega)\right)_{t=1}^T $ where $\omega$ denotes an element of the sample space $\Omegabf $. We want to make inference about the
unknown parameter $\theta_0=\E (X_1)$ through its finite-sample proxy, the average.  Now, consider   generic data   $X_{1:T}^{\bullet}:=\left(X_{t}^{\bullet} \right)_{t=1}^T $ that are independently generated by the same Gaussian distribution as the data  $X_{1:T} $, i.e., $X_{1:T} $ and $X_{1:T}^{\bullet}$ are independent, but have the same unconditional distribution. Then, the  average of the data $\overline{X}_T$, denoted $\overline{X}_T:=\frac{1}{T}\sum_{t=1}^TX_t $, and the average of the generic data, denoted $\overline{X}_T^{\bullet} $, have  the same unconditional distribution, $\mathcal{N}(\theta_0, \frac{s}{\sqrt{T}})$,  and are equally informative about $\theta_0$. Nevertheless, previous knowledge of the  realized data $X_{1:T}(\omega) $ typically affects the distribution of  $\overline{X}_T $,  but does  not affect the distribution of the generic  proxy $\overline{X}_T^{\bullet}$.  For example, if  the realized average $\overline{X}_T(\omega) $ is  known from a previous study, there is no uncertainty about it, so that its  distribution is a Dirac at the realized average (i.e., $ \delta_{\overline{X}_T(\omega) }$), while the distribution of $\overline{X}_T^{\bullet}$ is still  the same Gaussian distribution, $\mathcal{N}(\theta_0, \frac{s}{\sqrt{T}})$.
Thus, the idea of  the neoclassical inference theory is to rely on  an approximation of the distribution of  $\overline{X}_T^{\bullet}$ to make inference about $\theta_0$. Although the data $X_{1:T}$ are independent from the generic proxy $\overline{X}_T^{\bullet}$, their observation provides an approximation of its distribution.   For example, by the Lindeberg-L\'evy CLT (central limit theorem),  a Gaussian distribution centered at  the realized average, $\overline{X}_T(\omega) $,   with  standard deviation $\frac{s_T (X_{1:T}(\omega))}{\sqrt{T}}:=\frac{ \sqrt{\frac{1}{T}\sum_{t=1}^T\left(X_t(\omega)-\overline{X}_T(\omega)\right)^2}}{\sqrt{T}} $,    is an approximation of the distribution of  $\overline{X}_T^{\bullet}$, which is $\mathcal{N}(\theta_0, \frac{s}{\sqrt{T}})$.
\hfill $\diamond $ 
\end{ex}
\noindent A similar idea is present in Monte-Carlo simulation methods\,:   The observation of realized random variables enables the approximation of the distribution of a generic random  variable, which is independent from the realized ones. In fact, this similarity has a mathematical underpinning under the standard assumption of ergodicity, which stipulates an equivalence between exploration of  the sample space and exploration of  the time dimension.

In a way, the neoclassical inference theory generalizes the immunity of the standard  justification of point estimators, consistency, to  confidence intervals and tests. Unlike the Neyman-Pearson justifications for tests and confidence regions, consistency is immune to multiple use of the same data m.a.e. A consistent point estimator of a parameter $\theta_0$ does not depend on the realized data m.a.e.: by the definition of consistency,  for almost all possible realizations of the data, such a point estimator is arbitrary close to the  fixed parameter $\theta_0$ m.a.e. See Appendix \ref{ap1} on p. \pageref{ap1} for  a formal statement.  


\subsection{Literature overview}

In the statistical and econometric  literature, the issue raised by multiple use of the same data for Neyman-Pearson and Bayesian         inference theories has been  occasionally discussed.  E.g., Lehmann and Romano, 1959/2005, sec. 10.1 for Neyman-Pearson theory;  Berger, 1980/2006, pp. 112--113 and 284 for Bayesian theory; Leamer, 1978, pp. v--vii for an assessment of the acuteness of the issue, and chap. 9 for  an ad hoc proposal  to mitigate the issue for Bayesian inference.  The common wisdom seems to be that the issue is unavoidable.  To the best of our knowledge, no general formal solution has been proposed even m.a.e. Holcblat (2012) relies on the  idea behind   the neoclassical theory only in the particular case of the empirical saddlepoint (ESP) approximation. 

Multiple use of the same data is  \textit{not} treated in the large literature about multiple hypothesis testing (e.g., Lehmann and Romano, 1959/2005, chap. 9 for  a perspective \`a la Neyman-Pearson; Berger, 1980/2006, chap. 7  for a Bayesian perspective).  In this literature, it is assumed that  the set of all statistics to be potentially computed is determined before examination  of the data.
 This situation does not correspond to nonexperimental fields
as their evolution is often the result of a hard-to-predict dialogue between theory and empirical studies based on more or less the same realized data. For example, Compustat, CRSP (Center for Research in Security Prices)  and BEA (Bureau of Economic Analysis) data have been  re-used in numerous empirical studies in corporate finance, asset pricing and macroeconomics, respectively.

\subsection{Organization of the paper}

The paper is organized as follows. Section \ref{sec2} and \ref{sec3}, respectively, show that Neyman-Pearson and Bayesian inference theories  are inadequate for multiple use of the same data even m.a.e.   Section \ref{sec4} presents  elements of the neoclassical theory, and proves  its immunity to multiple use of the same data m.a.e. Section \ref{sec5} revisits model calibration and prominent econometric practices from a neoclassical point of view, and presents a simple adjustment to make standard errors account for the approximation error. An important point to note is that this standard-error adjustment holds under the usual   $\sqrt{T} $-asymptotic normality assumption, thus applying to a large part of econometric practice. 
Some readers may find sections \ref{sec2} and \ref{sec3} obvious, but the latter  should be considered in comparison with  section  \ref{sec4}. The contribution of  this paper is essentially theoretical, and  \textit{not} mathematical. 
Applied econometricians might want to focus on subsection \ref{sec5_3}, which assesses the most common econometric practice from a neoclassical point of view.

\begin{rk} \label{rk1}

In accordance with the main objective of this paper,  we reason m.a.e.  in sections  \ref{sec2}- \ref{sec4}.       For the Neyman-Pearson theory, ignoring approximation errors means that   we consider the asymptotic limit superior (or  limit inferior) of the outer  (or inner) probability distribution to be exact for the given sample size, when the sampling distribution is not available.
For the Bayesian theory, this has no bearing because probability distributions are assumed to be perfectly known by the econometrician (e.g., Savage, 1954/1972, pp. 59--60). For the
neoclassical inference theory, this means that  we consider the approximation of the  sampling distribution of  the   generic proxy  to be exact. Unlike sections  \ref{sec2}- \ref{sec4}, section \ref{sec5} does not reason m.a.e., and treats of the approximation error from a neoclassical point of view.
\hfill $\diamond $ 
\end{rk}
   
\section{Neyman-Pearson theory and multiple use of the same data}\label{sec2}

In this section, we explain the m.a.e. inadequacy of the Neyman-Pearson theory for multiple use of the same data. Subsection \ref{sec2_1} informally explains it in the standard case of asymptotic $t$-statistics. Subsection  \ref{sec2_2} formalizes it in the general case.  

\subsection{The case of asymptotic $t$-statistics} \label{sec2_1}

\qquad  Asymptotic $t$-statistics are among the most widely-used statistics to compute confidence regions or carry out hypothesis tests. The Neyman-Pearson theoretical justification of an asymptotic $t$-test of size $\alpha$ is that the $t$-statistic has a  probability $1-\alpha$ m.a.e.  to be between the $\alpha/2$ and  $1-\alpha/2$ quantiles of a standard Gaussian distribution under the test hypothesis. However,  once computed, the $t$-statistic is in the non-rejection region with probability 0 or 1, i.e.,  it \textit{is} or it is \textit{not }in the  non-rejection region. Thus, if the result of this first test leads the econometrician to compute a second $t$-test of size $\alpha$,  the corresponding $t$-statistic cannot typically have a probability of $1-\alpha$ m.a.e.
to be between the $\alpha/2$ and  $1-\alpha/2$ quantiles of a standard Gaussian distribution under the test hypothesis. The observation  of the first $t$-statistic has removed a part of the randomness of the second $t$-statistic.  Except in a few cases  (e.g., Gouri\'eroux and Monfort, 1989/1996, chap. 19), $t$-statistics computed on  the same data set are not independent. This means that the Neyman-Pearson theoretical justification does not hold for the second $t$-test. Because of the duality between hypothesis testing and confidence regions in the Neyman-Pearson theory, there is the same concern for confidence intervals based on $t$-statistics. Subsection \ref{sec2_2} proves that this concern about the Neyman-Pearson theoretical justification of confidence regions and tests is not limited to $t$-statistics.

\subsection{The general case} \label{sec2_2}

Assumption \ref{assp1} sets up  the minimal elements of the Neyman-Pearson theory that are necessary to formalize multiple use of the same data.

\begin{assp}\label{assp1}
\textbf{\emph{(a)}} Let $(\Omegabf, \mathcal{E}_\Omegabf)$ be a measurable  space where  $\mathcal{E}_\Omegabf$ denotes a $\sigma$-algebra of $\Omegabf$.    \emph{\textbf{(b)}} Let $\theta_0 \in \T$ be the unknown parameter, where $\T$ denotes the parameter space. \emph{\textbf{(c)}} Let $X_{1:T}$  be some data, i.e., a measurable  mapping from  $(\Omegabf, \mathcal{E}_\Omegabf )$ to a measurable space $(\underline{\Sbf}_{T},\underline{\mathcal{S}}_T)$, where $T$ and $\underline{\Sbf}_{T} $ denote the sample size and  the observation space, respectively.   
\end{assp}

\begin{rk} There is no restriction on the parameter space $\T $, so that  it can be Euclidean or infinite-dimensional. \hfill $\diamond $
\end{rk}
Definitions \ref{defn2} and \ref{defn9} recall the definition of Neyman-Pearson confidence regions and tests.

\begin{defn}[Neyman-Pearson confidence region]\label{defn2} Let $\alpha \in [0,1]$, and $\mathrm{P}$ a probability measure on $(\Omegabf, \mathcal{E}_\Omegabf)$.  Under $\Pter $, a $1-\alpha$ Neyman-Pearson confidence region $C_{1-\alpha,T}$  is a measurable random  subset of the parameter space that has a probability of at least $1- \alpha $ m.a.e. to contain the unknown parameter $\theta_0$, i.e.,  (i) for all $\omega \in \Omegabf$, $C_{1-\alpha,T}(X_{1:T}(\omega))\subset \T$, (ii) $\left\{ x_{1:T} \in \underline{\Sbf}_T: \theta_0 \in C_{1-\alpha,T}(x_{1:T})\right\} \in \underline{\mathcal{S}}_T$ m.a.e., and (iii)   $\mathrm{P}\left\{ \omega \in \Omegabf: \theta_0 \in C_{1-\alpha,T}(X_{1:T}(\omega))\right\} \geqslant 1-      \alpha$ m.a.e.
\end{defn}

\begin{defn}[Neyman-Pearson test]\label{defn9} Let   $\mathrm{H}$ be  a test hypothesis, and $\mathrm{P}$ a probability measure on $(\Omegabf, \mathcal{E}_\Omegabf)$.  Define the measurable decision space $(\D,\mathcal{P}(\D) )$ where $\D := \left\{d_\mathrm{H},d_\mathrm{A}\right\} $,  and $ \mathcal{P}(\D)$ denotes the power set of $\D$. The decisions $d_\mathrm{H}$ and $d_\mathrm{A}$, respectively, correspond to  the non-rejection and the rejection of the test hypothesis $\mathrm{H}$. Under $\Pter $,  a Neyman-Pearson test of level $\alpha\in [0,1]$ is a decision rule $d_{T}(.)$ that leads to the rejection  of $\mathrm{H}$ with  a probability  of at most $\alpha$ m.a.e. under $\mathrm{H}$, i.e.,   a $\underline{\mathcal{S}}_T/\mathcal{P}(\D)$-measurable function $d_{T}$ m.a.e. s.t.   $\mathrm{P} ( d_{T}(X_{1:T}) = d_\mathrm{A}) \leqslant \alpha$ m.a.e., if $\mathrm{H}$ is true.
\end{defn}
\begin{rk} As indicated by the qualification ``m.a.e.," when no  finite-sample distribution is available, we consider the asymptotic limit superior (or  limit inferior) of the outer  (or inner) probability distribution.  (see Remark \ref{rk1} on p. \pageref{rk1}). Thus, our setup covers
asymptotic Neyman-Pearson tests and confidence regions, and the case \`a la Hoffmann-J{\o}rgensen (see Wellner and van der Vaart, 1996), in which finite-sample statistics are not measurable although their limit is measurable.
 In the latter case, in Definitions \ref{defn2} and \ref{defn9},   $\mathrm{P}(\theta_0 \in C_{1-\alpha,T}(X_{1:T}))$ and $\mathrm{P} ( d_{T}(X_{1:T}) = d_\mathrm{A}) $,  respectively,  stand for  $\lim \inf_{T \rightarrow \infty} \mathrm{P}_*(\theta_0 \in C_{1-\alpha,T}(X_{1:T}))$ and $\lim \sup_{T \rightarrow \infty} \mathrm{P}^* ( d_{T}(X_{1:T}) = d_\mathrm{A})$ where $\Pter_* $ and $\mathrm{P}^*$, respectively, denote the  inner and  outer probabilities implied by $\Pter$.
  \hfill $\diamond $
\end{rk}

Theorem  \ref{thm1} formalizes the concern raised by  previous knowledge of   realized data that are not probabilistically independent from  Neyman-Pearson confidence regions and tests.  

\begin{thm}[Neyman-Pearson inadequacy] \label{thm1} Let $\P$ be   the unknown probability measure on $(\Omegabf, \mathcal{E}_\Omegabf)$, and $ \left\{ X_{1:T} \in A_T\right\}\in\mathcal{E}_\Omegabf$     a nonzero-probability event, i.e., $\P  (X_{1:T} \in A_T) =c>0$.
 For all $E \in \mathcal{E} $,  define $\P(E|X_{1:T} \in A_T) := \frac{\P\left(E\cap\left\{ X_{1:T} \in A_T\right\}\right)}{\P( X_{1:T} \in A_T)}  $.
Denote  a Neyman-Pearson $1-\alpha $ confidence region for $\theta_0$ under $\P$ with $C_{1-\alpha,T} $, and a Neyman-Pearson test of level $\alpha$ under $\P$ with $d(.)$. Under Assumption \ref{assp1},
\begin{enumerate}
\item[i)]  if  $\left\{ X_{1:T} \in A_T\right\}$ and $\left\{  \theta_0 \in C_{1-\alpha,T}(X_{1:T})\right\} $ are not independent m.a.e.,  then 
\begin{eqnarray*}
\P( \theta_0 \in C_{1-\alpha,T}(X_{1:T})|X_{1:T} \in A_T) \neq \P(\theta_0 \in C_{1-\alpha,T}(X_{1:T})) \text{ m.a.e.};
\end{eqnarray*}
\item[ii)] if  $\left\{ X_{1:T} \in S\right\}$ and $\left\{ d_T(X_{1:T}) = d_A\right\} $ are not independent m.a.e.,  then 
\begin{eqnarray*}
\P(  d_T(X_{1:T}) = d_\mathrm{A}|X_{1:T} \in A_T) \neq \P( d_T(X_{1:T}) = d_\mathrm{A}) \text{ m.a.e.} 
\end{eqnarray*}
\end{enumerate}

\end{thm}
\begin{proof} It is definition chasing, essentially. For (i) and (ii), respectively denote $\left\{  \theta_0 \in C_{1-\alpha,T}(X_{1:T})\right\}$ and $\left\{   d_T(X_{1:T}) = d_\mathrm{A}\right\} $ with $E$. By definition of independence between events, m.a.e., $\P( E\cap \left\{ X_{1:T} \in A_T\right\})  \neq  \P(E ) \P (X_{1:T} \in A_T) \stackrel{(a)}{\Leftrightarrow} \frac{\P\left(E\cap\left\{ X_{1:T} \in A_T\right\}\right)}{\P( X_{1:T} \in A_T)}  \neq  \P (E ) \stackrel{(b)}{\Leftrightarrow}  \P( E| X_{1:T} \in A_T)  \neq \ \P(E )$, where equivalences  can be seen as follows. \textit{(a)} By assumption, $ \P( X_{1:T} \in A_T)>0$. \textit{(b)} By definition, $\P(E| X_{1:T} \in A_T) := \frac{\P\left(E\cap\left\{ X_{1:T} \in A_T\right\}\right)}{\P( X_{1:T} \in A_T)}  $.  See Appendix \ref{ap5} on p. \pageref{ap5} for more details regarding the possible approximation error.   
\end{proof}

The key defining properties of  a Neyman-Pearson confidence region and test are, respectively, the probability that the confidence region contains the unknown parameter, and the probability of rejecting the hypothesis under the test hypothesis. Theorem  \ref{thm1} proves that these key defining properties are  affected by the previous observation of a nonzero-probability event $\left\{ X_{1:T} \in A_T\right\} $ that is not probabilistically independent from the  corresponding confidence region and  test.  Before the observation of $\left\{ X_{1:T} \in A_T\right\}$, the probability $\mathrm{P}$ of the Neyman-Pearson Definitions \ref{defn2} and \ref{defn9} is $\P$ m.a.e., but,  after  observation of $\left\{ X_{1:T} \in A_T\right\}$, $\mathrm{P}$ is $\P ( .| X_{1:T} \in A_T) $ m.a.e. 

%
%

A solution would be to systematically account for previous knowledge of the data  by determining conditional probability, such as $\P ( .| X_{1:T} \in A_T) $. However,  most of the time,  this is operationally impossible, especially in nonexperimental fields. In nonexperimental fields, this previous knowledge can correspond to   computed statistics or plots,  but also to historical events personally experienced or studied. For example, defining valid Neyman-Pearson tests or confidence regions for an applied American econometrician who studies the US economy  appears  an impossible task. Moreover, even if it was possible, it would make  criteria of validity of statistical discoveries path-dependent, and thus difficult to understand. Therefore,  the Neyman-Pearson inference theory appears practically inadequate for nonexperimental fields.    

%

\begin{rk}\label{rk2} Because of our focus on multiple use of the same data, in this paper, we present the operational impossibility to condition on previous knowledge of the realized data as the source of the Neyman-Pearson inadequacy. In fact, if one makes the distinction between unknown and random quantities as the Neyman-Pearson theory does (e.g., $\theta_0$ is unknown, but constant), it is the realization of the data and not the  knowledge of them that matters.   Thus, one  needs to condition on all the data that  have been realized prior to the determination of the test statistics and confidence regions to be computed. When only \textit{part} of the data at use   have  been previously realized,   we are back to Theorem \ref{thm1} and the generic operational impossibility to determine conditional probability.   When \textit{all} the  data at use have  been previously realized,   the conditioning is trivial, but  then tests should should have zero probability type I error m.a.e., and, under additional but general assumptions,  $C_{1-\alpha,T}(X_{1:T})=\T $ m.a.e. $\P$-a.s.  See Appendix  \ref{ap4} on p. \pageref{ap4}.\qquad      \text{ }
       \hfill      $\diamond $

\end{rk}

\section{Bayesian theory and multiple use of the same data}\label{sec3}

 \begin{flushright}
\begin{small}``In a strictly logical sense, this criticism of (practical) prior dependence 
on the data cannot be refuted."\\
\begin{scriptsize}  Berger (1980/2006, p. 112).\end{scriptsize}\end{small}
 \end{flushright}

As for Neyman-Pearson theory,   multiple use of the same data is a challenge for Bayesian inference theory. Subsection \ref{sec3_1} explains the concern in the basic   case in which  Bayes' formula holds, and subsection \ref{sec3_2} formalizes it in the general case. 

\subsection{The basic case} \label{sec3_1}
\qquad Taken literally, Bayesian theory regards  inference as a two-stage game between nature and an econometrician (e.g.,  Ferguson, 1967; Borovkov, 1984/1998). In the first stage, nature draws the  parameter $\theta_0 $  according to a prior distribution $\pi_{\theta_0}(.)$, and then  draws data $X_{1:T} $ according to a conditional probability density function (p.d.f.), $ \pi_{X_{1:T}|\theta_0}(.|.)$. In the second stage, the econometrician makes  inferences about the realized  parameter value $\theta_0$ given the sample at hand.\footnote{To avoid additional notations, the random parameter $\theta_0$ is defined as the identity mapping on the parameter space $\T $, so that its realized value is also denoted by $\theta_0 $.  } As usual in game theory, the p.d.f. $\pi_{X_{1:T}|\theta_0}(. |.)$ and $\pi_{\theta_{0}}(.)$ are common knowledge.   Thus, the econometrician updates the prior distribution,  $ \pi_{\theta_{0}}(.)$,  thanks to  data  according to Bayes' formula
\begin{eqnarray*}
\pi_{\theta_{0}|X_{1:T}}\left(\theta|X_{1:T}(\omega)\right)=\frac{\pi_{X_{1:T}|\theta_0}(X_{1:T}(\omega)|\theta)\pi_{\theta_{0}}(\theta)}{\int_{\mathbf{\Theta}} \pi_{X_{1:T}|\theta_0}(X_{1:T}(\omega) |\dot \theta)\pi_{\theta_0}(\dot \theta) \mu(\d\dot \theta)},
\end{eqnarray*}
to obtain  the posterior distribution $\pi_{\theta_0|X_{1:T}}(.|X_{1:T}(\omega)) $.

But, after the p.d.f. of the unknown parameter given the data has been computed, the data  are known and  fixed.   Their randomness has disappeared.  Thus, the econometrician cannot learn anymore from them. If the Bayes formula is applied a second time to the same data, the p.d.f. of data conditional on the unknown parameter  is  one, so that the second posterior is equal to the first posterior.  Mathematically,   Bayesian updating  becomes 
\begin{eqnarray*}
& &  \frac{\pi_{X_{1:T}|\theta_0,X_{1:T}}(X_{1:T}(\omega) |\theta,X_{1:T}(\omega))\pi_{\theta_{0}|X_{1:T}}(\theta|X_{1:T}(\omega))}{\int_{\mathbf{\Theta}} \pi_{X_{1:T}|\theta_0,X_{1:T}}(X_{1:T}(\omega) |\dot\theta,X_{1:T}(\omega))\pi_{\theta_{0}|X_{1:T}}(\dot\theta|X_{1:T}(\omega)) \mu(\d\dot \theta)}\\
 & &\qquad \qquad \qquad \qquad =  \frac{1\times \pi_{\theta_{0}|X_{1:T}}(\theta|X_{1:T}(\omega))}{\int_{\mathbf{\Theta}}1\times\pi_{\theta_{0}|X_{1:T}}(\dot\theta|X_{1:T}(\omega)) \mu(\d\dot \theta)}=\pi_{\theta_{0}|X_{1:T}}(\theta|X_{1:T}(\omega)).
%
\end{eqnarray*}
Therefore,  Bayes inference theory cannot justify multiple use of the same data.   Subsection \ref{sec3_2} shows that the conclusion remains unchanged in the general case, in which densities do not necessarily exist.    

\subsection{The general case}\label{sec3_2}
Assumption \ref{assp2} defines the general structure of   Bayesian inference  along the lines of Florens, Mouchart and Rolin (1990).\footnote{The main difference between their notations and our notations is the following. Unlike them, we do not identify $\sigma$-algebras with their inverse image by the coordinate map. See Florens, Mouchart and Rolin, 1990, p. 11, warning. Our choice makes the presentation less elegant, but it allows us to maintain notational consistency  within this paper.} 
\begin{assp} \label{assp2} \textbf{\emph{(a)}}  Let $(\Omegabf \times \T, \mathcal{E}_{\Omegabf}\otimes \mathcal{E}_{\T}, \Pi)$ be a probability space, where   $(\T, \mathcal{E}_\T )$ is the parameter space, and where $\mathcal{E}_{\Omegabf}\otimes \mathcal{E}_{\T}$ denotes the product $\sigma$-algebra   of the $\sigma$-algebras   $\mathcal{E}_{\Omegabf}$ and $\mathcal{E}_{\T} $.     \textbf{\emph{(b)}} Let  $\left\{\mathcal{F}_{\Omegabf, n}\right\}_{n\geqslant 0}$ be a filtration in $ \mathcal{E}_{\Omegabf}$. Define a    filtration   $\left\{\mathcal{F}_n\right\}_{n\geqslant 0}$ in $\mathcal{E}_{\Omegabf}\otimes \mathcal{E}_{\T} $ s.t., for all $n \in \N $, $\mathcal{F}_n =\mathcal{F}_{\Omegabf,n}\otimes \{ \T ,\emptyset \}$.   
\end{assp}

 The filtration $ \left\{ \mathcal{F}_n\right\}_{n\geqslant 0}  $ corresponds to the accumulation of information that comes from  the sample space. In other words,  $\mathcal{F}_n$ is the information set of the econometrician  after $n$ Bayesian updates.  Definition  \ref{defn1} reminds the general definition of posterior probabilities.

\begin{defn}[Posterior probability]\label{defn1} For all $n\in \N$ and $B \in  \{ \Omegabf ,\emptyset \}\otimes\mathcal{  E}_{\T}  $,
the  $\mathcal{F}_{ n }$-posterior probability of $B $ is the expectation of the indicator function $\ind_{B} $  conditional on $\mathcal{F}_n $, i.e., $\E( \ind_{B}| \mathcal{F}_n)$.  
\end{defn}

Implicitly, Definition \ref{defn1}  also defines priors as the distinction between a prior and a posterior depends on the update of reference. After $n$ Bayesian updates, $\E( .| \mathcal{F}_n)$ is the prior, while $\E( .| \mathcal{F}_{n+1})$ is the posterior. 

\begin{rk} The framework is  general: We do not impose restrictions on the parameter space $\T$, or require the existence of regular conditional probabilities. \hfill $\diamond $
\end{rk}

Assumption \ref{assp3} specify  the minimal additional ingredients necessary to study multiple use of the same data.

\begin{assp} \label{assp3} \textbf{\emph{(a)}}   Let $X_{1:T}$  be some data, i.e., a measurable  mapping from  $(\Omegabf, \mathcal{E} )$ to the measurable space $(\underline{\Sbf}_T,\underline{\mathcal{S}}_T)$, where $\underline{\Sbf}_{T}$ denotes the observation space. \textbf{\emph{(b)}} There exists $ n_{1} \in \N $ s.t. $\mathcal{F}_{ n_1+1 }= \mathcal{F}_{ n_1 }\vee \sigma(X_{1:T}) $, where $ \sigma(X_{1:T})$ denotes the  $\sigma$-algebra generated by $X_{1:T}$, and $\mathcal{F}_{ n_1 }\vee [\sigma(X_{1:T})\otimes \{ \T ,\emptyset \}] $ the $\sigma$-algebra generated by the union of   $\mathcal{F}_{ n_1 }$ and $ [\sigma(X_{1:T})\otimes \{ \T ,\emptyset \}] $, i.e., $\mathcal{F}_{ n_1 }\vee  [\sigma(X_{1:T})\otimes \{ \T ,\emptyset \}]:=\sigma\left(\mathcal{F}_{ n_1 }\cup  [\sigma(X_{1:T})\otimes \{ \T ,\emptyset \}] \right) $. 
\end{assp}

Assumption \ref{assp3}(a) requires the existence of data, $X_{1:T}$, while Assumption \ref{assp3}(b) requires that    update $n_1$ comes from the use of  the data $X_{1:T}$ . Theorem \ref{thm2} formalizes the effect of a second use of the same data $X_{1:T}$.

\begin{thm}[Bayesian  inadequacy] \label{thm2} Let $ n_2 \in \[ n_{1}+1, \infty \[$ s.t $\mathcal{F}_{ n_2+1 }= \mathcal{F}_{ n_2 }\vee  [\sigma(X_{1:T})\otimes \{ \T ,\emptyset \}] $. Then under Assumptions \ref{assp2} and \ref{assp3}, the $\mathcal{F}_{n_2+1 }$-posterior probability and $\mathcal{F}_{ n_2 }$-posterior probability are equal, i.e., for all $B \in  \{ \Omegabf ,\emptyset \}\otimes\mathcal{  E}_{\T}$, 
\begin{eqnarray*}
\E( \ind_{B}| \mathcal{F}_{ n_{2}+1 })= \E( \ind_{B}|  \mathcal{F}_{ n_{2} }).
\end{eqnarray*}
\end{thm}
\begin{proof}  It is definition chasing. By definition, $\mathcal{F}_{ n_2+1 }:=\mathcal{F}_{ n_2 }\vee  [\sigma(X_{1:T})\otimes \{ \T ,\emptyset \}]:=\sigma(\mathcal{F}_{ n_2 }\cup  [\sigma(X_{1:T})\otimes \{ \T ,\emptyset \}] ) \stackrel{(a)}{=}\sigma(\mathcal{F}_{ n_2 })\stackrel{(b)}{=}\mathcal{F}_{ n_2 }$, where equalities  can be seen as follows. \textit{(a)  }By Assumption \ref{assp3}(b), $ [\sigma(X_{1:T})\otimes \{ \T ,\emptyset \}] \subset \mathcal{F}_{ n_2 }$. \textit{(b)}  $\mathcal{F}_{ n_2 }$ is itself a $\sigma$-algebra by definition of a filtration. 
\end{proof}

The update $n_1+1$ corresponds to the first use of the data $X_{1:T}$, while the update $n_2+1$
corresponds to the second use. Theorem \ref{thm2} proves that the second use of the same data does not increase the  information set,  $\mathcal{F}_{n_2} $, and thus the posterior  remains  the same. An immediate corollary of this result is the absence of formal Bayesian  justification for analyses, in which   an econometrician claims to have obtained a different ``posterior" after a first use of the same data.  Theorem \ref{thm2} shows that such  analyses are incompatible with Bayesian inference theory.   

\begin{rk} For brevity and relevance, we mainly consider the Neyman-Pearson and Bayesian inadequacy when there is partial, and complete previous knowledge of the realized data, respectively. However, in parallel,  complete, and partial previous knowledge causes Neyman-Pearson and Bayesian inadequacy, respectively. Complete previous knowledge causes Neyman-Pearson inadequacy  when  the randomness needed to  justify a new test or a new confidence region has completely disappeared. For example, when one wants  a $t$-test of size $\alpha$ after computation of a $1-\alpha$ confidence interval   based on the same statistic,   the $t$-statistic is between the $\alpha/2 $ and $1-\alpha/2 $ quantiles with probability $0 $ or $1 $ because a confidence interval corresponds to the set of hypothesis that would not have been rejected. See also Proposition \ref{prop1} on p. \pageref{prop1}, which can be seen as a formalization of the case, in which all the data have been previously examined.  Partial knowledge causes Bayesian inadequacy when it is impossible to incorporate previous information through a formal Bayesian updating. For example, Bayesian inference theory is typically inadequate when partial previous knowledge corresponds  to historical events personally experienced.   See also  Savage  (1954/1972, pp. 59--60) for more details about this inadequacy.           \hfill $\diamond $ 

\end{rk}

As Neyman-Pearson theory, Bayesian inference theory is inadequate for multiple use of the same data.
Both theories rely on a randomness, which is disappearing as data are used, so that multiple use of the same data appears difficult to justify. However, in nonexperimental fields, multiple use of the same data is inescapable. Thus, the relevance of Neyman-Pearson  and Bayesian inference theories    to nonexperimental fields is not obvious.

\section{Elements of neoclassical inference theory } \label{sec4}


The purpose of this section is to introduce elements of a theory  that is immune to multiple use of the same data m.a.e., and that provides a common framework for point estimation, confidence regions and hypothesis testing. There does not seem to exist such an inference theory in the literature.  

  This section is organized as follows. Subsection \ref{sec4_1} presents the main idea of the neoclassical theory, subsection \ref{sec4_2} its main elements, and subsection \ref{sec4_3} proves that it is theoretically immune to multiple use the same data m.a.e.
Hereafter, for simplicity, we only consider the parameter space $ \T $ to be an Euclidean space.

\subsection{Main idea} \label{sec4_1}  The setup of the neoclassical inference theory is standard (e.g., Borovkov, 1984/1998, chap. 2). An  econometrician wants to infer a constant and  unknown parameter $\theta_0$ of an econometric model $(\Omegabf, \mathcal{\mathcal{E}_\Omegabf}, \P)$. The  parameter  $\theta_0$ is assumed to belong to a known parameter space, denoted $\T$. The only difference between the unknown parameter, $\theta_0$, and other elements of the parameter space $\T $ is that the former one equals a mapping of the generating probability measure,  $\P $, i.e.,
\begin{eqnarray}
\theta_0=:G(\P) \label{eq1}
\end{eqnarray}
 where $G(.)$ maps  probability measures to elements of the parameter space $\T$. The econometrician does not know $\P$, but has  access to some  data $X_{1:T}:= \left(X_{t} \right)_{t=1}^T $ that are assumed to be generated by the econometric model  $(\Omegabf, \mathcal{\mathcal{E}_\Omegabf}, \P)$. Thus, the econometrician  approximates $\theta_0 $ using the  data $X_{1:T}$, i.e., defines a proxy
\begin{eqnarray}
\theta^*_T:= H_T(X_{1:T})   \label{eq2}
\end{eqnarray}
where $H_T $ is a mapping from the observation space $\underline{\Sbf}_{T} $ to the parameter space $\T$. Often, $H_T(X_{1:T})=G_{}(\P_{X_{1:T}})$, where     $\P_{X_{1:T}}:=\frac{1}{T}\sum_{t=1}^T \delta_{ X_t} $ is the empirical measure with  $\delta_{X_t} $ denoting the Dirac measure at $X_t $. We call $\theta^*_T$ a finite-sample proxy of the unknown parameter $\theta_0$. 

Now, if there exist some data $X_{1:T}^\bullet $ with the same unconditional distribution as $X_{1:T} $ (i.e., $\P \circ X_{1:T}^{-1}=\P \circ{X_{1:T}^{\bullet}}^{-1}  $) but independent from  them, these  data $X_{1:T}^\bullet $  induce an equally informative  finite-sample proxy of $\theta_0$ 
 \begin{eqnarray}
 \theta^\bullet_T:=H_T(X_{1:T}^\bullet).\label{eq4} \end{eqnarray}
Building on this remark, the idea of  the neoclassical inference theory is to base inference of $\theta_0$ on   an approximation of the distribution of a generic finite-sample  proxy of $\theta_0$  that has the same unconditional distribution as $\theta^*_T $, but  is independent from the data $X_{1:T} $.
 We denote the generic finite-sample  proxy with  $\theta^{\bullet}_T $. In this paper, we take the finite-sample proxy  $\theta^*_T$ (i.e., the choice of $H_T(.)$) as given, in order to stay away from the question  of  the properties  of the generic proxy $\theta^{\bullet}_T $.\footnote{This question, which can be seen as one of the main topic of the statistical  and  econometric literature,  corresponds to the study of the properties of what is called the estimator in the Neyman-Pearson theory.}

\begin{ex}(continued from p. \pageref{ex}).    From equation (\ref{eq1}),     $\theta_0=\frac{1}{s\sqrt{2 \pi}}\int_\R x\exp[-\frac{1}{2}\frac{(x-\theta_0)^{2}}{s^2}]\d x\\ =\int_{\Omegabf}X_1(\omega)\P(\d\omega )=:G(\P)$.  From equation (\ref{eq2}),   $\theta^*_T:=H_{T}(X_{1:T}):= \frac{1}{T}\sum_{t=1}^T X_t=\overline{X}_T$.  From equation (\ref{eq4}), $\theta^\bullet_T:=\overline{X}_T^\bullet$. Assume that, as previously mentioned,  the  asymptotic Gaussian approximation is used to approximate   the distribution of  $\theta^\bullet_T $.  Then,     the distribution of the generic finite-sample proxy $\theta^\bullet_T $  is $\mathcal{N}\left(\overline{X}_T(\omega), \frac{s_T(X_{1:T}(\omega))}{\sqrt{T}}  \right) $ m.a.e.  
 
 \text{ }  \hfill $\diamond $
\end{ex}

\subsection{Definitions} \label{sec4_2}

We require the following assumptions, in addition to Assumption \ref{assp1}, to outline the neoclassical theory.

\begin{assp}\label{assp4} \textbf{\emph{(a)}} Let $\P$ be  the unknown probability measure on $(\Omegabf, \mathcal{E})$.  \textbf{\emph{(b)}} Let $(\T,  \mathcal{E}_{\T})$ be a measurable space s.t.    $\T $ is a Borel subset of $\R^p $ with $p\in \N \setminus\{0\} $,
 and $ \mathcal{E}_{\T}$ denotes the Borel $\sigma$-algebra on $\T$. \textbf{\emph{(c)}} Let $U$ be a uniformly distributed random variable with support  $[0,1]$ on the probability  space, $(\Omegabf, \mathcal{E}_\Omegabf, \P)$, s.t.  $X_{1:T}$  and  $U$  are independent for all $T\in \[1, \infty \[ $. 
\end{assp}

Assumption \ref{assp4}(c) is the only assumption that is new with respect to Neyman-Pearson theory. Nevertheless, its novelty is limited as  econometric reasoning (e.g., asymptotic theory) and implementations of the Neyman-Pearson theory (e.g., bootstrap) often implicitly require it. The random variable $U$ is a randomization device that ensures the existence of the  generic finite-sample proxy $\theta^\bullet_T $. 
   More generally, Assumption \ref{assp4}(c)  ensure the existence of  a countable number of random variables with any probability distribution (e.g., Kallenberg, 1997/2002, Lemmas 3.21 and 3.22).  Such an assumption is innocuous. We can always redefine  the probability space $(\Omegabf, \mathcal{E}_\Omegabf, \P)$  as the product of  an original probability space with the probability space $([0,1], \mathcal{B}([0,1]), \lambda) $, where $\mathcal{B}([0,1]) $ and $\lambda$, respectively, denote the Borel $\sigma$-algebra on $[0,1]$ and the Lebesgue measure (e.g., Kallenberg, 1997/2002, pp. 111--112).

  Thanks to Assumption \ref{assp4}(c), given a (data-based) finite-sample proxy of $\theta_0$, Lemma \ref{lem1} proves the existence of a corresponding generic  proxy of $\theta_0$.  
 
\begin{lem}[Existence of a generic  proxy] \label{lem1}Let the finite-sample proxy of $\theta_0$ be  $\theta^*_T:=H_T(X_{1:T})$, where $H_T $ is a measurable mapping $ $ from $(\underline{\Sbf}_T, \underline{\mathcal{S}}_T)$ to  $(\T,  \mathcal{E}_{\T})$.   Under Assumptions \ref{assp1} and \ref{assp4},  there exists a corresponding generic  finite-sample proxy of $\theta_0 $, i.e., a  measurable mapping $\theta^{\bullet}_T$ from $(\Omegabf, \mathcal{E}_\Omegabf)$ to  $(\T,  \mathcal{E}_{\T})$    that is independent from $X_{1:T} $, but has the same unconditional distribution as $\theta^*_T$.
\end{lem}
\begin{proof}
This is an application of a  known generalization of the standard inverse transform method that is used to simulate random variables from a uniform distribution (e.g.,  Kallenberg, 1997/2002, p. 56,  Lemma 3.22). By Assumption \ref{assp4}(b),  $\T$ is a Borel space, i.e., there exists a Borel measurable bijection $h: \T \rightarrow A$, with $A\in \mathcal{B}([0,1])$, s.t. $h^{-1}$ is also measurable. Denote  the c.d.f. of $h(\theta^*_T)$ with $F_{h(\theta^*_T)}$  and put $ F_{h(\theta^*_T)}^{-1}(u):=\inf_{z}\left\{ z \in A: F_{h(\theta^*_T)}(z)\geqslant u   \right\} $, $\forall u \in [0,1] $. Then, under Assumption \ref{assp4}(a), for all $B \in  \mathcal{E}_{\T} $, $\P (h^{-1}(F_{h(\theta^*_T)}^{-1}(U)) \in B) \stackrel{(a)}{=}\P( F_{h(\theta^*_T)}^{-1}(U) \in h(B) )\stackrel{(b)}{=}\P( h(\theta^*_T) \in h(B))\stackrel{(c)}{=}\P( \theta^*_T \in B) $, where \textit{(a)} and \textit{(c)} are a consequence of the bimeasurability and bijectivity of $h(.)$, and \textit{(b)} is an application of the standard inverse  transform method by Assumption \ref{assp4}(c). Now, again by Assumption \ref{assp4}(c), $U$ is independent from data. Thus, put $\theta^{\bullet}_T:=h^{-1}(F_{h(\theta^*_T)}^{-1}(U)) $.       
\end{proof}

In this paper, the definitions of neoclassical estimators and confidence regions and test
are based on the  generic proxy. To simplify their statements, we require the following assumption.

\begin{assp}\label{assp7} Denote the Borel $\sigma$-algebra on $\R$ with $\mathcal{B}(\R) $. There is a   $ \mathcal{E}_{\T}/\mathcal{E}_{\R}$-measurable p.d.f. $f_{\theta^\bullet_T}(.)$ s.t., for all $B \in  \mathcal{E}_{\T} $, $\P(\theta^\bullet_T \in B)=\int_B f_{\theta^\bullet_T}(\theta) \mu(\d \theta)$ m.a.e.,   where $\mu$ is the Lebesgue measure $\lambda$, or the counting measure $\nu$. 
\end{assp}

%
%

By the Radon-Nikodyn theorem, Assumption \ref{assp7} requires the distribution of the generic finite-sample proxy $\P \circ {\theta^\bullet_T}^{-1} $, or its approximation to be a probability measure dominated  by the Lebesgue or the counting measure. In practice, because $\P \circ {\theta^\bullet_T}^{-1} $ is unknown, it requires the approximation of $\P \circ {\theta^\bullet_T}^{-1} $ to be a probability measure dominated by  the Lebesgue or the counting measure.  Instead of requiring the existence of $f_T(.)$, we could apply the Lebesgue decomposition theorem to write m.a.e. the measure $\P \circ {\theta^\bullet_T}^{-1} $ as the sum of a continuous, a discreet and a singular measure. However, it would  complicate the upcoming definitions without much tangible gain. In particular, under Assumption \ref{assp7}, the neoclassical estimator is simply a maximizer of the p.d.f.  $f_{\theta^\bullet_T}$  m.a.e. 

\begin{defn}[Neoclassical estimator]\label{defn4} A neoclassical estimator, denoted $\check{\theta}_T$, is a maximizer of the  p.d.f.  $f_{\theta^\bullet_T}$ m.a.e.,   i.e., 
 \begin{eqnarray*}
 \check{\theta}_T \in\arg \max_{\theta \in \mathbf{\Theta}} f_{\theta^\bullet_T}(\theta) \quad  m.a.e.
 \end{eqnarray*}
 
 \end{defn}
\begin{ex}(continued)$\P\negthickspace\left(\theta^\bullet_T \negthickspace\in\negthickspace B_{r}(\theta)\right)\negthickspace =\negthickspace \frac{1}{[s_T(X_{1:T}(\omega))/\sqrt{T}]\sqrt{2 \pi}}\negthickspace  \int_{B_{r}(\theta)}\exp\negthickspace\left[ -\frac{1}{2}\negthickspace\left(\frac{\dot \theta-\overline{X}_T(\omega)}{s_T(X_{1:T}(\omega))/\sqrt{T}}\right)^2\negthickspace\right] \negthickspace\d \dot \theta$ 
m.a.e.  Thus,  the neoclassical estimate is the mode of  $\mathcal{N}\left(\overline{X}_T(\omega),  \frac{s_T(X_{1:T}(\omega))}{\sqrt{T}}  \right) $, i.e.,  $\hat{\theta}_T =\overline{X}_T(\omega)$ m.a.e.  \hfill $\diamond $
\end{ex} 

  By Definition \ref{defn4},  a neoclassical estimator is an element of  the parameter space $\T$  that has  the highest  probability density to be the generic finite-sample proxy $\theta^\bullet_T $ m.a.e. Thus, it is a maximum-probability based estimator. In the neoclassical theory,   confidence regions are also maximum-probability based.

\begin{defn}[Neoclassical confidence region]\label{defn5}  Denote the support of $f_{\theta^\bullet_T}$ with $\supp( f_{\theta^\bullet_T}) $, i.e., $\supp (f_{\theta^\bullet_T}):=\{ \theta \in \T:f_{\theta^\bullet_T}(\theta)>0\} $. A $\mathcal{B}(\mathbf{\Theta})$-measurable set, $R_{1-\alpha, T} $, is  a neoclassical confidence  region of level $1-\alpha $ with $\alpha\in[0,1]$ if, and only if,   
\begin{eqnarray*}
R_{1-\alpha, T}=\left\{ \theta \in \supp( f_{\theta^\bullet_T}) :f_{\theta^\bullet_T}(\theta) \geqslant k_{\alpha,T}\right\} \quad m.a.e.,
\end{eqnarray*} 
where $k_{\alpha,T}:= \sup_{ k\in \R}\left\{ k: \int_{\{\theta \in \T:f_{\theta^\bullet_T}(\theta)\geqslant k \} } f_{\theta^\bullet_T}(\theta) \mu(\d \theta)\geqslant 1-\alpha \right\} $.

\end{defn}

%

\begin{ex} (continued) Because a Gaussian distribution is unimodal and symmetric with respect to its mean, $R_{1-\alpha,T}=\left[\overline{X}_T(\omega)-\frac{s_T(X_{1:T}(\omega))}{\sqrt{T}}u_{1-\frac{\alpha}{2}},\overline{X}_T(\omega)-\frac{s_T(X_{1:T}(\omega))}{\sqrt{T}}u_{\frac{\alpha}{2}} \right]$ m.a.e., where $u_{\alpha/2} $ denotes the $\alpha/2$ quantile of a standard Gaussian distribution. \hfill $\diamond $
\end{ex}



\begin{rk}
 The existence of neoclassical  confidence region is typically not a concern. Appendix \ref{ap3} on p. \pageref{ap3} proves the existence of neoclassical confidence regions  under mild assumptions. \hfill $\diamond $
\end{rk}

   By construction, a neoclassical confidence region is an indicator of the confidence we can have in a neoclassical estimate. It is the set of parameter values that are the closest to being the neoclassical estimate, such that the whole set has a probability at least $1- \alpha $ to contain the generic finite-sample proxy $\theta^\bullet_T $ m.a.e.
Thus, a small connected neoclassical confidence region indicates a well-separated estimate, which is reliable. In contrast,  a large neoclassical confidence region or a neoclassical confidence region that consists of the union of disjoint sets indicates an unreliable  estimate.

\begin{rk} \label{rk5} If the purpose of confidence regions is to indicate the confidence we can have in an estimate, their neoclassical definition is more satisfactory than their Neyman-Pearson definition (inadequacies caused by   multiple
use of the  data and their past realization set aside). The Neyman-Pearson definition of  confidence regions is not about the estimate, but about  coverage. In particular,  Neyman-Pearson confidence regions do not necessarily contain the estimate.   
 \hfill $\diamond $
\end{rk}

\begin{rk} In Definition \ref{defn5}, the definition of neoclassical estimator  is formally close to the  definition of Bayesian highest  posterior density (HPD) sets (e.g., Berger, 1980/2006, sec. 4.3.2., Definition 5), although their theoretical justification and meaning are fundamentally different.  \hfill $\diamond $
\end{rk}

Although Definition \ref{defn5} corresponds to a joint  confidence region, marginal and conditional neoclassical confidence regions can also be defined by considering the marginal and conditional distribution of $\theta^{\bullet}_T $. From neoclassical confidence regions, we define neoclassical tests.

\begin{defn}[Neoclassical test] \label{defn6}  Let $\mathrm{H}: \theta_0=\dot \theta $ be a test hypothesis, and $R_{1-\alpha,T}$ a $1-\alpha$ neoclassical confidence region, where $\alpha \in [0,1]$. As in Definition \ref{defn9}, denote the  decision space  with  $\D := \left\{d_\mathrm{H},d_\mathrm{A}\right\} $.  A  neoclassical  test of level $\alpha$ for $\mathrm{H} $  is a decision rule, denoted $d_T$, s.t.  if 
\begin{eqnarray*}
 \dot \theta  \in R_{1-\alpha,T} \quad m.a.e.
    \end{eqnarray*}
 then  $d_{T}= d_\mathrm{H}$; otherwise $d_{T}= d_\mathrm{A} $.
\end{defn}


Definition \ref{defn6} leads  the econometrician to reject  hypotheses  that do not correspond to the set of parameter values with the highest probability density of being  equal to the generic proxy $\theta^{\bullet}_T$ m.a.e. By Definitions \ref{defn5}, all elements in  a neoclassical confidence region have a  higher probability density  of being equal to the generic finite-sample proxy  than the ones outside it m.a.e.

\begin{ex}(continued)  If $ \dot \theta \in \left[\overline{X}_T(\omega)-\frac{s_T(X_{1:T}(\omega))}{\sqrt{T}}u_{1-\frac{\alpha}{2}},\overline{X}_T(\omega)-\frac{s_T(X_{1:T}(\omega))}{\sqrt{T}}u_{\frac{\alpha}{2}} \right]$, then  we do not reject the test hypothesis, i.e., $d_{T}= d_{\mathrm{H}}$. Note that, in this example, the neoclassical estimate, confidence region and test are practically equivalent to their usual Neyman-Pearson counterparts, although their theoretical justification is different. Nevertheless, there are Neyman-Pearson confidence regions and tests that do not practically correspond to  neoclassical confidence regions or tests. E.g. under the assumption that data $X_{1:T}$ have not  been realized prior to the decision  to compute the confidence interval $\left]-\infty,\overline{X}_T(\omega)-\frac{s_T(X_{1:T}(\omega))}{\sqrt{T}}u_{.5+\frac{\alpha}{2}}\right] \cup \left[\overline{X}_T(\omega)-\frac{s_T(X_{1:T}(\omega))}{\sqrt{T}}u_{.5-\frac{\alpha}{2}},\infty\right[ $, the latter is a valid $1-\alpha$ Neyman-Pearson confidence region, while it is not a neoclassical  confidence region.
\qquad      \text{ }
       \hfill      $\diamond $
\end{ex}

\begin{rk} While the neoclassical definition of confidence regions appears more satisfactory than their Neyman-Pearson definition (see Remark \ref{rk5}, p. \pageref{rk5}), the reverse seems to be true for tests  (inadequacies caused by    multiple
use of the  data and their past realization set aside). Unlike  Neyman-Pearson  tests,   neoclassical tests  do not directly control the probability of making an error, so that their outcome should be understood  in terms of evidence \textit{in favor of}, or \textit{against} the hypothesis. However, it should be noted that Neyman-Pearson tests control the probability of type I error only ex ante: after computation of the test statistic, the probability of error is $0$ or  $1$. Moreover, work in progress by the authors suggest that the direct control of type I error can be regained within the neoclassical theory.   \hfill $\diamond$

\end{rk}

  \begin{rk} When $\mu = \lambda$, the precise choice of $f_{\theta^\bullet_T}$ is typically crucial for Definitions \ref{defn4}-\ref{defn6} :  a modifications of the p.d.f. $f_{\theta^\bullet_T} $ on a $\lambda$-null set yields another p.d.f. of $\P \circ {\theta^\bullet_T}^{-1} $ w.r.t. $\lambda$ m.a.e. that can lead to a completely  different estimate, confidence region and result of a test (see subsection \ref{sec5_2}). This peculiarity, from which we take advantage in subsection \ref{sec5_2} (see Remark \ref{rk4}, p. \pageref{rk4}),  also arises in Neyman-Pearson   and Bayesian theories (e.g., Gouri\'eroux and Monfort, 1989/1996, sec. 7.A.2). Nevertheless, under the mild assumption that  $\theta \in \T $ is a Lebesgue point, by Lebesgue's differentiation theorem (e.g., Folland, 1984/1999,  Theorem 3.21), $f_{\theta^\bullet_T}(\theta) $ is often  s.t. $f_{\theta^\bullet_T}(\theta)= \lim_{r \downarrow0 }\frac{\P (\theta^\bullet_T \in B_r(\theta))}{\lambda(B_r(\theta))}$ m.a.e., where $B_r(\theta) $ denotes a ball in $\T$ centered at $\theta $ with radius $r>0 $.  \hfill $\diamond $

\end{rk}

\begin{rk} As Definitions \ref{defn4}, \ref{defn5} and \ref{defn6} respectively indicate, neoclassical  estimators,  confidence regions and tests are not random m.a.e., and do not depend on the realized data m.a.e. In the examples, their dependence  on  the realized data is only due to the approximation error. \qquad      \text{ }      \hfill $\diamond $

\end{rk}

\subsection{Neoclassical theory and multiple use of the same data} \label{sec4_3}

The upcoming Theorem  \ref{thm3} investigates the adequacy of the neoclassical theory when data have already been used, and thus are known.  Because the neoclassical theory is based on the distribution of the generic proxy $\theta^{\bullet}_T$, it is sufficient to investigate the effect of   previous knowledge of the realized data on this distribution.    

\begin{thm}[Neoclassical adequacy] \label{thm3}
 Under Assumptions \ref{assp1} and \ref{assp4}, for all $B \in  \mathcal{E}_{\T} $, 
\begin{enumerate}
\item[i)] for all  $A_{T} \in \underline{\mathcal{S}}_T $, $\left\{ X_{1:T} \in A_T\right\}$ and $\left\{ \theta^\bullet_T \in B\right\}$ are independent m.a.e., i.e.,  
\begin{eqnarray*}
\P\left(\left\{ \theta^\bullet_T \in B\right\} \cap \left\{X_{1:T} \in A_T\right\}\right) = \P(\theta^\bullet_T \in B) \P(X_{1:T} \in A_T) \ m.a.e.;
\end{eqnarray*}
\item[ii)] for all $ A_T\in \underline{\mathcal{S}}_T $ s.t.  $\P  ( X_{1:T} \in A_T)>0$,  
\begin{eqnarray*}
\P(\theta^\bullet_T \in B|X_{1:T} \in A_T) = \P(\theta^\bullet_T \in B)\ m.a.e.
\end{eqnarray*}
\end{enumerate}

\end{thm}
\begin{proof}It is a consequence of Lemma  \ref{lem1}. \textit{i)} By Lemma \ref{lem1}, $ \theta^\bullet_T$ and $X_{1:T} $  are independent, so that all events in $\sigma(\theta^\bullet_T )$ and $\sigma(X_{1:T})$ are independent (e.g., Kallenberg, 1997/2002, p. 50). \textit{ii) }Using (i), replace in the proof of Theorem \ref{thm1}, the nonequal sign by an equal sign, and set $E=\left\{\theta^\bullet_T \in B \right\}$. 
 \end{proof}
Theorem \ref{thm3} shows that the  distribution of the generic proxy $\theta^{\bullet}_T$ is immune to previous knowledge (or realization) of the data. Then, the inadequacy of neoclassical confidence regions and tests  follow.

\begin{cor}[Neoclassical confidence region and test adequacy] \label{cor1} 
Let  $R_{1-\alpha,T} $ be a neoclassical  $1-\alpha $ confidence region for $\theta_0$. Under Assumptions \ref{assp1} and \ref{assp4}, for all $ A_T\in \underline{\mathcal{S}}_T $ s.t.  $\P  ( X_{1:T} \in A_T) >0$, 
\begin{eqnarray*}
 \displaystyle\P( \theta^\bullet_T \in R_{1-\alpha,T}|X_{1:T} \in A_T) = \P(\theta^\bullet_T \in R_{1-\alpha,T}) \text{ m.a.e.} 
\end{eqnarray*} 

\end{cor}

\begin{proof} Apply Theorem \ref{thm3}(ii) putting  $B=R_{1-\alpha,T} $. \end{proof}

\begin{ex}(continued)  For clarity, we  now explicitly distinguish between the fixed  $\omega \in \Omegabf$ due to the approximation error and the random elements of the sample space. We denote the latter ones with $\tilde{\omega} $.  By Theorem \ref{thm3} (i), m.a.e.,
\begin{eqnarray*} 
 & &\ \P \left\{\left.\tilde{\omega}\in \Omegabf:\theta^\bullet_T(\tilde{\omega}) \in B\right|\tilde{\omega}\in \Omegabf:X_{1:T}(\tilde{\omega}) \in A_{T}\right\}\\
& = & \P(\theta^\bullet_T \in B_{r}(\theta)|X_{1:T} \in A_{T})=\P(\theta^\bullet_T \in B) = \P \left\{\tilde{\omega}\in \Omegabf:\theta^\bullet_T(\tilde{\omega}) \in  B\right\}\\
& = &  \frac{1}{[s_T(X_{1:T}(\omega))/\sqrt{T}]\sqrt{2 \pi}}\int_{ B}\exp\left[ -\frac{1}{2}\left(\frac{\dot \theta-\overline{X}_T(\omega)}{s_T(X_{1:T}(\omega))/\sqrt{T}}\right)^2\right] \lambda (\d \dot \theta)
\end{eqnarray*}
 By Corollary \ref{cor1}, m.a.e., 
\begin{eqnarray*}
& &\negthickspace \negthickspace \negthickspace  \negthickspace \negthickspace \P \left\{\left.\negthickspace\tilde{\omega}\negthickspace\in \Omegabf\negthickspace:\negthickspace\theta^\bullet_T(\tilde{\omega})\negthickspace \in\negthickspace \left[\overline{X}_T(\omega) \negthickspace- \negthickspace\frac{s_T(X_{1:T}(\omega))}{\sqrt{T}}u_{1-\frac{\alpha}{2}},\overline{\negthickspace X}_T(\omega)\negthickspace-\negthickspace\frac{s_T(X_{1:T}(\omega))}{\sqrt{T}}u_{\frac{\alpha}{2}} \right]\right|\tilde{\omega\negthickspace}\in \negthickspace\Omegabf\negthickspace:\negthickspace X_{1:T}(\tilde{\omega}) \negthickspace\in \negthickspace A_{T} \negthickspace\right\} \\
& = & \P( \theta^\bullet_T \in  R_{1-\alpha,T}|X_{1:T} \in  A_{T})=  \P(\theta^\bullet_T \in  R_{1-\alpha,T})\\
& = &\P \left\{\tilde{\omega}\in \Omegabf:\theta^\bullet_T(\tilde{\omega}) \in \left[\overline{X}_T(\omega)-\frac{s_T(X_{1:T}(\omega))}{\sqrt{T}}u_{1-\frac{\alpha}{2}},\overline{ X}_T(\omega) -\frac{s_T(X_{1:T}(\omega))}{\sqrt{T}}u_{\frac{\alpha}{2}} \right]\right\}
%
\end{eqnarray*} 
The distinction between the fixed $\omega $ and the varying $\tilde{\omega} $ is essential. The fixed $\omega $ can be ignored as long as the approximation error  is negligible, i.e., the approximation is justified.
 \hfill $\diamond $ \end{ex}

\begin{rk}

Corollary \ref{cor1}(ii) does not mean or imply that, for all  $A_{T} \in \underline{\mathcal{S}}_T $,   $ \lim_{T \rightarrow \infty}\P(\theta^\bullet_T \in R_{1 -\alpha,T}| X_{1:T} \in A_{T}) = \lim_{T \rightarrow \infty}\P(\theta^\bullet_T \in R_{1 -\alpha,T})$. First,  in the neoclassical theory,  approximation errors  do not necessarily come from asymptotic approximation (see Remark \ref{rk1}). Second, even in the case, in which the whole approximation error would come from an asymptotic approximation, Corollary \ref{cor1}(ii)  only implies independence between  $X_{1:\infty}$ and any neoclassical confidence region $R_{1-\alpha} $ deduced from   $ \lim_{T \rightarrow \infty} \hP  $, where $\hP$ denotes an approximation of the distribution of the generic proxy $\theta^\bullet_T$. E.g., in the Example, $\hP \stackrel{d}{\sim}\mathcal{N}\left(\overline{X}_T(\omega), \frac{s_T(X_{1:T})}{\sqrt{T}}\right)$, so that   $\lim_{T \rightarrow \infty}\hP= \delta_{\theta_0} $ $\P$-a.s., which, in turn, implies,   $R_{1-\alpha, \infty}=\{ \theta_0\} $ $\P$-a.s. Therefore,  $\{\tilde{\omega}\in \Omegabf:\theta_0 \in R_{1-\alpha, \infty}\}=\{\tilde{\omega}\in \Omegabf:\theta_0 \in \{ \theta_0\}\} $  has probability one, and  is independent from the data  $X_{1:\infty}$. Note that, in the Neyman-Pearson theory, we would need to consider  $\lim_{T \rightarrow \infty}\P\left\{ \theta_0 \in \left[\overline{X}_T-\frac{s_T(X_{1:T})}{\sqrt{T}}u_{1-\frac{\alpha}{2}},\overline{X}_T-\frac{s_T(X_{1:T})}{\sqrt{T}}u_{\frac{\alpha}{2}} \right] \right\}=1-\alpha$, where $ \left[\overline{X}_T-\frac{s_T(X_{1:T})}{\sqrt{T}}u_{1-\frac{\alpha}{2}},\overline{X}_T-\frac{s_T(X_{1:T})}{\sqrt{T}}u_{\frac{\alpha}{2}} \right] $  depends on the data, and is random even asymptotically. This difference between the two theories should help to understand why, in the Example,  the same finite-sample confidence interval depends on the data m.a.e. for the Neyman-Pearson theory, while it is does not depend on the data m.a.e. for   the neoclassical theory.   
    \hfill $\diamond $
\end{rk}


\begin{rk} In this paper, the generic proxy $\theta^\bullet_T$ is introduced for expository purpose, i.e., to allow the use of probability symbolism. From a strict logical point of view, the immunity of the unconditional distribution of $\theta^*_T $ to multiple use of the same data is all that is needed for the neoclassical adequacy: by definition the unconditional distribution of $\theta^*_T $ is about all the possible values of $\theta^*_T $  induced by all the possible samples that could have been observed.   In other words, the key difference between the  Neyman-Pearson and Bayesian theories  on the one hand,  and  the neoclassical theory on the other hand is that, in the latter, inference exclusively relies on a unconditional distribution m.a.e., while, in  the other, inference relies  on the  realized data, even m.a.e.    
\hfill $\diamond $ 
\end{rk}

The probabilistic statements, on which  neoclassical estimators, confidence regions and tests are based, are immune to previous information about the data, and thus to multiple use of the same data m.a.e. To our knowledge, the neoclassical theory  is the  first general inference theory immune to multiple use of the same data m.a.e.

%
%
%
%
%

\section{A neoclassical point of view on some calibration and econometric practices.} \label{sec5}

This section aims at presenting  some prominent  practices from the point of view of the neoclassical inference theory. The elementary version of the  theory outlined in the subsection \ref{sec4_2} is    sufficient for this purpose. By-products of  the current section are examples of implementation of the neoclassical theory, novel theoretical justifications for the presented calibration and econometric practices, and a standard-error adjustment to account for approximation errors.  


Subsection \ref{sec5_1} discusses requirements for proxies and approximations of  their distribution. Subsection \ref{sec5_2} presents choices of proxies and of  approximations  that correspond to  different econometric and calibration practices.  
 Subsection \ref{sec5_3} assesses the most common econometric practice  through Monte-Carlo simulations, and presents the standard-error adjustment. Because, in this section, we discuss the choice of approximations, we  distinguish between the distribution of the generic proxy  $\P \circ {\theta^\bullet_T}^{-1} $, and its chosen approximation, which we denote $\hP(.)=\int_. \hat{f}_{\theta^\bullet_T}(\theta)\mu(\d \theta) $.

\subsection{On generic proxies and approximations of their distribution}\label{sec5_1}
An implementation of the neoclassical theory requires two inputs: a generic proxy and  an approximation of its distribution.  These inputs do not have to satisfy any particular criteria other than being considered a proxy of $\theta_0 $,  and an approximation of $\P \circ {\theta^\bullet_{T}}^{-1} $, respectively. In particular, the neoclassical theory does not require consistency of any of the two: consistency is about situations where the number of observations can be
infinitely increased, while practice  is necessarily based on
a bounded number of them. Nevertheless,   hereafter, except in the subsection \ref{sec5_2_1} about calibration, we focus on asymptotically normal proxies and consistent approximations, so that we can rely on insights from the asymptotic theory: the proxy typically  corresponds to  what is called an estimator in the  Bayesian or Neyman-Pearson theory. The following Assumption \ref{assp9} requires asymptotic normality of $\theta^\bullet_T $, which is a property of most estimators considered in the Neyman-Pearson and  Bayesian theories (e.g., Chernozhukov and Hong, 2003). 

 \begin{assp}[Asymptotic normality of $\theta^\bullet_T $] \label{assp9} The   generic proxy of $\theta_0 $ is asymptotically normal, i.e., (by  Assumption \ref{assp4}(c))  there exist a random variable $\xi^\bullet $, and a sequence of random variables $ (R^\bullet_{T})_{T =1}^\infty$ on $(\Omegabf,\mathcal{E}_{\Omegabf}) $  s.t.   
 \begin{eqnarray*}
\theta^\bullet_T=\theta_0+\frac{\xi^\bullet}{T^{1/2}}+R^\bullet_{T} \label{}
\end{eqnarray*}
where  $\P \circ {\xi^\bullet}^{-1} \stackrel{d}{\sim} \mathcal{N}(0, \Sigma^{\frac{1}{2}})$, and $R^\bullet_{T} = o_{\P}(T^{-1/2})$,  as $T \rightarrow  \infty $.
\end{assp}
    Assumption \ref{assp9} means that the generic proxy $\theta^\bullet_T $ asymptotically converges to $\theta_0$ as a Gaussian random variable centered at $\theta_0$ with a standard deviation that goes to zero at rate $\sqrt{T}$.  We could weaken Assumption \ref{assp9} to allow rates of convergence different from $\sqrt{T}$, or to allow different distributions for  ${\xi^\bullet} $ (e.g., Dickey-Fuller distributions), but it would complicate the presentation.
The following Assumption \ref{assp11} requires the approximation of the distribution of  the generic proxy to be consistent.

\begin{assp}[Consistency of $\hP $] \label{assp11} The approximation of the distribution of the generic proxy is consistent, i.e.,   as $T \rightarrow \infty $, 
\begin{eqnarray*}
\rho\left(\hP, \P\circ {\theta^\bullet_T}^{-1}\right) \stackrel{\P}{\rightarrow} 0, \qquad 
\end{eqnarray*}
where  $\rho(.,.) $ denotes a metric on  the space of probability measures on $(\T,  \mathcal{E}_{\T})$. 
 
\end{assp}
 Assumption \ref{assp11} means that the distribution of the generic proxy and its approximation converge to each other as the number of observations increases.  In Appendix \ref{ap7}, we verify Assumption \ref{assp11} for the approximations considered in this paper. In practice, the distribution of the generic proxy, $\P\circ {\theta^\bullet_T}^{-1} $, is typically unknown, so that Assumption \ref{assp11} cannot be directly verified. Nevertheless, the asymptotic limit of $\P \circ {\theta^\bullet_T}^{-1} $ is often known, so that the following lemma provides  a usable criterion for checking Assumption \ref{assp11}.

\begin{lem} \label{lem4}  Under Assumptions \ref{assp1} and \ref{assp4}, if
there exists a probability measure $\P\circ {\theta^\bullet_\infty}^{-1}  $ on $(\T,  \mathcal{E}_{\T})$ s.t., as $T \rightarrow \infty $,
\begin{enumerate}
\item[(a)] $\rho \left(\hP,  \P\circ {\theta^\bullet_\infty}^{-1} \right) \stackrel{\P}{\rightarrow} 0  $   and 
\item[(b)] $\rho( \P\circ {\theta^\bullet_\infty}^{-1} , \P\circ {\theta^\bullet_T}^{-1})  \stackrel{\P}{\rightarrow} 0 $,
\end{enumerate}
then $\hP $ is a consistent approximation of $\tP $, i.e., Assumption \ref{assp11} holds.  \end{lem}

\begin{proof} Triangle inequality yields  $ \rho\left(\hP, \P\circ {\theta^\bullet_T}^{-1}\right)\leqslant   \rho(\hP, \P\circ {\theta^\bullet_\infty}^{-1} )+ \rho(\P\circ {\theta^\bullet_\infty}^{-1}, \P\circ {\theta^\bullet_T}^{-1})$, where the two terms of the RHS go to zero in probability as $T \rightarrow \infty$ by (a) and (b), respectively.
\end{proof}

\begin{ex} (continued) Let $\rho $ be the Prokhorov metric on the space of probability measures on $(\T,  \mathcal{E}_{\T}) $.  The Prokhorov metric generates the topology of the convergence in law (e.g., Billingsley, 1968/1999, pp. 72--73), which, in turn, corresponds to the point-wise convergence of cumulative distribution functions (c.d.f.) at continuity points of the limiting c.d.f. (Portmanteau theorem). Denote the c.d.f. of the Gaussian distribution $\mathcal{N}(\tau, s)$ with $\mathfrak{N}(.; \tau; s)$. Then,  for all $\theta \in \T \setminus \{\theta_0 \} $, $\lim_{T \rightarrow \infty}\mathfrak{N}(\theta; \theta_0; \frac{s}{\sqrt{T}})=\ind_{[\theta_0, \infty[}(\theta)$, because $\mathfrak{N}(\theta; \theta_0; \frac{s}{\sqrt{T}})=\mathfrak{N}(\sqrt{T}\frac{\theta -\theta_0 }{s};0; 1) $, and $\lim_{T \rightarrow \infty}\sqrt{T}\frac{\theta -\theta_0}{s}=-\infty $, if $\theta<\theta_0 $, and $ \infty$ otherwise. Similarly,  for all $\theta \in \T \setminus \{\theta_0 \} $,  $\lim_{T \rightarrow \infty}\mathfrak{N}(\theta; \overline{X}_T; \frac{s_T(X_{1:T}(\omega))}{\sqrt{T}})=\ind_{[\theta_0, \infty[}(\theta) $, and $\lim_{T \rightarrow \infty}\ind_{[\overline{X}_T, \infty[}(\theta)=\ind_{[\theta_0, \infty[}(\theta) $ $\P$-a.s. (see also Appendix \ref{ap7_2_1}, Proposition \ref{prop5} on p. \pageref{prop5})   Thus, by Lemma \ref{lem4},  $\delta_{\overline{X}_T(\omega)} $ and $\mathcal{N}\left(\overline{X}_T(\omega), \frac{s_T(X_{1:T}(\omega))}{\sqrt{T}}  \right) $ are   consistent approximations of the distribution of the generic proxy $\overline{X}_T^{\bullet}$, $\mathcal{N}(\theta_0, \frac{s}{\sqrt{T}}) $.
\hfill $\diamond $

\end{ex}

 \begin{rk}
 
 Unlike for Neyman-Pearson and Bayesian theories,  implementations of the elementary version of the neoclassical inference theory presented in this paper seem to always rely on an approximation, the
approximation of the distribution of the generic proxy. This is a disadvantage   of the elementary version of the neoclassical theory.  However, firstly, most of econometric practices  rely on approximations: implementation of the  Neyman-Pearson and Bayesian theories  typically requires asymptotic (e.g. CLT) or computational (e.g., Markov Chain Monte Carlo  algorithms) approximations.  Secondly,  in some particular cases, we can bound the approximation error in probability, or even   derive the exact distribution of $\hP $, and of $\rho(\hP, \tP) $ : see Appendix \ref{ap8} on p. \pageref{ap8}. Thirdly, in practice, there is a trade-off between approximations and the approximative assumptions that  are needed to avoid the (explicit) approximations.
\hfill $\diamond $ \end{rk}

\subsection{Some practices in neoclassical terms} \label{sec5_2} In this subsection, we frame some calibration and econometric practices in terms of the neoclassical theory, so that they are provided with a theoretical foundation  immune to multiple use of the same data m.a.e. For convenience and brevity,  we present the practices according to the kind of approximations of $\tP $ they rely on. Thus, practices that combine different approximations (e.g., mix of calibration and econometrics in  Canova, 2007, chap. 7) are only indirectly treated. Appendix \ref{ap7}  studies the asymptotic limit of
the approximations considered in this subsection, while Table \ref{tab1} in Appendix \ref{ap9} on p. \pageref{tab1}  provides a panoramic view of them. 


\subsubsection{Model-calibration approximations}\label{sec5_2_1}
 By calibration, we mean the  more or less formal process   through which  the parameter values   of a model are selected in view of data.   In finance and economics, this process  often correspond to  the minimization of some goodness-of-fit measure, or to the choice of estimates from various existing empirical studies.      While model calibration is used in many fields (e.g., Oreskes, Shrader-Frechette and  Belitz, 1994), it has become common in economics and finance with the development of general-equilibrium models (Johansen, 1960; Kydland and Prescott, 1982; Shoven and Walley, 1984) and  derivatives pricing, respectively. We distinguish two types of calibration: plain calibration and criterion-adjusted calibration.

\emph{Plain calibration.}   In  plain calibration, the selected parameter values are just plugged in the model in lieu of the unknown parameter $ \theta_0$.  No  information about potential calibration error or model-specification error is reported.    Plain calibration is widely used to price derivatives in finance (e.g., Cont, 2010, p. 1217). If the selected parameter values are assumed to be a realization of  a random variable $\theta^*_{T,C} $,  under Assumptions \ref{assp1} and \ref{assp4}, plain calibration can be regarded as an implementation of the neoclassical theory s.t. 
\begin{itemize}
\item  the generic proxy $\theta^\bullet_{T,C}$ is a random variable that has the same unconditional distribution as $\theta^*_{T,C} $, i.e., $\theta^\bullet_{T,C}=F^{-1}_{\theta^*_{T,C}}(U_{C}^\bullet)$ where  $F^{-1}_{\theta^*_{T,C}} $ denotes the inverse of the unconditional c.d.f. of $\theta^*_{T,C} $, and $U_{C}^\bullet $ a random variable uniformly distributed on $[0,1] $ that is independent from the data; 
\item the approximation of the unconditional distribution of the proxy is the unit point mass at the selected parameter value $\theta^*_{T,C}(\omega) $, i.e., for all $\theta \in \T $, 
\begin{eqnarray*}
\hat{f}_{\theta^\bullet_T}(\theta)= \ind_{\{ \theta^*_{T,C}(\omega)\}}(\theta) \text{ with }\mu=\nu.
\end{eqnarray*}
\end{itemize}
Then, by Definitions \ref{defn4} and \ref{defn5}, both the neoclassical estimate and confidence region correspond to the calibrated value, i.e., $\check{\theta}_T=\theta^*_{T,C}(\omega) $ and  $R_{1-\alpha,T}= \{\theta^*_{T,C}(\omega) \} $ m.a.e. 
If the selected parameter value is  close to $\theta_0 $  (e.g., the proxy $\theta^*_{T,C} $ is consistent and  the number of observations $T $ is large as in Proposition \ref{prop3}(i) in Appendix \ref{ap7_1}), plain calibration may be sufficient. However,  in economics and finance,  this is not often the case, so that indication of the potential calibration error and model-specification error are often needed.

\emph{Criterion-adjusted calibration.} By criterion-adjusted calibration, we mean  a plain calibration  accompanied  by indications of calibration error and model-specification error based on  nonstatistical criteria.  The indication of calibration error comes from the determination of a range of plausible values for the model parameter. The indication of model-specification error comes from the computation of measures of discrepancy  between the calibrated model and the data (e.g., difference between moments of the calibrated model and moments of the data).

To cast criterion-adjusted calibration as an implementation of the neoclassical theory, it is useful to introduce new notation.  We denote the model  parameter and the measures of discrepancy with $\beta $ and  $\Delta $, respectively. We also denote the selected value for $\beta $ with   $\beta^*_{T,C} $, and  the  measure of discrepancy  implied by $\beta^*_{T,C} $ with $\Delta_{T,C}^* $. The parameter $\beta$ of the model of interest are not to be confused with the global parameter $\theta:=(\beta' \quad \Delta')' $ .   The  determination of the range of plausible values for $\beta$ and  acceptable values for $\Delta$ can be formalized by a positive criterion function $u: (\theta, \dot \theta)\mapsto u(\theta, \dot \theta) $, which indicates the  adequacy between its two arguments, and which equal zero outside $\T^2 $. The criterion function $u $ is maximized when its two arguments are equal, i.e., for all $\dot\theta \in \mathbf{\Theta}$ and $\theta \in \mathbf{\Theta} \setminus \{\dot \theta\} $,  $u(\theta, \dot \theta) \leqslant u(\dot\theta,\dot\theta) $. With this notation, and under the assumåption that the following quantities exist,  criterion-adjusted calibration  can be regarded as an implementation of the neoclassical theory s.t.  
\begin{itemize}
\item the generic proxy is  $\theta^{\bullet}_{T,CC}:=F^{-1}_{\theta^\bullet_{T,CC}|\theta^{\bullet}_{T,C}}(U^{\bullet}_{CC}) $, where $F_{\theta^\bullet_{T,CC}|\theta^{\bullet}_{T,C}} $ is a conditional  c.d.f. s.t. $F_{\theta^\bullet_{T,CC}|\theta^{\bullet}_{T,C}}(.):=\int_{-\infty}^. \frac{u(\theta, \theta^\bullet_{T,C})}{\int_\T u( \dot \theta, \theta^\bullet_{T,C})\lambda( \d \dot \theta)}\lambda(\d \theta)$ with  $ $ $]-\infty, \theta]:=]-\infty, \theta_1]\times ]-\infty, \theta_2]\times \cdots \times ]-\infty, \theta_p] $ for $( \theta_1 \ \theta_2 \ \theta_3 \ \cdots \ \theta_p)':=\theta\in \T $, and where $U^{\bullet}_{CC} $ is a random variable uniformly distributed on $[0,1] $ that is independent from the data and $\theta^{\bullet}_{T,C} $;

   \item  the approximation of the distribution of the  generic proxy $\theta^{\bullet}_{T,CC} $ is the normalized criterion function, i.e., for all $\theta \in \T $,   
\begin{eqnarray*}
\hat{f}_{\theta^\bullet_T}(\theta) \propto u(\theta, \theta^*_{T,C}(\omega)) \text{ with }\mu=\lambda.
\end{eqnarray*}
\end{itemize}

\subsubsection{Gaussian approximation} The Gaussian approximation is one of the most-widely used approximations. Under the assumption that the unconditional distribution of  the $t$-statistic  $\sqrt{T}{\hat{\Sigma}_T}^{-\frac{1}{2}}(\theta^*_{T,G}-\theta_0)$ converges asymptotically to a standard Gaussian  $\mathcal{N}(0,I) $, econometricians typically deduce univariate confidence regions and sets of nonrejected point hypotheses, $\left[\theta^*_{T, G, k}-\frac{\hat s_{T, G,k,k}}{\sqrt{T}}u_{1-\alpha/2},\right.$ $\left.\theta^*_{T, G, k}-\frac{\hat s_{T, G,k,k}}{\sqrt{T}} u_{\alpha/2}\right]$, where $\theta^*_{T, G, k}$, ${\hat s_{T, G,k,k}}^2$ and $u_{\alpha/2}$, respectively,  denote the $k$-th  element of the random vector $\theta^*_{T,G} $, the $k$-th diagonal element of the matrix $\hat\Sigma_T $, and the  $\alpha/2$ quantile of a standard univariate Gaussian $\mathcal{N}(0,1)$. As in the Example of this paper,  such practice can be  regarded as an implementation of the neoclassical theory s.t.
\begin{itemize}
\item  the generic proxy $\theta^\bullet_{T,G}$ is a random variable that has the same unconditional distribution as $\theta^*_{T,G} $, i.e., $\theta^\bullet_{T,G}=F^{-1}_{\theta^*_T}(U_{G}^\bullet)$ where  $F^{-1}_{\theta^*_{T,G}} $ denotes the inverse of the unconditional c.d.f. of $\theta^*_{T,G} $, and $U_{G}^\bullet $ a random variable uniformly distributed on $[0,1] $ that is independent from the data; 
\item the approximation of the unconditional distribution of the proxy is the Gaussian distribution centered at $\theta^*_{T,G}(\omega) $ with variance-covariance matrix the diagonal matrix of the diagonal elements of $\hat{\Sigma}_T $, i.e., for all $\theta \in \T $, 
\begin{eqnarray}
\hat{f}_{\theta^\bullet_T}(\theta) \propto \exp\left[-\frac{T}{2}(\theta-\theta^*_{T,G}(\omega))'\diag(\hat{\Sigma}_T(\omega))^{-1}(\theta-\theta^*_{T,G}(\omega)) \right] \label{eq6} 
\end{eqnarray}
with $\mu=\lambda$, and where $\diag(\hat{\Sigma}_T(\omega))$ is the diagonal matrix that has the same diagonal elements as $ \hat{\Sigma}_T(\omega)$.
 \end{itemize}  
In the Appendix \ref{ap7}, we show that   the approximation (\ref{eq6}) is  consistent under Assumptions \ref{assp1}, \ref{assp4} and  \ref{assp9}. (see Proposition \ref{prop5} in Appendix \ref{ap7_2}) 


\subsubsection{Laplace  approximation and ``Bayesian" practice }\label{secLap} In most econometric practices, the unknown parameter $\theta_0 $ is approximated by a  maximizer $\theta^*_{T,L} $ of a  converging objective function $ Q_T(X_{1:T}, \theta)$, i.e., 
\begin{eqnarray}
\theta^*_{T,L} \in\arg \max_{\theta \in \T} Q_T(X_{1:T}, \theta) \label{eq7}
\end{eqnarray}
where, as $T \rightarrow \infty $, $\sup_{\theta \in \T} \Vert Q_T(X_{1:T}, \theta)-Q(\theta)\Vert=o_\P(1) $ $ $ with $\theta_0 = \arg \max_{\theta \in \T} Q(\theta)$.   See Bierens (1981), Amemiya (1985, chap. 4), Gallant and White (1988), Newey and McFadden (1994),   and P\"otscher and Prucha (1997), all of which follow from earlier contributions by  Wald (1950), Malinvaud (1964/1970, chap. 9; 1970), and Jennrich (1969).   
If $TQ_T(X_{1:T}, \theta) $ is a log-likelihood $L_T(X_{1:T}, \theta) $, for all $x_{1:T} \in\underline{\Sbf}_{T} $ and $\theta \in \T $,  the function  \begin{eqnarray}
(x_{1:T}, \theta)\mapsto \e^{L_T(x_{1:T},\theta)}
\label{eq5}  \end{eqnarray} 
is numerically equal to the Bayesian distribution of the data $X_{1:T} $ conditional on $\theta_0 $, denoted $\pi_{X_{1:T}|\theta_0}(x_{1:T}|\theta) $  (see subsection \ref{sec3_1}).
Thus, by analogy,
often motivated by the Laplace approximation, several papers (e.g., Zellner, 1997; Kim, 2002;  Yin, 2009) have used the function  
\begin{eqnarray}
(x_{1:T}, \theta)\mapsto\e^{TQ_T(x_{1:T},\theta)}
\label{eq8}  \end{eqnarray}
in lieu of $\pi_{X_{1:T}|\theta_0}(.|.) $, even when the former is not numerically equal to the latter. Then, they consider that the Bayesian posterior distribution $\pi_{\theta_{0}|X_{1:T}}\left(\theta|X_{1:T}(\omega)\right)$ equals
\begin{eqnarray}\frac{\e^{TQ_T(X_{1:T}(\omega),\theta)}w(\theta)}{\int_\T\e^{TQ_T(X_{1:T}(\omega),\dot\theta)}w(\dot \theta)\lambda(\d \dot \theta)}   \label{eq9} \end{eqnarray}
where $w: \T \mapsto \R_+$ is a function that they regard as the Bayesian prior $\pi_{\theta_0}(.) $.
As previously noticed (e.g., Chernozhukov and Hong, 2003), such a practice is not in line with Bayesian theory, even if we set aside the invalidation due to multiple use of the same data: approximations are not compatible with Bayesian inference theory, which  requires  an econometrician 
to know exactly the distribution  of the data conditional on true parameter (Savage, 1954/1972, pp. 59--60).  However,  such a practice is in line with the neoclassical inference  theory whether the function (\ref{eq8}) is numerically equal to $\pi_{X_{1:T}|\theta_0}(.|.) $ or not. 

Under general assumptions, a  strand of literature that goes back at least to Laplace (1774/1878) has shown that (\ref{eq9}) is a consistent  approximation of the unconditional distribution of $\theta^*_{T,WL}$.
See literature on  consistency of Bayesian posteriors (e.g., Doob, 1949), and the Bernstein-von Mises theorem  (e.g., Le Cam, 1953, 1958;   Chen, 1985; Kim, 1998; Chernozhukov and Hong, 2003), which implies consistency under $\P$ (see Appendix \ref{ap7_2_3} on p. \pageref{ap7_2_3}).
We distinguish three types of  Laplace approximations: plain Laplace approximation, weighted Laplace approximation, and criterion-adjusted weighted Laplace approximation.
For brevity, we do not treat the plain Laplace approximation separately from the weighted Laplace approximation, as the former is a particular case of the latter with $w(\theta)=1$, for all $\theta \in \T $.  

\textit{Plain and weighted Laplace approximation.} The neoclassical theory justifies  practices 
that treats  (\ref{eq9}) as a Bayesian posterior, and report the counterpart of a mode  and of an HPD region.  Such practices can be seen  as an implementation of the neoclassical theory s.t.
\begin{itemize}
\item  the generic proxy $\theta^\bullet_{T,WL}$ is a random variable that has the same unconditional distribution as $\theta^*_{T,WL}:= \arg \max_{\theta \in \T}\e^{TQ_T(X_{1:T}(\omega),\theta)}w(\theta) $, i.e., $\theta^\bullet_{T,WL}=F^{-1}_{\theta^*_{T,WL}}(U_{WL}^\bullet)$ where  $F^{-1}_{\theta^*_{T, WL}} $ denotes the inverse of the unconditional c.d.f. of $\theta^*_{T,WL} $, and $U_{WL}^\bullet $ a random variable uniformly distributed on $[0,1] $ that is independent from the data; 
\item the approximation of the unconditional distribution of the proxy is the expression (\ref{eq9}) viewed as a function of $\theta$, i.e., for all $\theta \in \T$,
\begin{eqnarray*}
\hat{f}_{\theta^\bullet_T}(\theta)\propto \e^{TQ_T(X_{1:T}(\omega),\theta)} w(\theta) \text{ with }\mu =\lambda. \label{eq10}
\end{eqnarray*}
 \end{itemize}

From a neoclassical point of view, $w(.)$ weights the  evidence from data. Mathematically, it  corresponds to a change of measure from  the plain Laplace approximation, $\widehat{\P\circ {\theta^\bullet_{T, L}}^{-1}} $, to the weighted Laplace approximation, $\widehat{ \P\circ { \theta^\bullet_{T, WL} }^{-1}} $, i.e., for all $\theta \in \T $,   
\begin{eqnarray*}
w(\theta)=\frac{\d (\widehat{\P\circ {\theta^\bullet_{T, WL}}^{-1}})}{\d (\widehat{\P\circ {\theta^\bullet_{T, L}}^{-1}})}(\theta).
\end{eqnarray*}
The weighting function $w(.)$ allows to incorporate  additional information in the proxy of $\theta_0$. While, in Bayesian theory, the dependance of  a prior on data is typically problematic, in the neoclassical theory, the weighting function $w(.) $ can depend on the data. The neoclassical theory only requires $\theta^\bullet_{T,WL} $ and the integral of (\ref{eq9}) to be considered a proxy of $\theta_0 $,  and an approximation of $\P \circ {\theta^\bullet_{T,WL}}^{-1} $, respectively. Thus, in particular, the neoclassical theory provides a theoretical foundation to the practice called parametric empirical Bayes  (Morris, 1983). Petrone, Rousseau, Scricciolo (2014)
present conditions under which the weighted Laplace approximation (\ref{eq9}) is a consistent approximation of $\P \circ {\theta^\bullet_{T,WL}}^{-1} $ when $w(.)$ depends on data through an estimated hyperparameter.

\textit{Criterion-adjusted weighted Laplace approximation.} When  (\ref{eq9}) is treated as if it was a Bayesian posterior,  the counterpart of a mode  and of an HPD region are not always reported. Instead, the econometrician  chooses a utility function (i.e., opposite of a loss function), $u: (\theta_{e},\theta)\mapsto u(\theta_{\e},\theta)$,   and then maximizes w.r.t. $\theta_e $ the expected utility $\int_{\mathbf{\Theta}} u(\theta_{e},\theta) \frac{\e^{TQ_T(X_{1:T}(\omega),\theta)} w(\theta)}{\int \e^{TQ_T(X_{1:T}( \omega),\dot\theta)} w(\dot \theta)\lambda(\d \dot\theta) }\lambda(\d\theta)$. If $u(.,.)$ is nonnegative (see upcoming Remark \ref{rk3}), such a practice can be seen  as an implementation of the neoclassical theory s.t.
\begin{itemize}
\item the generic proxy is $\theta^{\bullet}_{T,CWL}:=F^{-1}_{\theta^\bullet_{T,CWL}|\theta^{\bullet}_{T,WL}}(U^{\bullet}_{CWL}) $ where $F_{\theta^\bullet_{T,CWL}|\theta^{\bullet}_{T,WL}}(.):=\int_{-\infty}^. \frac{u(\theta, \theta^\bullet_{T,WL})}{\int_\T u(\dot \theta, \theta^\bullet_{T,WL})\lambda(\d \dot \theta)}\lambda(\d \theta)$, and where $U^{\bullet}_{CWL} $ is a random variable uniformly distributed on $[0,1] $ that is independent from the data and $\theta^{\bullet}_{T,WL} $;

   \item  the approximation of the distribution of the  generic proxy $\theta^{\bullet}_{T,CWL} $ is the expected criterion function $u $, i.e., for all $\theta \in \T $,   
\begin{eqnarray}
\hat{f}_{\theta^\bullet_T}(\theta) \propto \int_\T u(\theta,\dot \theta )\e^{TQ_T(X_{1:T}(\omega),\dot \theta)} w(\dot \theta)\lambda (\d \dot \theta) \text{ with }\mu=\lambda.
\end{eqnarray}
\end{itemize}   
In Appendix \ref{ap7},  we show that the criterion-adjusted weighted Laplace approximation (\ref{eq10}) is consistent under general assumptions.
The Bayesian approach and the  neoclassical approach based on the criterion-adjusted weighted Laplace approximation (\ref{eq10}) are fundamentally different, although they are numerically equivalent when  $TQ_T(X_{1:T}, \theta) $ is a log-likelihood (and the data have not been previously used).  For  Bayesian inference theory, an econometrician faces a \textit{known} probabilized uncertainty of a \textit{random}  parameter $\theta_0 $, while from a neoclassical perspective an econometrician faces an\textit{ estimated} probabilized uncertainty of a \textit{random} proxy of a \textit{constant}  parameter $\theta_0$. Thus, the neoclassical theory acknowledges the existence of an ``unmeasurable" uncertainty described by Knight (1921, chap. 7--8).

\begin{rk} \label{rk3}
In this paper, we assume nonnegative  criterion functions to remain within the framework of the elementary version of the neoclassical theory presented in section \ref{sec4}. This requirement limits the type of neoclassical confidence regions considered in this paper. Nevertheless, under the mild assumption that $\inf_{(\dot\theta, \ddot \theta) \in \T^2  }u(\dot\theta,\ddot \theta )> -\infty $, this requirement is without loss of generality for neoclassical point estimation: we can define the criterion function  $ \tilde{u}(\theta_{e} , \theta)= u(\theta_{e},\theta)-\inf_{(\dot\theta, \ddot \theta) \in \T^2  }u(\dot\theta,\ddot \theta )$ which yields the same point estimate.  \hfill $\diamond $ 
\end{rk}

\begin{rk}\label{rk4} As explained in the introduction of this subsection,  our presentation does not present hybrid practices, so that we do not explicitly cover the diversity of the econometric practices labelled Bayesian. However,  they  appear to also be implementations of  the elementary version of the  neoclassical theory presented in the subsection \ref{sec4_2}. For example,  reporting the HPD region of a Bayesian posterior with  its mean can be seen as an implementation of the neoclassical theory s.t. the generic proxy is $\theta^\bullet_{T,WL}$, and the approximation of its distribution is
\begin{eqnarray*}
\hat{f}_{\theta^\bullet_T}(\theta) \propto \begin{cases} \e^{TQ_T(X_{1:T}(\omega),\theta)} w(\theta)  & \text{ if } \theta \in \T \setminus \{\bar{\theta}_T \}\\
 \infty  & \text{ if } \theta=\bar{\theta}_T  \\
\end{cases} 
\end{eqnarray*}
where $\bar{\theta}_T:=\int_{\T}\dot{\theta}\frac{\e^{TQ_T(X_{1:T}(\omega),\dot{\theta})} w(\dot{\theta})}{\int_\T \e^{TQ_T(X_{1:T}(\omega),\ddot{\theta})} w(\ddot{\theta})\lambda(\ddot{\theta})}\lambda(\d \dot{\theta}) $, and $\mu=\lambda$. Similarly, reporting marginal equal-tailed 68\% confidence interval of a Bayesian posterior with  its mean can be seen as an implementation of the neoclassical theory s.t. the generic proxy is $\theta^\bullet_{T,WL}$, and the approximation of its distribution is the product of  $p $ Gaussian p.d.f. centered at the midpoint of  the $k $-th interval with a standard deviation approximately equal to half of the length of the same interval, i.e., 
\begin{eqnarray*}
\hat{f}_{\theta^\bullet_T}(\theta) \propto \begin{cases} \prod_{k=1}^p \mathfrak{n}(\theta_k; \frac{a_k +b_k}{2}; \bar{s}_{k,T}) & \text{ if } \theta \in \T \setminus \{\bar{\theta}_T \}\\
 \infty  & \text{ if } \theta=\bar{\theta}_T  \\
\end{cases} 
\end{eqnarray*} 
with $\mu=\lambda $, and  where $[a_k,b_k] $ denotes the reported intervals,  $\bar{s}_{k,T}\approx\ \frac{|b_k -a_k|}{2}$, and $\mathfrak{n}(\theta; \tau; s) $ is a Gaussian p.d.f. with expectation $\tau            $ and standard deviation $s $.
 \hfill $\diamond $
\end{rk}


\subsection{Assessment of the Gaussian approximation and a standard-error adjustment } \label{sec5_3} 

From a neoclassical point of view, the issue raised by multiple use of the same data boils down to the question of   approximation errors. This subsection studies the average effect of approximations errors, and develops a standard-error adjustment to account for it. For brevity and relevance, we focus on the Gaussian approximation, which corresponds to a large part of econometric practice. Assessments of other approximations are left for future research. 

\begin{table}\caption{Monte-Carlo assessment of Gaussian approximations for $M=10 000 $. RMSE is the square root of the mean-square error.$^a$  The columns $L^2$ and $\sup$, respectively,  correspond to $\frac{1}{M}\sum_{m=1}^M\sqrt{\int_\T [\hat{F}_{\theta^\bullet_T}^{(m)}(\theta)-F_{\theta^\bullet_T}(\theta)]^2\lambda(\d \theta) }$ and $\frac{1}{M}\sum_{i=1}^M \sup_{\theta \in \T} |\hat{F}_{\theta^\bullet_T}^{(m)}(\theta)-F_{\theta^\bullet_T}(\theta)| $. $\tilde{R}_{1-\alpha, T}:= \left[\theta^*_{T}\negthickspace-\negthickspace\sqrt{\frac{2}{T}}\hat{s}_{T}u_{1-\frac{\alpha}{2}},\, \theta^*_{T  }\negthickspace-\negthickspace\sqrt{\frac{2}{T}}\hat{s}_{T} u_{\frac{\alpha}{2}}\right] $} \label{tab5}
\begin{small}
\begin{tabular}{llllllllllllll}
\hline
\hline
 &  & $\overline{X}_T $ & $\hat{s}_T/\negthickspace\sqrt{T}$ & \multicolumn{2}{c}{ $\rho(\hat{F}_{\theta^\bullet_T},{F}_{\theta^\bullet_T} )$} & \multicolumn{4}{c}{$\P (\theta^\bullet_T \negthickspace\in\negthickspace \hat{R}_{1-\alpha, T}(\omega)) $} & \multicolumn{4}{c}{ $\P (\theta^\bullet_T \negthickspace\in\negthickspace \tilde{R}_{1-\alpha, T}(\omega)) $  } \\ 
\cline{3-14}
$T$ & $\negthickspace\P \circ X_1^{-1} $ & RMSE & RMSE & $L^2 $ & $ \sup$ & $.68$ & $.90$ & $ .95$ & $.99 $ &  $.68$ & $.90$ & $ .95$ & $.99 $ \\ 
\hline
$20$ & $\mathcal{\negthickspace N}(0, .2) $ &  .045$^b$ & .007$^b$& .089 & .311 & .5 & .732 & .811 & .912 & .658 & .878 & .932 & .981  \\ 
$50$ &  &  .028 & .003 & .07 & .303 & .51 & .746 & .825 & .924 & .671 & .892 & .943 & .987  \\ 
$100$ &  & .02 & .001 & .058 & .299 & .515 & .752 & .831 & .929 & .677 & .897 & .948 & .989  \\ 
\hline
$20$ & $\mathcal{\negthickspace N}(0, .4) $ &  .09 & .015 & .126 & .311 & .5 & .732 & .811 & .912 & .658 & .878 & .932 & .981  \\ 
$50$ &  &  .057 & .006 & .099 & .303 & .51 & .746 & .825 & .924 & .671 & .892 & .943 & .987  \\ 
$100$ &  &  .04 & .003 & .082 & .299 & .515 & .752 & .831 & .929 & .677 & .897 & .948 & .989 \\ 
\hline
$20$ & $\negthickspace B( \frac{2-\sqrt{3}}{4})^c $ & .056 & .043 & .117 & .37 & .345 & .638 & .677 & .726 & .568 & .725 & .728 & .74  \\ 
$50$ &  & .035 & .027 & .083 & .405 & .48 & .74 & .794 & .885 & .632 & .87 & .89 & .937  \\ 
$100$ &  & .025 & .019 & .067 & .374 & .501 & .741 & .824 & .918 & .683 & .881 & .929 & .971\\ 
\hline
$20$ & $\negthickspace B( .5) $ &  .113 & .053 & .143 & .375 & .566 & .729 & .838 & .911 & .686 & .88 & .931 & .981  \\ 
$50$ &  &  .071 & .035 & .111 & .346 & .515 & .728 & .806 & .922 & .631$ ^d$ & .898 & .942 & .987  \\ 
$100$ &  & .05 & .025 & .093 & .332 & .474 & .768 & .82 & .922 & .682 & .895 & .943 & .988  \\ 
\hline
\end{tabular}
\begin{tiny}
\begin{flushleft} $^a\!$ We do no report the bias as we know that $\overline{X}_T$ and $\frac{T}{T-1}\hat{s}_T^2 $ are  unbiased (e.g., Gouri\'eroux and Monfort, 1989/1996, Example 6.4).  $^b\!$ RMSE$(\hat{s}_T/\negthickspace\sqrt{T})\negthickspace=\negthickspace\text{RMSE}(\hat{s}_T)/\sqrt{T} $, which elucidates why  RMSE$(\overline{X}_T)<\text{RMSE}(\hat{s}_T/\negthickspace\sqrt{T})$. 
$^c$ As in the Gaussian case, this Bernoulli parameter  is chosen to halves the standard deviation of the other Bernoulli, i.e.,  $\sqrt{\frac{2-\sqrt{3}}{4}(1-\frac{2-\sqrt{3}}{4})}=\frac{.5}{2} $. $^d$\! The  non-monotonic convergence to .68 is due to the discontinuities induced by the Bernoulli distribution. See for example Brown, Cai and DasGupta (2002) for a similar phenomenon.   \!\end{flushleft}. \end{tiny}
\end{small}
\end{table}

The basic algorithm of our Monte-Carlo simulations is the following.

\texttt{
For $m =1, 2, 3, \ldots, M $
\vspace{-10pt}
\begin{enumerate}
\item Draw  i.i.d. data $X_{1:T}^{(m)}(\omega):=(X_t^{(m)}(\omega))_{t=1}^T$
 \item Compute  
 \begin{itemize}
\item  $\theta_T^{*(m)}(\omega)=\overline{X}_{T}^{(m)}(\omega)$;
\item    $\frac{\hat{s}_T^{(m)}(\omega)}{\sqrt{T}}=\frac{1}{\sqrt{T}}\sqrt{\frac{1}{T}\sum_{t=1}^T(X_t^{(m)}(\omega)-\overline{X}_T^{(m)}(\omega))^2}; $ 
\item    $\hat{F}_{\theta^\bullet_T}^{(m)}(\theta)=  \mathfrak{N}\left( \theta; \overline{X}_{T}^{(m)}(\omega);  \frac{\hat{s}_T^{(m)}(\omega)}{\sqrt{T}}\right)  $  ;
\item   
$\P\negthickspace\left(\theta^\bullet_T \negthickspace\in\hat{\negthickspace R}_{1-\alpha, T}^{(m)}(\omega)\negthickspace  \right)$, where\!\ $\hat{R}_{1-\alpha, T}^{(m)}\negthickspace:=\negthickspace\left[\theta_T^{*(m)}\negthickspace-\negthickspace\frac{\hat{s}_T^{(m)}}{\sqrt{T}}u_{1-\frac{\alpha}{2}},\theta_T^{*(m)}\negthickspace-\negthickspace\frac{\hat{s}_T^{(m)}}{\sqrt{T}}u_{\frac{\alpha}{2}} \right] $\!.
\end{itemize}
\end{enumerate}
}
We  draw data either from a Gaussian distribution, or from a Bernoulli distribution. Both families of distributions are interesting for different reasons.  Data from a  Bernoulli distribution are known to be relatively challenging for   Gaussian approximations, especially when the Bernoulli parameter is close to 0 or 1 (e.g., Agresti and Coull, 1998; Brown , Cai and DasGupta, 2002 and references therein). Data from a Gaussian distribution  neutralizes the part of the approximation error coming  from the  distribution family: the average of Gaussian random variables is also a Gaussian random variable.


Table \ref{tab5}  shows that, in both cases, Gaussian approximations globally converge  in terms of the first two moments, of the  $L^2 $ norm, and of the  $\sup $ norm.  However,  the  probability of neoclassical confidence regions to contain the generic proxy appears   downward biased.
Proposition \ref{prop4} formalizes this downward bias, and proposes an asymptotic adjustment for it.

\begin{figure}\caption{Examples of approximation  errors for  $\P \circ X_1^{-1} \stackrel{d}{\sim}\mathcal{N}(0,.4) $ and $T=20 $. The solid black line is $\P\circ {\theta^\bullet_T}^{-1} $, while the others are realizations of its Gaussian approximation $\mathcal{N}(\overline{X}_T,\frac{\hat{s}_T^{(m)}}{\sqrt{T}} ) $. Vertical dashed lines correspond to the 95\% neoclassical confidence region from $\P\circ {\theta^\bullet_T}^{-1} $.}\label{fig1} 
\begin{tabular}{cc}
\includegraphics[scale=0.7]{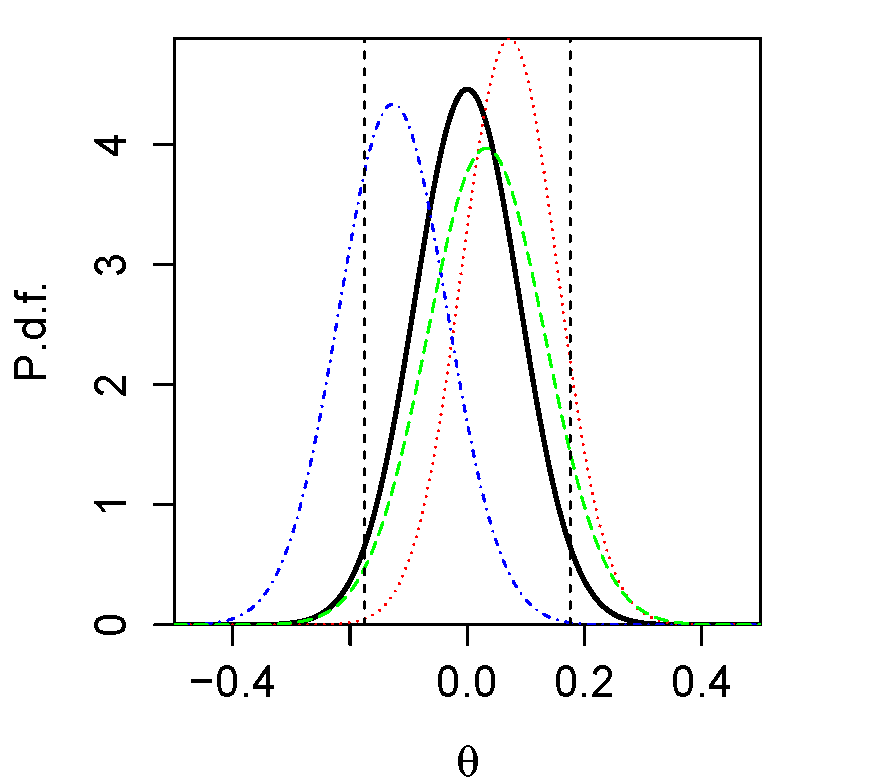}& \includegraphics[scale=0.7]{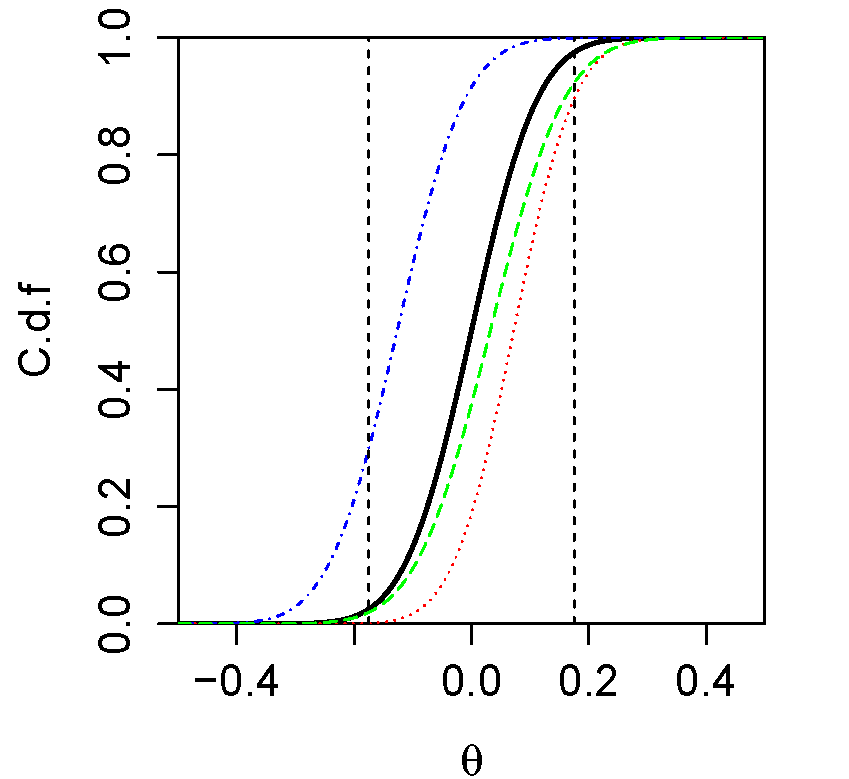} \\
\end{tabular}
 
\end{figure}

\begin{prop}[Downward bias and standard-error adjustment] \label{prop4}  Under Assumptions \ref{assp1}, \ref{assp4}, and \ref{assp9},  if $\hat{\Sigma}_T \stackrel{\P}{\rightarrow} \Sigma $ as $T \rightarrow \infty $, then, for all $\mathcal{E}_T\subset \sigma(X_{1:T}) $ and $k\in \[1,p \] $, 

\begin{enumerate}
\item[i)] $ \displaystyle \lim_{T \rightarrow \infty}\E \left\{ \P\left(\left.\theta^\bullet_{T,k} \in \left[\theta^*_{T,  k}-\frac{\hat{s}_{T, k,k}}{\sqrt{T}}u_{1-\frac{\alpha}{2}},\, \theta^*_{T,  k}-\frac{\hat{s}_{T,k,k}}{\sqrt{T}} u_{\frac{\alpha}{2}}\right]\right| \mathcal{E}_T\right)\right\}<1-\alpha$

\item[ii)] $ \displaystyle \lim_{T \rightarrow \infty} \E\left\{\P\negthickspace\left(\left.\negthickspace\theta^\bullet_{T,k} \negthickspace\in\negthickspace \left[\theta^*_{T,  k}\negthickspace-\negthickspace\sqrt{\frac{2}{T}}\hat{s}_{T, k,k}u_{1-\frac{\alpha}{2}},\, \theta^*_{T,  k}\negthickspace-\negthickspace\sqrt{\frac{2}{T}}\hat{s}_{T,k,k} u_{\frac{\alpha}{2}}\right]\right| \mathcal{E}_T\right)\right\}= 1-\alpha$
\end{enumerate}
 where  $\theta^*_{T, k}$, ${\hat{s}_{T, k,k}}^2$ and $u_{\alpha/2}$, respectively,  denote the $k$-th  element of the random vector $\theta^*_{T} $, the $k$-th diagonal element of the matrix $\hat{\Sigma}_T $, and the  $\alpha/2$ quantile of a standard univariate Gaussian $\mathcal{N}(0,1)$.
\end{prop}
\begin{proof}It is an immediate consequence of the asymptotic normality of $ \theta^\bullet_{T,k} $.  For all $\eta >0 $,
by iterated conditioning,\begin{eqnarray*}
& & \lim_{T \rightarrow \infty}\E \left\{ \P\left(\left.\theta^\bullet_{T,k} \in \left[\theta^*_{T,  k}-\frac{\eta \hat{s}_{T, k,k}}{\sqrt{T}}u_{1-\frac{\alpha}{2}},\, \theta^*_{T,  k}-\frac{\eta \hat{s}_{T,k,k}}{\sqrt{T}} u_{\frac{\alpha}{2}}\right]\right| \mathcal{E}_T\right)\right\}\\
& = &  \lim_{T \rightarrow \infty}\P\left(\theta^\bullet_{T,k} \in \left[\theta^*_{T,  k}-\frac{\eta \hat{s}_{T, k,k}}{\sqrt{T}}u_{1-\frac{\alpha}{2}},\, \theta^*_{T,  k}-\frac{\eta \hat{s}_{T,k,k}}{\sqrt{T}} u_{\frac{\alpha}{2}}\right]\right)\\
& \stackrel{(a)}{=} & \lim_{T \rightarrow \infty} \P \left( u_{\frac{\alpha}{2}} \leqslant\frac{1}{\eta} \sqrt{T} \frac{(\theta^*_{T,  k}-\theta^\bullet_{T,k})}{\hat{s}_{T, k,k}} \leqslant  u_{1-\frac{\alpha}{2}} \right) \\
& \stackrel{(b)}{=} &\lim_{T \rightarrow \infty} \P\left( u_{\frac{\alpha}{2}} \leqslant \frac{1}{\eta} \sqrt{T} \frac{(\theta^*_{T,  k}-\theta_0)}{\hat{s}_{T, k,k}} +\frac{1}{\eta} \sqrt{T} \frac{(\theta_0-\theta^\bullet_{T,  k})}{\hat{s}_{T, k,k}}   \leqslant u_{1-\frac{\alpha}{2}} \right)\\
& \stackrel{(c)}{=} & \P\left( u_{\frac{\alpha}{2}} \leqslant \mathcal{N}(0,  \frac{\sqrt{2 }}{\eta}) \leqslant u_{1-\frac{\alpha}{2}} \right)\\
& \stackrel{}{=} &\begin{cases}\P\left( u_{\frac{\alpha}{2}} \leqslant \mathcal{N}(0,  \sqrt{2 }) \leqslant u_{1-\frac{\alpha}{2}} \right)< 1-\alpha & \text{if }\eta=1 \text{, so that it yields i);}\\
\P\left( u_{\frac{\alpha}{2}} \leqslant \mathcal{N}(0,  1) \leqslant u_{1-\frac{\alpha}{2}} \right)= 1-\alpha & \text{if } \eta=\sqrt{2} \text{, so that it yields ii).}\\
\end{cases}
\end{eqnarray*}
\textit{(a)} On one hand,      $\theta^\bullet_{T,k} \leqslant \theta^*_{T,  k}-\frac{\eta \hat{s}_{T,k,k}}{\sqrt{T}} u_{\frac{\alpha}{2}}  \Leftrightarrow  u_{\frac{\alpha}{2}} \leqslant  \frac{1}{\eta} \sqrt{T} \frac{(\theta^*_{T,  k}-\theta^\bullet_{T,k})}{\hat{s}_{T, k,k}}  
$.   On the other hand, similarly, $ \theta^*_{T,  k}-\frac{\eta \hat{s}_{T, k,k}}{\sqrt{T}}u_{1-\frac{\alpha}{2}} \leqslant \theta^\bullet_{T,k} \Leftrightarrow  \frac{1}{\eta} \sqrt{T} \frac{(\theta^*_{T,  k}-\theta^\bullet_{T,  k})}{\hat{s}_{T, k,k}}\leqslant u_{1-\frac{\alpha}{2}}   $. 
 \textit{(b)} Add and subtract $\theta_0$.  \textit{(c)} Under Assumption \ref{assp9}, by  the continuous mapping theorem  (e.g., Kallenberg, 1997/2002, 
Lemma 4.3), as $T \rightarrow \infty $,
$ \frac{1}{\eta} \sqrt{T} \frac{(\theta^*_{T,  k}-\theta_0)}{\hat{s}_{T, k,k}}+ \frac{1}{\eta} \sqrt{T} \frac{(\theta_0-\theta^\bullet_{T,k})}{\hat{s}_{T, k,k}} \stackrel{\P}{ \rightarrow} \frac{\xi^*_k}{\eta s_{k,k} }+\frac{\xi^\bullet_k}{\eta s_{k,k} }$, where $\xi^*_k \stackrel{d}{\sim} \mathcal{N}(0, s_{k,k})$ is independent from $\xi^\bullet_k $.
\end{proof}

Proposition \ref{prop4}i) shows that the downward bias holds under  general assumptions, independently of  the sub-sigma algebra of the data we condition on. Figure \ref{fig1} (p. \pageref{fig1}) illustrates the reason of this downward bias\,: the Gaussian approximation does not account for the fact that its mean and standard deviation are not known, but estimated, so that there is an approximation error. Proposition \ref{prop4}ii) shows that multiplying the  standard error  by $\sqrt{2} $  asymptotically accounts for the average approximation error, independently of  the sub-sigma algebra of the data we condition on. The RHS columns of Table \ref{tab5} suggest that this asymptotic adjustment is effective in finite sample. The proof of Proposition \ref{prop4} formalizes the rationale behind the  adjustment: asymptotically, after centering and scaling by $\sqrt{T}$, the average  approximation error exactly corresponds to the uncertainty about  $\theta^\bullet_T$, so that the variance is doubled by independence, which means that the standard error is multiplied by $\sqrt{2} $.

 As can be seen from Figure \ref{fig2} (p. \pageref{fig2}), the adjustment has a nonlinear effect on  confidence region and test levels. The adjustment has a stronger effect in the tails because Gaussian distributions are exponentially decreasing in the tails.   Tables \ref{tab3}, \ref{tab4}  and \ref{tab6} are conversion tables  that documents the effect of the adjustment at conventional levels. They should be of special interest to applied econometricians. Table \ref{tab3} shows that tests at nominal levels .01, .05 and .1 are  tests at  approximate adjusted nominal levels .069, .166 and .245, respectively. Conversely, Table \ref{tab4} shows that tests at adjusted nominal level  .01, .05, and .1  respectively requires   the non-adjusted p-values  computed by standard software to be approximately below $.027(10^{-2}) $, $.056( 10^{-1}) $, and $.020 $ for rejection of the test hypothesis.  Table \ref{tab6} shows the effect of the adjustment on critical values   for the non-adjusted $t\text{-values}$  computed  by standard software.    Tables \ref{tab3}, \ref{tab4} and \ref{tab6} shed a new light on
results
published in nonexperimental fields. In particular, in view of the data collected by Brodeur, L\'e, Sangnier and Zylberberg (2013, Figure I), the adjustment appears to affect   the significance at conventional levels of many results  in the economic literature.

\begin{figure}\caption{Relation between nominal level and adjusted nominal level of a test.}\label{fig2}
\includegraphics[scale=0.7]{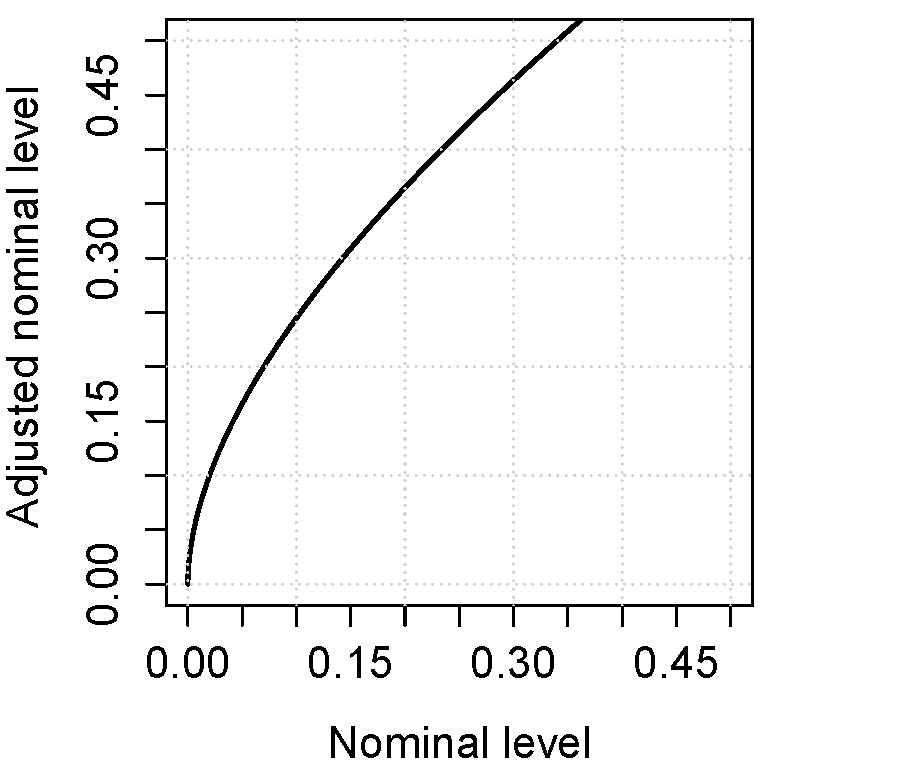}
\end{figure}
 
\begin{table}\caption{From nominal levels to adjusted nominal levels. } \label{tab3}
\begin{small}
\begin{tabular}{lllll} 
\hline
\hline
Nominal Level & \multicolumn{1}{c}{.01} & \multicolumn{1}{c}{.05} & \multicolumn{1}{c}{.1} & \multicolumn{1}{c}{.32} \\ 
\hline
Adj. nominal level (appr.) & \multicolumn{1}{c}{.069} & \multicolumn{1}{c}{.166} & \multicolumn{1}{c}{.245} & \multicolumn{1}{c}{.482} \\ 
\hline
\multicolumn{5}{l}{\begin{tiny}Adj. and appr. respectively stand for adjusted and approximation.\end{tiny}}\\
\end{tabular}
\end{small}
\end{table}

\begin{table} \caption{From adjusted nominal levels to nominal levels, and  adjusted critical values for non-adjusted t-statistics. } \label{tab4}
\begin{small}

 \begin{tabular}{lllll}
\hline
\hline
Adj. nominal level & \multicolumn{1}{c}{.01} & \multicolumn{1}{c}{.05} & \multicolumn{1}{c}{.1}& \multicolumn{1}{c}{.32} \\ 
\hline
Nominal Level (appr.) & \multicolumn{1}{c}{$.027( 10^{-2})$} & \multicolumn{1}{c}{$.056( 10^{-1})$} & \multicolumn{1}{c}{.020} & \multicolumn{1}{c}{.16}\\ 
\hline
\multicolumn{5}{l}{\begin{tiny}Adj. and appr. respectively stand for adjusted and approximation.\end{tiny}} \\
\end{tabular}
\end{small}

\end{table}

\begin{table} 

  \caption{From non-adjusted critical values to  adjusted critical values for non-adjusted t-statistics.  } \label{tab6}
\begin{small} 
 \begin{tabular}{lllll}
\hline
\hline
Non-adjusted critical values & \multicolumn{1}{c}{2.58} & \multicolumn{1}{c}{1.96} & \multicolumn{1}{c}{1.64}& \multicolumn{1}{c}{.99} \\ 
\hline
Adj. critical values (appr.) & \multicolumn{1}{c}{$3.64$} & \multicolumn{1}{c}{$2.77$} & \multicolumn{1}{c}{2.33} & \multicolumn{1}{c}{1.41}\\ 
\hline
\multicolumn{5}{l}{\begin{tiny}Adj. and appr. respectively stand for adjusted and approximation.\end{tiny}}\\
\end{tabular}

\end{small}

\end{table}

\section{Conclusion} 
In nonexperimental fields, it is inescapable to compute  test statistics and confidence regions  that are not probabilistically independent from previously examined data.   It has been  known for decades that  Neyman-Pearson and Bayesian inference theories are inadequate for such a practice. This paper recalls these inadequacies, and formally shows that they also hold  m.a.e. A novel inadequacy of the Neyman-Pearson theory for past-realized data is also established. Then, a general inference theory compatible with multiple use of the same data m.a.e. is outlined. We call it the neoclassical inference theory. 

The starting point of the neoclassical theory is   the acknowledgement that  econometric inference  relies on the use of a sample counterpart of the unknown parameter $\theta_0 $   as  a proxy for the latter one. Then, the idea  is to base inference on an approximation of the unconditional distribution of  the proxy. By definition, the unconditional distribution of the proxy is about all the possible values of the proxy  induced by all the possible samples that could have been observed. Thus, neoclassical inference does not depend on the realized data m.a.e. Therefore, if we set aside approximation errors, the neoclassical theory    explains why   econometric inference can rely on multiple use of the same data. 

The other side of the coin is that,  from a neoclassical point of view, the issue raised by multiple use of the same data boils down to the question of  the approximation errors, which is the topic of most of the econometric and statistical literature. Nevertheless,  Monte-Carlo simulations show that finding accurate approximations is not sufficient. Even when the approximation method is known to be  accurate,  errors can have a consequential effect on tests and confidence regions. In particular,   we prove that the Gaussian approximation yields a downward bias in the probability of neoclassical confidence regions to contain the generic proxy. Thus, we derive a general, but simple asymptotic standard-error adjustment to remove this bias.  Monte-Carlo simulations suggest that the adjustment is effective in finite sample. However, more work would be needed to study the impact of  approximation errors in other situations. The authors have work in progress in that direction.  

Beyond the question of multiple use of the same data, the  neoclassical inference theory is promising. The neoclassical inference theory sheds a new light on foundational and methodological debates in statistics, economics and finance  (e.g., calibration vs. estimation, and Bayesian inference vs. classical inference). The Example and section \ref{sec5} show that the neoclassical theory  provides a  unifying  framework for model calibration and several common econometric practices, whether they are labelled Bayesian or \`a la Neyman-Pearson.  Moreover, work in progress by the authors indicates that the  version of the neoclassical developed  in this paper is  generalizable.



\section*{Acknowledgements}
Helpful comments were provided by 
seminar participants at BI (finance and economics), Universit\'e Catholique de Louvain (CORE), University of Oslo (statistics), RCEF 2014 (Bayesian econometrics), SIPTA 2014, Tinbergen Institute (ector), CFE-ERCIM  2014, and Institut Henri Poincar\'e (semstats). Benjamin Holcblat acknowledges support from the Centre for Asset Pricing Research. 

\appendix
\section{Consistency adequacy} \label{ap1}

The following proposition shows that multiple use of the same data does not affect the consistency of a point estimator. In this Appendix \ref{ap1}, m.a.e. means that   we always consider the asymptotic limit to be exact for any sample size $T$. Nevertheless, for simplicity,  we also exclude the case \` a la Hoffmann-J{\o}rgensen (see Wellner and van der Vaart, 1996) in which finite-sample statistics do not need to be measurable.  

\begin{prop}[Consistency adequacy]\label{prop6} Let   $\hat{\theta}_T$ be an estimator of $\theta_0 $, i.e.,  a  measurable mapping from $(\Omegabf, \mathcal{E}_\Omegabf)$ to  $(\T, \mathcal{E}_\T)$, where $\mathcal{E}_\T $ denotes a $\sigma$-algebra on $\T$. Under Assumption \ref{assp1}, \begin{enumerate}
\item[i)] if $\hat{\theta}_T$ is strongly consistent, then, for all  $S \in \underline{\mathcal{S}}_T $, $\left\{ X_{1:T} \in S\right\}$ and $\{ \hat{\theta}_T  =\theta_0\}$ are independent m.a.e., i.e.,  
\begin{eqnarray*}
\P\left(\{ \hat{\theta}_T  =\theta_0\} \cap \left\{X_{1:T} \in S\right\}\right) = \P(\hat{\theta}_T  =\theta_0) \P(X_{1:T} \in S) \, m.a.e.;
\end{eqnarray*}

\item[ii)]  if $\hat{\theta}_T$ is weakly consistent, then, for all  $S \in \underline{\mathcal{S}}_T $, and for all neighborhood $N_{\theta_0} $ of $\theta_0 $, $\left\{ X_{1:T} \in S\right\}$ and $\{ \hat{\theta}_T \in N_{\theta_0} \}$ are independent m.a.e., i.e.,  
\begin{eqnarray*}
\P\left(\{ \hat{\theta}_T \in N_{\theta_0}   \}\cap \left\{X_{1:T} \in S\right\}\right) = \P( \hat{\theta}_T \in N_{\theta_0}   )  \P(X_{1:T} \in S)\, m.a.e.
\end{eqnarray*}

\end{enumerate}

\end{prop}
\begin{proof} It is definition chasing. By a standard  property of probability, for all $A,B \in \mathcal{E}_\Omegabf  $, $\P(A \cup B)= \P(A)+\P(B)-\P(A \cap B) $. Thus, if $\P(A)=1 $, adding $\P(A \cap B)-1$  on both sides yields
\begin{eqnarray}
\P(A \cap B)= \P(B)=\P(A)\P(B) \label{eq3}
\end{eqnarray}
because $1=\P(A)\leqslant \P(A \cup B)\leqslant 1 $.  

 i) By definition of strong consistency, $\P(\lim_{T \rightarrow \infty}\hat{\theta}_T  =\theta_0)=1$, which means that    $\P(\hat{\theta}_T  =\theta_0)=1$ m.a.e. Then, apply (\ref{eq3}) with $A=\{\hat{\theta}_T  =\theta_0\} $ and $B=\left\{X_{1:T} \in S\right\}$. 

ii) By definition of weak consistency, for all neighborhood $N_{\theta_0} $ of $\theta_0 $,  $\lim_{T \rightarrow \infty}\P_{*}( \hat{\theta}_T \in N_{\theta_0}   ) =1$,  which means that $\P( \hat{\theta}_T \in N_{\theta_0}   ) =1 $ m.a.e.
Then, apply (\ref{eq3}) with $A= \{ \hat{\theta}_T \in N_{\theta_0}   \} $ and $B=\left\{X_{1:T} \in S\right\}$. 
\end{proof}

\begin{rk}\label{rk6} Inspection of the proof shows that consistency is independent of any event, i.e., we can replace $\left\{ X_{1:T} \in S\right\}$  by any event $E \in \mathcal{E}_\Omegabf $ in the proof, and thus in the statement of Proposition  \ref{prop6}. \hfill $\diamond $
 \end{rk}

\section{Details of  the proof of Theorem \ref{thm1}} \label{ap5}
In this appendix,  we provide a detailed    proof of Theorem \ref{thm1}, i.e., we provide details regarding the qualification ``m.a.e."   We only consider the   case \`a la Hoffmann-J{\o}rgensen, as the standard asymptotic case follows easily from it.
\begin{proof}
i)   $\left\{ X_{1:T} \in A_T\right\}$ and $\left\{  \theta_0 \in C_{1-\alpha,T}(X_{1:T})\right\} $ are not independent m.a.e. if, and only if,
\begin{eqnarray*}
& & \liminf_{{T} \rightarrow \infty} \P_*(\{\theta_0 \in C_{1-\alpha, T}(\underline{X}_{{T}})\}\cap \{X_{1:T} \in A_T\}) \neq  \liminf_{{T} \rightarrow \infty} \P_*(\theta_0 \in C_{1-\alpha, T}(\underline{X}_{{T}})\P (X_{1:T} \in A_T) \\
& \stackrel{(a)}{\Leftrightarrow} & \liminf_{{T} \rightarrow \infty} \frac{\P_*(\{\theta_0 \in C_{1-\alpha, T}(\underline{X}_{{T}})\}\cap\{X_{1:T} \in A_T\})}{\P (X_{1:T} \in A_T)} \neq  \liminf_{{T} \rightarrow \infty} \P_*(\theta_0 \in C_{1-\alpha, T}(\underline{X}_{{T}})) \\
& \stackrel{(b)}{\Leftrightarrow} & \liminf_{{T} \rightarrow \infty} \P_*(\theta_0 \in C_{1-\alpha, T}(\underline{X}_{{T}})| X_{1:T} \in A_T) \neq  \liminf_{{T} \rightarrow \infty} \P_*(\theta_0 \in C_{1-\alpha, T}(\underline{X}_{{T}}))\\
& \Leftrightarrow & \P( \theta_0 \in C_{1-\alpha, T}(X_{1:T})|X_{1:T} \in A_T) \neq \P(\theta_0 \in C_{1-\alpha, T}(X_{1:T})) \text{ m.a.e.}
\end{eqnarray*}
\textit{(a)} By assumption, $ \P( X_{1:T} \in A_T)=c>0$. \textit{(b)} $\frac{\P_*(\{\theta_0 \in C_{1-\alpha, T}(\underline{X}_{{T}})\}\cap\{X_{1:T} \in A_T\})}{\P (X_{1:T} \in A_T)}=\frac{1}{\P (X_{1:T} \in A_T)}\sup_{}\{ \P(E):E\subset (\{\theta_0 \in C_{1-\alpha, T}(\underline{X}_{{T}})\}\cap \{X_{1:T} \in A_T\}) \wedge\ E \in \mathcal{E} \} = \sup_{}\{ \frac{\P(E)}{\P (X_{1:T} \in A_T)}:E\subset (\{\theta_0 \in C_{1-\alpha, T}(\underline{X}_{{T}})\}\cap \{X_{1:T} \in A_T\})\wedge\ E \in \mathcal{E} \}=\sup_{}\{ \frac{\P(E\cap\{X_{1:T} \in A_T\})}{\P (X_{1:T} \in A_T)}:E\subset (\{\theta_0 \in C_{1-\alpha, T}(\underline{X}_{{T}})\}\cap \{X_{1:T} \in A_T\}) \wedge\ E \in \mathcal{E} (X_{1:T} \in A_T) \} = \sup_{}\{ \P(E| X_{1:T} \in S):E\subset (\{\theta_0 \in C_{1-\alpha, T}(\underline{X}_{{T}})\}\cap \{X_{1:T} \in A_T\}) \wedge\ E \in \mathcal{E} (X_{1:T} \in A_T)\}=\P_*(\theta_0 \in C_{1-\alpha, T}(\underline{X}_{{T}})| X_{1:T} \in A_T)$, where for all $A \in \mathcal{E} $,  $\mathcal{E}(A):=\{B \cap A: B \in\mathcal{E}  \} $.  

ii) Replace in the proof of (i), $\{\theta_0 \in C_{1-\alpha, T}(\underline{X}_{{T}}) \} $, $\inf $, and $\P_* $ by $\left\{ d_{ T}(\underline{X}_{ T}) = d_{\mathrm{A}}\right\} $, $\sup $, and $\P^* $, respectively. 
\end{proof}

\begin{rk} In the above proof, the assumption $  \P( X_{1:T} \in A_T)=c>0$ can be weakened to $\lim_{T \rightarrow \infty}\P( X_{1:T} \in A_T)=c>0 $. However, the non-existence of the limit  or     the non-measurability of the conditioning event  $\{X_{1:T} \in A_T\}$ would make the proof more difficult. In particular, in the latter case, we would need a  conditional version of    nonadditive outer and inner measures,  and there does not seem to be a consensus on this subject (e.g., Young and Wang, 1998).  These difficulties do  not affect the main conclusion of section \ref{sec2} as they are the counterpart of the difficulties  to establish the Neyman-Pearson validity of  a confidence region or a test.  
\qquad \hfill      $\diamond $

\end{rk}

%
%
%
%
\section{Neyman-Pearson inadequacy for past-realized data} \label{ap4}

As pointed out in Remark \ref{rk2} on p. \pageref{rk2},  Theorem \ref{thm1} can be viewed as a formalization of the Neyman-Pearson \textit{inadequacy  for past-realized data} when only\textit{ part} of the data have been realized before the determination of the confidence regions and tests. This appendix formalizes this  inadequacy in the case in which \textit{all} data at use have been previously realized. For simplicity, we  rule out  the case \`a la Hoffmann-J{\o}rgensen  in which finite-sample statistics do not need to be measurable. We also require the following assumption for the determination of the confidence intervals. 

\begin{assp} \label{assp6} Let $\P$ be  the probability measure on $(\Omegabf, \mathcal{E})$ s.t. $\P \circ X_{1:T}^{-1} $ is the unconditional  physical and unknown distribution of $X_{1:T} $.   \textbf{\emph{(a)}}There exists a mapping $G $ from the space of all probability measures on $(\Omegabf, \mathcal{E})$ to the parameter space $\T $  s.t.  $G(\P):=\theta_0$.  \textbf{\emph{(b)}} There exists a family of probability measures $(\Pter_{\theta} )_{\theta \in \T}$ on $(\Omegabf, \mathcal{E}_\Omegabf) $  s.t., for all $\theta  \in \T $, $\theta=G(\Pter_{\theta}) $, and $ \P$ is dominated by $\Pter_{\theta} $, i.e.,  $\P \ll\Pter_{ \theta} $. 

\end{assp}

Assumption \ref{assp6} is often  satisfied. Assumption \ref{assp6}(a) requires the parameter  $\theta_0 $ to depend on the underlying probability measure that defines the distribution of the data $X_{1:T}$. Without this assumption, it seems difficult to see how  $\theta_0 $ can be inferred from the data $X_{1:T} $.  Assumption \ref{assp6}(b)
first requires  Assumption \ref{assp6}(a) to hold  independently of the location of $\theta_0 $ in $\T $.   It corresponds to the idea that the parameter space $\T $ is the set of  possible values for $\theta_0 $. Second, Assumption \ref{assp6}(b) requires all the measures in the family $(\Pter_{\theta} )_{\theta \in \T }$  to dominate the unknown probability measure $\P$. This can be restrictive. Nevertheless,  Assumption \ref{assp6}(b) is often satisfied, and it is  weaker than  some assumptions in the  literature on maximum-likelihood (e.g., Lehmann and Casella, 1983/1998, sec. 6.3; Gouri\'eroux and Monfort, 1989/1996, sec. 7.D.) or empirical processes (e.g., Khmaladze, 1981). These references require, for  all $ (\theta, \dot \theta) \in \T^2$, the existence of a unique  pair of probability measures    $ (\Pter_{\theta},\Pter_{\dot\theta} )$ s.t. $ \theta=G(\Pter_\theta) $, $\dot\theta=G(  \Pter_{\dot\theta}) $, and $\Pter_\theta $ is equivalent to $\Pter_{\dot \theta} $, i.e.,  $\Pter_\theta \sim\Pter_{\dot \theta}$. We require Assumption \ref{assp6}(b) to ensure that, for all $\theta \in \T $, $\Pter_\theta$-null sets are also $\P$-null sets.

 Under Assumption \ref{assp6}, the following Proposition \ref{prop1} formalizes the Neyman-Pearson inadequacy  that arises when all data at use have been previously  realized.

\begin{prop}[Neyman-Pearson inadequacy for past-realized data] \label{prop1}
\begin{itemize}
\item[i)] Let  $\alpha \in [0, 1[ $, and $C_{1-\alpha,T}(X_{1:T}) $ be a $1-\alpha$  pivotal Neyman-Pearson confidence region under  $(\Pter_{\theta}(.| X_{1:T}) )_{\theta \in \T}$, i.e.,  (i) for all $\omega \in \Omegabf$, $C_{1-\alpha,T}(X_{1:T}(\omega))\subset \T$, (ii) for all $ \theta \in \T $,  $\{ x_{1:T} \in \underline{\Sbf}_T: \theta \in C_{1-\alpha,T}(x_{1:T})\} \in \underline{\mathcal{S}}_T$, and (iii)   for all $ \theta \in \T $, $\Pter_{\theta}(  {\theta} \in C_{1-\alpha,T}(X_{1:T})| X_{1:T}) \geqslant 1-      \alpha$ m.a.e.
Under Assumptions \ref{assp1} and \ref{assp6},   $C_{1-\alpha,T}(X_{1:T})= \T $ $\P$-a.s. m.a.e. \item[ii)] Let $d_{T}$ be a Neyman-Pearson test of level $\alpha \in [0,1[$ under $\P(.|X_{1:T}) $, i.e.,  $\P ( d_T(X_{1:T}) = d_\mathrm{A}| X_{1:T}) \leqslant \alpha$ m.a.e., if $\mathrm{H}$ is true. Under Assumption \ref{assp1}, if $\mathrm{H}$ is true,  $ d_T(X_{1:T})=d_{\mathrm{H}}  $  $\P$-a.s. m.a.e., which implies that   there does not exist a Neyman-Pearson test of level $\alpha \in [0,1[$ under $\P(.|X_{1:T}) $ s.t. the probability of type I error is
nonnegative m.a.e. 
\end{itemize}

\end{prop}
\begin{proof} i) By  definition of conditional probabilities, for all $ \theta \in \T $, $\Pter_{\theta}$-a.s.,
\begin{eqnarray*}
\Pter_{\theta}({\theta} \in C_{1-\alpha,T}(X_{1:T})| X_{1:T}) &= & \Eter_\theta[ \ind_{C_{1-\alpha,T}(X_{1:T})}({\theta})| X_{1:T}]\\
&= & \ind_{C_{1-\alpha,T}(X_{1:T})}({\theta}) \text{ }\\
& = & \begin{cases}1 & \text{ if } {\theta} \in C_{1-\alpha,T}(X_{1:T})  \\
0 & \text{otherwise} \\
\end{cases} \end{eqnarray*}
where the second equality comes from the upcoming Lemma \ref{lem3}i).
Then, $\Pter_{\theta}({\theta} \in C_{1-\alpha,T}(X_{1:T})| X_{1:T}) \geqslant 1-      \alpha $ m.a.e.,  for all $ \theta \in \T $, if, and only if, $\theta \in C_{1-\alpha,T}(X_{1:T}) $ $\Pter_{\theta}$-a.s. m.a.e. for all $ \theta \in \T $. Now, by Assumption \ref{assp6}(b),  for all $\theta  \in \T $, $\P \ll\  \Pter_{\theta} $. Thus,  $ C_{1-\alpha,T}(X_{1:T})=\T$ $\P$-a.s. m.a.e.

ii) By definition of conditional probabilities, \text{ $\P$-a.s.},  
\begin{eqnarray*}
\P ( d_T(X_{1:T}) = d_\mathrm{A}| X_{1:T})
&= & \E[ \ind_{\{ d_T(X_{1:T}) = d_\mathrm{A} \}}| X_{1:T}] \\
&= & \ind_{\{ d_T(X_{1:T}) = d_\mathrm{A} \}}\\
& = & \begin{cases}1 & \text{ if }  d_T(X_{1:T}) = d_\mathrm{A}  \\
0 & \text{ if }  d_T(X_{1:T}) = d_\mathrm{H}  \\
\end{cases}
\end{eqnarray*}
where the second equality comes from the upcoming Lemma \ref{lem3}ii). Now,  $\P ( d_T(X_{1:T}) = d_\mathrm{A}| X_{1:T}) \leqslant \alpha$ m.a.e., if, and only if,  $d_T(X_{1:T}) = d_\mathrm{H}  $ $\P$-a.s. m.a.e. Thus,  the result follows.\end{proof}

Proposition   \ref{prop1}ii) shows that, when all data at use have been previously realized,   only Neyman-Pearson tests with zero probability type I error m.a.e. are possible.  Although it is possible to design such tests,  most available tests have  a nonnegative probability of type I error m.a.e. In addition, the Neyman-Pearson approach to testing is to minimize the probability of type II error rather than the probability of type I error (Neyman and Pearson, 1933).    

Proposition \ref{prop1}i) shows that, when all data at use have been previously realized,  the only possible Neyman-Pearson confidence region is  the whole parameter space $\T $ $\P$-a.s. m.a.e. Such a confidence region is uninformative.  The assumptions of Proposition \ref{prop1}i) that are new w.r.t. Assumptions  \ref{assp1} and \ref{assp6} are mild.  When the unknown parameter $\theta_0 $ can be any value inside the parameter space $\T $, it seems difficult to see how a confidence region cannot be  pivotal. Thus, Assumption  \ref{assp1} and \ref{assp6} are often part of the definition of confidence regions. E.g., Ferguson, 1967, sec. 5.8; Gouri\'eroux and Monfort, 1989/1996, sec. 20.

  The combination of Theorem \ref{thm1} and  Proposition \ref{prop1}
suggests that Neyman-Pearson theory is inadequate for past-realized data. This inadequacy  is stronger than the inadequacy for multiple use of the same data, because used data are necessarily a subset of the realized data.  Moreover, conditioning on the past realization instead of on the knowledge is more in line with Neyman-Pearson theory for at least two  reasons. First, unlike conditioning on  knowledge, conditioning on  past realizations is not individual-specific, which is a feature often presented as an advantage of the Neyman-Pearson theory over Bayesian theory. Second, unlike Bayesian theory, the Neyman-Pearson theory  distinguishes between unknown and random
quantities, so that  it is difficult to understand why past-realized data, which are now fixed, should be regarded as random.  

\begin{rk} If  randomized test decision rules are allowed, Proposition \ref{prop1}(ii) is weaker. In this case, if, instead of a level of the test,  the probability of type I error is fixed, it can be shown that randomized test decision rules $d_T $ and the data   $\underline{X_T} $  are independent under $\mathrm{H} $.   In this paper, we do not consider   randomized test decision rules in details for brevity and relevance:  they  are rarely used in econometrics. Moreover, Theorem \ref{thm1} remains mainly unchanged for randomized  test decision rules.  

\hfill $\diamond $    

\end{rk}

\begin{rk}While the Neyman-Pearson inadequacy for
multiple use of the same data has been mentioned in the literature, the
inadequacy for past-realized data is novel to the best of our knowledge. \hfill $\diamond $  

\end{rk}

\begin{lem} \label{lem3}
\begin{itemize}
\item[i)]  Under the notations and assumptions of Proposition \ref{prop1}(i), for all $\dot \theta \in \T $, the indicator function $ \ind_{C_{1-\alpha,T}(X_{1:T}(.))}(\dot{\theta}): \Omegabf \rightarrow \{0,1\}$ is  $\sigma(X_{1:T})/\mathcal{P}(\{0,1\}) $-measurable.

\item[ii)] Under the notations and assumptions of Proposition \ref{prop1}(ii),  the indicator function $ \ind_{\{ d_T(X_{1:T}(.)) = d_\mathrm{A} \}}: \Omegabf \rightarrow \{0,1\}$ is  $\sigma(X_{1:T})/\mathcal{P}(\{0,1\}) $-measurable.

\end{itemize}
 \end{lem}
\begin{proof} i)  Let $\dot \theta \in \T $. For this proof, define the functions $f:  \underline{\Sbf}_T \rightarrow \{0,1 \} $ and $h: \Omegabf \rightarrow \{0,1 \}   $ s.t.  $f(.):=\ind_{C_{1-\alpha,T}(.)}(\dot{\theta}) $ and $h(.):=\ind_{C_{1-\alpha,T}(X_{1:T}(.))}(\dot{\theta}) $. Then,  $h^{-1}(1)= (f \circ X_{1:T})^{-1}(1)=X_{1:T}^{-1}[ f^{-1}(1)]=X_{1:T}^{-1}(\{x_{1:T} \in \underline{\Sbf}_T: \dot\theta \in C_{1-\alpha,T}(x_{1:T})  \}) \in \sigma(X_{1:T})$,
where the last equality follows 
the defining property (ii) of confidence regions (see Definition \ref{defn2} on p. \ref{defn2}) and Assumption \ref{assp1}(c). Now, $\sigma(\{1\})=\mathcal{P}(\{ 0,1\}) $, and inverse mapping preserves the set operations that generates $\sigma$-algebra (e.g., Kallenberg, 1997/2002, p. 3, eq. (1)). Thus, the result follows. 

ii) Follow the same reasoning as for (i) with $ f(.):=\ind_{\{ d_T(.) = d_\mathrm{A} \}}$ and $h(.):=\ind_{\{ d_T(X_{1:T}(.)) = d_\mathrm{A} \}} $.
\end{proof}

%

\section{Existence of neoclassical confidence regions} \label{ap3}

In this Appendix, we prove the existence of neoclassical confidence regions under mild assumptions. We adapt a proof from Holcblat (2012). 
For notational convenience, we omit the qualification m.a.e., although all equalities should be understood m.a.e.

\begin{prop}[Existence of neoclassical confidence regions]\label{prop2} Under Assumptions \ref{assp1}, \ref{assp4}, and \ref{assp7}, for all $\alpha \in [0,1] $, there exists a neoclassical confidence region, $R_{1-\alpha,T}$.

\end{prop}
\begin{proof} 
Define, for all $\alpha \in [0,1] $,  
\begin{eqnarray*}
\dot k := \sup_{k\in \R}\left\{k: \tP\left( \left\{ \theta\in \mathbf{\mathbf{\Theta}}: f_{\theta^\bullet_T}(\theta)\geqslant k \right\}\right) \geqslant1-\alpha \right\}
\end{eqnarray*}
On the one hand, under Assumptions \ref{assp1}, \ref{assp4}, and  \ref{assp7}, $\tP\left( \left\{ \theta\in \mathbf{\mathbf{\Theta}}: f_{\theta^\bullet_T}(\theta)\geqslant 0 \right\}\right)=1  $. On the other hand, under Assumptions \ref{assp1}, \ref{assp4}, and  \ref{assp7}, by the upcoming Lemma \ref{lem2}, for all $\alpha \in [0,1] $, there exists $\ddot k \in \R_+\cup \{\infty\} $ s.t. $k \geqslant \ddot{k}  $ implies $\tP\left( \left\{ \theta\in \mathbf{\mathbf{\Theta}}: f_{\theta^\bullet_T}(\theta)\geqslant {k}\right\}\right) \leqslant 1-\alpha$. Thus, there exists an increasing sequence $\left( k_{n}\right)_{n\geqslant1} $ with  to $k_n \uparrow\dot k $ such that $\forall n \geqslant 1 $, $\tP\left( \left\{ \theta: f_{\theta^\bullet_T}(\theta)\geqslant k_n \right\}\right) \geqslant 1-\alpha $. Then,  by Lemma \ref{lem2}i),
\begin{eqnarray*}
\tP\left( \left\{ \theta\in \mathbf{\mathbf{\Theta}}: f_{\theta^\bullet_T}(\theta)\geqslant \dot k\right\}\right)  & \stackrel{}{=} &  \lim_{n\rightarrow \infty} \tP\left( \left\{ \theta\in \mathbf{\mathbf{\Theta}}: f_{\theta^\bullet_T}(\theta)\geqslant k_n \right\}\right)\\
& \geqslant & 1-\alpha. 
\end{eqnarray*}
Now, under Assumption  \ref{assp7},     $ \left\{ \theta\in \mathbf{\mathbf{\Theta}}: f_{\theta^\bullet_T}(\theta)\geqslant \dot k\right\} \in \mathcal{E}_{\T} $, so that, by construction,  it is a neoclassical confidence region. 
\end{proof}

%

\begin{lem}\label{lem2} Under Assumptions \ref{assp1}, \ref{assp4}, and  \ref{assp7}, \begin{enumerate}
\item[i)]  $\forall k\geqslant0 $,  $k \mapsto \tP\left( \left\{ \theta\in \mathbf{\mathbf{\Theta}}: f_{\theta^\bullet_T}(\theta)\geqslant k \right\}\right) $ is a left-continuous decreasing function;
\item[ii)] for all $\alpha \in [0,1] $, there exists $\ddot k \in \R_+ \cup \{ \infty\} $ s.t. $\tP\left( \left\{ \theta\in \mathbf{\mathbf{\Theta}}: f_{\theta^\bullet_T}(\theta)\geqslant \ddot{k}\right\}\right)< 1-\alpha$.
\end{enumerate}

\end{lem}
\begin{proof}

\textit{i)} Under   Assumption  \ref{assp7}, $\tP $ is probability measure, so that $k \mapsto \tP\left( \left\{ \theta\in \mathbf{\mathbf{\Theta}}: f_{\theta^\bullet_T}(\theta)\geqslant k \right\}\right) $ is decreasing by monotonicity of measures. Prove left-continuity. Let $\left( k_{n}\right)_{n\geqslant1}$ s.t. $k_n \uparrow\overline k \in \R_+ \cup \{\infty\}$. Then,  
\begin{eqnarray*}
\tP\left( \left\{ \theta\in \mathbf{\mathbf{\Theta}}: f_{\theta^\bullet_T}(\theta)\geqslant \overline{k}\right\}\right)  & = & \tP\left(\bigcap_{n \geqslant 1} \left\{ \theta\in \mathbf{\mathbf{\Theta}}: f_{\theta^\bullet_T}(\theta)\geqslant k_n \right\}\right)\\
& \stackrel{}{=} &  \lim_{n\rightarrow \infty} \tP\left( \left\{ \theta\in \mathbf{\mathbf{\Theta}}: f_{\theta^\bullet_T}(\theta)\geqslant k_n \right\}\right)
\end{eqnarray*}
where the last equality follows from a standard continuity property of measures under Assumption  \ref{assp7} (e.g.  in Kallenberg, 1997/2002, p. 8,  Lemma 1.14). \textit{}

\textit{ii)} For all real number $k >0 $, 
\begin{eqnarray*}
& &   k\int_\T \ind_{\{ \dot\theta \in \T: f_{\theta^\bullet_T}(\dot\theta)\geqslant k\}}(\theta) \mu( \d \theta) \leqslant \int_{\T}  f_{\theta^\bullet_T}(\theta)\mu( \d \theta)\\
%
& \stackrel{(a)}{\Rightarrow} & \lim_{k \rightarrow \infty}  \int_\T \ind_{\{ \dot\theta \in \T: f_{\theta^\bullet_T}(\dot\theta)\geqslant k\}}(\theta) \mu( \d \theta) \leqslant\lim_{k \rightarrow \infty} \frac{1}{k} =0\\
& \stackrel{(b)}{\Rightarrow} &   \int_\T \lim_{k \rightarrow \infty}\ind_{\{ \dot\theta \in \T: f_{\theta^\bullet_T}(\dot\theta)\geqslant k\}}(\theta) \mu( \d \theta)= \lim_{k \rightarrow \infty}  \int_\T \ind_{\{ \dot\theta \in \T: f_{\theta^\bullet_T}(\dot\theta)\geqslant k\}}(\theta) \mu( \d \theta)=0\\
& \stackrel{(c)}{\Rightarrow} &   \lim_{k \rightarrow \infty}f_{\theta^\bullet_T}(\theta)\ind_{\{ \dot\theta \in \T: f_{\theta^\bullet_T}(\dot\theta)\geqslant k\}}(\theta)=0 \text{ $\mu$-a.e.} \\
& \stackrel{(d)}{\Rightarrow} &   \lim_{k \rightarrow \infty} \int_\T f_{\theta^\bullet_T}(\theta)\ind_{\{ \dot\theta \in \T: f_{\theta^\bullet_T}(\dot\theta)\geqslant k\}}(\theta) \mu( \d \theta)=\int_\T \lim_{k \rightarrow \infty}f_{\theta^\bullet_T}(\theta)\ind_{\{ \dot\theta \in \T: f_{\theta^\bullet_T}(\dot\theta)\geqslant k\}}(\theta) \mu( \d \theta)=0  
\end{eqnarray*}
\textit{(a) }By Assumption \ref{assp7},  $ \int_{\T}  f_{\theta^\bullet_T}(\theta)\mu( \d \theta)=\tP(\T)=1$. \textit{(b)} Under Assumption  \ref{assp7}, for all $k \geqslant 1 $, $\ind_{\{\theta \in \T: f_{\theta^\bullet_T}(\theta)\geqslant k \}} \leqslant k\ind_{\{\theta \in \T: f_{\theta^\bullet_T}(\theta)\geqslant k \}} \leqslant f_{\theta^\bullet_T}(\theta) $ where $\int_\T |f_{\theta^\bullet_T}(\theta)| \mu(\d \theta) < \infty $.  Thus apply Lebesgue's dominated convergence theorem. \textit{(c)} First, by  definition of Lebesgue's integral, $\lim_{k \rightarrow \infty}\ind_{\{ \dot\theta \in \T: f_{\theta^\bullet_T}(\dot\theta)\geqslant k\}}(\theta)=0 $  $\mu$-a.e. (e.g.,  in Kallenberg, 1997/2002, p. 13,  Lemma 1.24). Second, Assumption \ref{assp7} implies that $f_{\theta^\bullet_T}(\theta) $ is finite $\mu$-a.e.: if there existed $B \in  \mathcal{E}_{\T}$ s.t. $\mu(B)>0 $ and $f_{\theta^\bullet_T}=\infty$ on $ B$, then, by definition of Lebesgue's integral, $\infty = \int_B f_{\theta^\bullet_T}(\theta)\mu(\d \theta) \leqslant \int_\T f_{\theta^\bullet_T}(\theta)\mu(\d \theta)=\tP(\T)=1 $.  \textit{(d)} 
Under Assumption  \ref{assp7}, for all $k\in \R $, $|f_{\theta^\bullet_T}(\theta)\ind_{\{ \dot\theta \in \T: f_{\theta^\bullet_T}(\dot\theta)\geqslant k\}}(\theta)| \leqslant f_{\theta^\bullet_T}(\theta)$
where $\int_\T |f_{\theta^\bullet_T}(\theta)| \mu(\d \theta) < \infty $. Thus apply Lebesgue's dominated convergence theorem.
 
Therefore, for all $\alpha \in [0,1] $, there exists $\ddot k \in \R_+ \cup \{\infty\} $ s.t. $\tP\left( \left\{ \theta\in \mathbf{\mathbf{\Theta}}: f_{\theta^\bullet_T}(\theta)\geqslant \ddot{k}\right\}\right)= \int_\T f_{\theta^\bullet_T}(\theta)\ind_{\{ \dot\theta \in \T: f_{\theta^\bullet_T}(\dot\theta)\geqslant \ddot k\}}(\theta) \mu( \d \theta)\leqslant 1-\alpha$. 
\end{proof}

\section{Analyses of the approximation error in the Example and  some variants of it}\label{ap8} 
The central object of study in assessing an applied neoclassical method is the distribution of $\widehat{\mathbb{P} \circ {\theta_T^\bullet}^{-1}}$. 
This problem can conceptually be dealt with in the same way as in the classical case. We here include some illustrations for deriving aspects of the law of $\widehat{\mathbb{P} \circ {\theta_T^\bullet}^{-1}}$ with respect to the probability measure of the data generating process. 
Generalizing this investigation parallels the development of classical statistical inference theory. We here mean merely to point out that such an investigation is a matter of mathematical development and sophistication, and is not a conceptual problem for the neoclassical framework.

In subsection \ref{section::exactnormal} we study the distribution of the cumulative distribution function and density of $\hP $ where we restrict attention to the example with i.i.d $\mathcal{N}(\theta_0, s)$ observations. In subsection \ref{appendix::metricbound} we study the concentration of $\hP$ around $\tP$ when $\mathcal{N}(\bar X_T, s)$ is used to approximate the distribution of $\theta_0 = \E (X_1)$ when $X_{1:T}$ are i.i.d. observations not necessarily from a Gaussian distribution. In the Gaussian case, we deduce the exact finite-sample distribution of the distance between $\hP$ and $\tP$ for the Hellinger and Wasserstein distances.

\subsection{The exact distribution of $\hP $} \label{section::exactnormal}
Let $\hat{F}_{\theta^\bullet_T}(x, \omega)$ be the cumulative distribution function induced by $\widehat{\mathbb{P} \circ {\theta_T^\bullet}^{-1}}$. 
Consider the basic example with i.i.d. $\mathcal{N}(\theta_0, s)$ observations, so that 
$
\hat{F}_{\theta^\bullet_T}(x, \omega) = \mathfrak{N}(\frac{x - \bar X_T(\omega)}{s_T(\omega)/\sqrt{T} }; 0; 1).
$
Let us identify the distribution of $\hat{F}_{\theta^\bullet_T}(x)$, first when $x$ is fixed, and, for the simpler case when $s$ is assumed known also deal with $x \mapsto \hat{F}_{\theta^\bullet_T}(x)$ as a stochastic process. 

For a given $x$, we have that the distribution of the random variable $\hat{F}_{\theta^\bullet_T}(x, \omega)$ is known exactly when the observations are i.i.d. $\mathcal{N}(\theta_0, s)$, since clearly
$H_x(y) := \mathbb{P} ( \omega \in \Omega : \hat{F}_{\theta^\bullet_T}(x, \omega) \leqslant y) = \mathbb{P} ( \omega \in \Omega :  \mathfrak{N} ( [x - \bar X_T(\omega)]/[s_T(\omega)/\sqrt{T}] ; 0; 1 )  \leqslant y ) = \mathbb{P} ( [x - \bar X_T]/[s_T/\sqrt{T}] \leqslant \mathfrak{N}^\inv(y  ; 0; 1) )$ where $\mathfrak{N}^\inv(y  ; 0; 1)$ is the inverse function of $\mathfrak{N}(x  ; 0; 1)$, i.e.~the standard Normal quantile function.
It is well known that $W := (T-1)s_T^2/s^2 \stackrel{d}{\sim} \chi_{T-1}^2$ is such that $(\bar X_T,W)$ are independent. Hence $Z := (x - \bar X )/(s/\sqrt{T}) \stackrel{d}{\sim} \mathcal{N}(-x/s,1)$ and $W$ are also independent. We hence see that $H_x$ is the cumulative distribution function of a quotient of two independent random variables with a known distribution, and that $H_x$ is therefore easy to obtain. Clearly, $H_x$ only depends on $T$ and the ratio $x/s$. If $s$ is assumed known, the distribution is known exactly. In the case of $s$ known, we also note that $\hat{F}_{\theta^\bullet_T}(x) = \mathfrak{N}( G(x) ; 0; 1)$ where $G$ is the continuous Gaussian process $G(x) = \sqrt{T} (x - \theta_0)/s + \varepsilon$ where $\varepsilon \stackrel{d}{\sim} \mathcal{N}(0, 1)$ and that $\hat{f}_{\theta^\bullet_T}(x) = (d/dx) \hat{F}_{\theta^\bullet_T}(x)$ equals $\sqrt{T} \mathfrak{n}( G(x) ; 0 ; 1) / s$.
Finally, $\arg \max_x \hat{f}_{\theta^\bullet_T}(x) = \arg \max_x \sqrt{T} \mathfrak{n}( G(x) ; 0 ; 1) / s = \arg \max_x \mathfrak{n}( G(x) ; 0 ; 1)$  is the solution to $G(x) = 0$, i.e. we regain the observation that $\arg \max_x \hat{f}_{\theta^\bullet_T}(x) = \theta_0 - (s/\sqrt{T}) \varepsilon \stackrel{d}{\sim} \mathcal{N}(\theta_0, s/\sqrt{T})$.

\subsection{A probability bound for $\rho (\hP, \tP) $} \label{appendix::metricbound}
 
A feature of $\widehat{\mathbb{P} \circ {\theta_T^\bullet}^{-1}}$ which is of special interest is its concentration around $\mathbb{P} \circ {\theta_T^\bullet}^{-1}$. 
Consider $\varrho = \rho(\widehat{\mathbb{P} \circ {\theta_T^\bullet}^{-1}}, \mathbb{P} \circ {\theta_T^\bullet}^{-1})$ based on some metric $\rho$ on the space of all probability measures on $(\T, \mathcal{E}_\T)$. As $\widehat{\mathbb{P} \circ {\theta_T^\bullet}^{-1}}$ is data-dependent, clearly $\varrho$ is a random mapping. 
Studying the law of $\varrho$ directly is generally complicated. This problem is however, connected to several well-studied problems in classical statistics and probability. We here illustrate some basic issues relating to the study of $\varrho$.

Let $\mathbb{P} \circ {\theta_{T, \infty}^\bullet}^{-1}$ be a distribution that is known to approximate $\mathbb{P} {\circ \theta_\infty^\bullet}^\inv$ on some appropriate scale. As in the proof of Lemma \ref{lem4}, we use the triangle inequality to see that 
\begin{equation} \label{equ::dbound}
\varrho \leqslant \rho(\widehat{\mathbb{P} \circ {\theta_T^\bullet}^{-1}}, \mathbb{P} \circ {\theta_{T, \infty}^\bullet}^{-1}) + \rho( \mathbb{P} \circ {\theta_{T, \infty}^\bullet}^{-1}, \mathbb{P} \circ {\theta_T^\bullet}^{-1}).
\end{equation}
If $\mathbb{P} \circ {\theta_{T, \infty}^\bullet}^{-1}$ is well-chosen, this bound can be used to derive bounds for the exceedance probabilities of $\varrho$. We note that this triangle inequality bound is general, but is likely to be crude compared to other bounds where more of the structure of the problem is used. 

In the case when $\sqrt{T} (\theta_T^\bullet - \theta_0) \xrightarrow[T \rightarrow \infty]{\mathcal{L}} \mathcal{N}(0, \Sigma^{\frac{1}{2}})$ for some covariance matrix $\Sigma$, the natural choice of $\mathbb{P} \circ {\theta_{T, \infty}^\bullet}^{-1}$ is the distribution $\mathcal{N}(\theta_0, \Sigma^{\frac{1}{2}}/\sqrt{T})$. 
In this case, we typically have that $\widehat{\mathbb{P} \circ {\theta_T^\bullet}^{-1}}$ has the data-dependent distribution $\mathcal{N}(\theta_0^*(\omega),  \hat \Sigma^{\frac{1}{2}}(\omega)/\sqrt{T})$, and we see that the first term in the above display quantifies the loss in precision in not knowing the population parameters that are in the Normal approximation of the law of $\theta_T^\bullet$. 

The first term in the bound of eq.~\eqref{equ::dbound} compares the distance between two exactly Normal distributions, one of which is data-dependent. This is the only stochastic term in eq.~\eqref{equ::dbound} and bounds, or even exact expressions for comparing the law of two Normals are available for several probability metrics as we will illustrate shortly.

The second term in the bound of eq.~\eqref{equ::dbound} quantifies how good this Normal approximation is to the distribution of $\theta_T^\bullet$ if these parameters are known, and is a well-studied problem, especially for the Kolmogorov metric where the celebrated Berry-Esseen Theorem applies.

\subsubsection{The i.i.d. Gaussian case}
In our example with i.i.d. $\mathcal{N}(\theta_0, s)$ observations, both the distribution of $\bar X_T^\bullet$ and the approximation $\mathcal{N}(\bar X_T(\omega), \frac{s_T(X_{1:T}(\omega))}{\sqrt{T}})$ are Gaussian and we can work directly with $\varrho = \rho(\mathcal{N}(\bar X_T, s_T/\sqrt{T}), \mathcal{N}(\theta_0, s/\sqrt{T}))$.

When $\rho$ is the Hellinger or total variation distance, we will see that $\varrho$ does not converge to zero, but does converge to zero with the Wasserstein distance. The Wasserstein distance is a finer metric than the Prohorov and the L\'evy metric. If at least one of the probability distributions that are compared have a density with respect to Lebesgue measure with finite supremum -- which is the case when comparing two Gaussians -- the L\'evy metric is in turn equivalent to the Kolmogorov metric (Gibbs and Su, 2002). On the ladder of probability metrics, the fact that $\varrho$ is $o_P(1)$ when $\rho$ is the Wasserstein metric but not when $\rho$ is the Hellinger metric gives information on how fine-grained the neoclassical approximation is in an elementary case.

Let $\rho$ be the Hellinger distance, i.e., the metric on the space of probability measures with densities with respect to Lebesgue measure given by $\rho_H(\nu_1,\nu_2) = \left[ \int_{\mathbb{R}} (\sqrt{f_1} - \sqrt{f_2})^2 \, \d \lambda \right]^{1/2}$ where $f_1, f_2$ are the densities with respect to Lebesgue measure of $\nu_1, \nu_2$ respectively (Gibbs and Su, 2002).
Letting $f_1, f_2$ be the densities of $\mathcal{N}(\bar X_T, s_T/\sqrt{T})$ and $\mathcal{N}(\theta_0, s/\sqrt{T})$ respectively, we see that
$1 - \varrho^2/2 =  \int_\mathbb{R} \sqrt{f_1 f_2} \, \d \lambda  = \sqrt{ \frac{2 (s_T/\sqrt{T}) (s/\sqrt{T})}{s_T^2/T + s^2/T}    } \exp \left\{ -\frac{1}{4} \frac{ (\bar X_T - \theta_0)^2}{s_T^2/T + s^2/T} \right\} 
= \sqrt{ \frac{2 s_T s }{s_T^2 + s^2}    } \exp \left\{ -\frac{T }{4 } \frac{ (\bar X_T - \theta_0)^2}{s_T^2 + s^2} \right\}$, which shows that $\varrho = \sqrt{2} \left( 1 - \sqrt{\frac{2 s_T s }{s_T^2 + s^2}} \exp \left\{ -\frac{s^2 }{4  } \frac{ Z_T^2}{s_T^2 + s^2} \right\}  \right)^{1/2}$ where $Z_T := \sqrt{T} (\bar X_T - \theta_0)/s \stackrel{d}{\sim} \mathcal{N}(0,1)$.
If $s$ is assumed known, we can replace $s_T$ with $s$ in the above expression and get $\varrho = \sqrt{2} \sqrt{ 1 -  \e^{ - \chi_1^2/8 } }$, which has a distribution that does not change with $T$. In both cases, $\varrho$ does not converge to zero, but is a non-degenerate random variable. The same conclusion holds for the total variation metric, since it is equivalent to the Hellinger metric (Gibbs and Su, 2002).

Let $\rho$ be the Wasserstein distance, i.e., the metric on the space of probability measures given by $\rho_W(\nu_1,\nu_2) = \int_{\mathbb{R}} |F_1(x) - F_2(x)| \d x$ where $F_1, F_2$ are the cumulative distribution functions $\nu_1, \nu_2$ respectively (Gibbs and Su, 2002). 
Letting $F_1, F_2$ be the cumulative distribution functions of $\mathcal{N}(\bar X_T, s_T/\sqrt{T})$ and $\mathcal{N}(\theta_0, s/\sqrt{T})$ respectively, (Givens and Shortt, 1984, equation (4)) shows 
\begin{equation}
\varrho = \rho_W(F_1, F_2) = \sqrt{ |\bar X_T - \theta_0|^2 + s_T^2/T + s^2/T - 2 s s_T/T)}
\end{equation}
which clearly goes to zero in probability. When $s$ is assumed known, we see that $\varrho = |\bar X_T - \theta_0| = s |\mathcal{N}(0,1)|/\sqrt{T}$, which follows a folded Normal distribution, so that $\E \varrho = T^{-1/2} \sqrt{2/\pi}$. In both cases, $\varrho = O_P(T^{-1/2})$.

\subsubsection{The general i.i.d. case}
Let us now work with the case with i.i.d. observations $X_{1:T}$ from a distribution with finite third order moment $\zeta$, variance $s^2$ and expectation $\theta_0$, but not necessarily being absolutely continuous with respect to Lebesgue measure. We use $\bar X_T$ as a proxy for $\theta_0$ and, motivated by the Central Limit Theorem, approximate $\mathbb{P} \circ {\theta_{T}^\bullet}^{-1}$ with $\mathcal{N}(\bar X_T(\omega), s_T(\omega)/\sqrt{T})$. Because $\mathbb{P} \circ {\theta_{T}^\bullet}^{-1}$ is not Normal, we resort to the bound in eq.~\eqref{equ::dbound}.
For the second term in the bound, we will use the Berry-Esseen bound, which is given in terms of the Kolmogorov distance $\rho_K(\nu_1, \nu_2) = \sup_{x} |F_1(x) - F_2(x)|$ where $F_1, F_2$ being the c.d.f. induced by $\nu_1, \nu_2$ respectively.

From eq.~\eqref{equ::dbound}, we see 
$\rho_K(\mathcal{N}(\bar X_T, s_T/\sqrt{T}),  \mathbb{P} \circ {\theta_{T}^\bullet}^{-1})  \leqslant \varrho_1 + \varrho_2$ where $\varrho_1 := \rho_K(\mathcal{N}(\bar X_T, s_T/\sqrt{T}), \mathcal{N}(\theta_0, s/\sqrt{T}))$ and $\varrho_2 := \rho_K(\mathcal{N}(\theta_0, s/\sqrt{T}), \mathbb{P} \circ {\theta_{T}^\bullet}^{-1})$.
The first term compares two Gaussians. Calculations show that
$\varrho_1 = |\mathfrak{N}(Q;0;1) - \mathfrak{N}(s_t Q/s + Z_T;0;1)$ for $Z_T = \sqrt{T}(\bar X_T - \theta_0)/s $
where $Z_T = \sqrt{T} (\bar X_T - \theta_0)/s \stackrel{d}{\sim} \mathcal{N}(0,1)$ and $Q = (s_T Z/s - \sqrt{ Z_T^2 + 2(s^2/s_T^2 - 1) \log (s/s_T) })/(1-s^2/s_T^2)$. We hence have an exact expression for $\varrho_1$, and we see it is $o_P(1)$. 
We have
$
\rho_K(\mathcal{N}(\theta_0, s/\sqrt{T}), \mathbb{P} \circ {\theta_{T}^\bullet}^{-1}) = \sup_{x \in \mathbb{R}} | \mathfrak{N}(\sqrt{T} (x - \theta_0)/s ; 0; 1) - \P( \bar X_n \leqslant x) | 
= \sup_{x \in \mathbb{R}} | \mathfrak{N}(\sqrt{T} (x - \theta_0)/s ; 0; 1) - \P( Z_n \leqslant \sqrt{T} (x - \theta_0)/s)) |  \
= \sup_{z\in \mathbb{R}} | \mathfrak{N}(\sqrt{T} z ; 0; 1)- \P( Z_T \leqslant \sqrt{T} z) | = \sup_{z\in \mathbb{R}} | \mathfrak{N}(z ; 0; 1)- \P( Z_T \leqslant z) | 
$
which is bounded by $C \zeta s^{-3}/\sqrt{T}$ by the Berry-Esseen bound for a $(2 \pi)^{-1/2} < C < 0.4748$ (Shevtsova, 2011).

\section{Consistency of approximations} \label{ap7}

In this appendix, we investigate the consistency of the approximations presented in subsection \ref{sec5_2}. As in the main text of this paper, we use the Prokhorov metric, and  the  Lemma \ref{lem4} (p. \pageref{lem4}). The following Lemma \ref{lem5}
ensures that the condition (b) of Lemma \ref{lem4}  is satisfied,  so that it only remains to check the condition (a) of Lemma \ref{lem4}
for each approximation.  
\begin{lem}\label{lem5} Let $\rho_{P}(.,.)$ be the Prokhorov metric. Under Assumptions \ref{assp1} and \ref{assp4}, if $\theta^\bullet_{T} \stackrel{\P}{\rightarrow} \theta_0 $ as $T \rightarrow \infty $, then $\rho_P(\tP, \delta_{\theta_0}) \stackrel{}{\rightarrow} 0 $ as $T \rightarrow \infty $.
\end{lem}
\begin{proof} Convergence in probability implies convergence in law (e.g., Kallenberg, 1997/2002, Lemma 4.7), which, in turn, is equivalent to the convergence w.r.t. $\rho_P $.
\end{proof}

\subsection{Consistency of calibration} \label{ap7_1}
In this subsection, we study the asymptotic properties of calibration from a neoclassical point of view under the assumption that calibration yields consistent parameter values.  This assumption should  hold if the calibrated parameter value  corresponds to estimates from existing
empirical studies, or to the minimizer
of some goodness-of-fit measure (e.g., Gouri\'eroux and Monfort, 1996, sec. 2.1.2.).  For  criterium-adjusted calibration, we require the following properties from the criterium function.   

\begin{assp}\label{assp12} \emph{\textbf{ (a)}} The criterion function  $u : \R^p  \times \R^p  \rightarrow \R_+ $ equals zero outside $\T^2 $, and $\theta \mapsto u(\theta, \dot \theta) $ is Borel measurable for all $\dot \theta \in \T $. \emph{\textbf{ (b)}} For all $\dot \theta \in \T $, $\theta \mapsto u(\dot \theta,\theta )$ is continuous in a neighborhood of $\theta_0 $. \emph{\textbf{ (c)}} There exists $r_{1}>0 $ s.t.  $\int_\T \sup_{\dot\theta \in B_{r_1}(\theta_0)} u( \theta, \dot\theta ) \lambda( \d \theta) < \infty$, where $B_{r_1}(\theta_0)$ denotes a ball in $\T$ centered at $\theta_0$ with radius $r_{1}$. \emph{\textbf{ (d)}}  There exists $r_{2}>0 $ s.t., for all $\dot \theta \in B_{r_2}(\theta_0) $, $0<\int_\T u( \theta, \dot\theta ) \lambda( \d \theta) $.
\end{assp}
Assumption \ref{assp12}(a) requires the criterion function to be positive and to take zero values outside $\T^2$.   As explained in Remark \ref{rk3} on p. \pageref{rk3},  Assumption \ref{assp12}(a) (combined with Assumption \ref{assp12}(d)) allows us to transform criterion functions into a p.d.f., so that we remain within the elementary framework of section \ref{sec4_2}.  Assumption \ref{assp12}(b) requires continuity of criterion function, which is a standard and relatively mild requirement: continuous functions are dense in the set of Borel measurable functions (Kallenberg, 1997/2002, p. 19, Lemma 1.37).
Assumption \ref{assp12}(c) allows us to use Lebesgue dominated convergence theorem. Compactness of $\T$, and continuity of the criterion function over $\T^2 $ would ensure Assumption \ref{assp12}(c). Assumption \ref{assp12}(d), which seems mild, and allows us to normalize the criterion function, so that it integrates to one. 

\begin{prop}\label{prop3} Denote the Prokhorov metric with $\rho_P $. Under Assumptions \ref{assp1}, and \ref{assp4},  if $\theta^*_{T,C} \stackrel{\P}{\rightarrow} \theta_0 $ as $T \rightarrow \infty $, then
\begin{enumerate}
\item[i)]  $\rho_{P}(\delta_{\theta^*_{T,C}}, \delta_{\theta_0} ) \stackrel{\P}{\rightarrow} 0 $, as $T \rightarrow \infty $, which, in turn, implies consistency of plain calibration by Lemmas \ref{lem4} and \ref{lem5};
\item[ii)] under the additional Assumption  \ref{assp12},  $\rho_P\left( \frac{\int_. u(\theta,\theta^*_{T,C} )\lambda( \d \theta )}{ \int_\T u(\dot \theta, \theta^*_{T,C})\lambda(\d \dot \theta) }, \frac{\int_. u(\theta,\theta_0 )\lambda( \d \theta )}{ \int_\T u(\dot \theta, \theta_0)\lambda(\d \dot \theta) } \right) \stackrel{\P}{\rightarrow} 0 $, as $T \rightarrow \infty $, which, in turn, implies consistency of weighted calibration by Lemmas \ref{lem4} and \ref{lem5}. 
\end{enumerate}
\end{prop}
\begin{proof} 
\textit{i)} By Portmanteau theorem, it is sufficient to check  the point-wise convergence of cumulative distribution functions (c.d.f.) at the continuity points of the limiting c.d.f.:
for all $\theta \in \T  $ that are continuity points of $\ind_{[\theta_0, \infty[}(.) $, $\lim_{T \rightarrow \infty}\ind_{[\theta^*_{T,C}, \infty[}(\theta) \stackrel{\P}{\rightarrow} \ind_{[\theta_0, \infty[}(\theta) $, as $T \rightarrow \infty $ , by assumption.  

\textit{ii)} Under Assumption \ref{assp12} (a)-(c), by the Lebesgue dominated convergence theorem,  for all $B \in  \mathcal{E}_{\T}  $, as $T \rightarrow \infty $,
\begin{eqnarray*}
\int_B u(\theta,\theta^*_{T,C} )\lambda( \d \theta ) \stackrel{\P}{\rightarrow} \int_B u(\theta,\theta_0 )\lambda( \d \theta ).
\end{eqnarray*}
Therefore, by Assumption \ref{assp12} (d),  for all $B \in  \mathcal{E}_{\T}  $, as $T \rightarrow \infty $,
\begin{eqnarray*}
\frac{\int_B u(\theta,\theta^*_{T,C} )\lambda( \d \theta )}{ \int_\T u(\dot \theta, \theta^*_{T,C})\lambda(\d \dot \theta) } \stackrel{\P}{\rightarrow} \frac{\int_B u(\theta,\theta_0 )\lambda( \d \theta )}{ \int_\T u(\dot \theta, \theta_0)\lambda(\d \dot \theta) }.
\end{eqnarray*}
where the LHS and RHS are probability measures  on $(\T,  \mathcal{E}_{\T}) $ by Assumption  \ref{assp12} (a) and standard properties of the Lebesgue integral. Then the result follows by the Portmanteau theorem.    
\end{proof}

\subsection{Consistency of Gaussian and Laplace approximations} \label{ap7_2} In this subsection, we  show that asymptotic normality of $ \hP $ implies consistency of $ \hP $ under mild assumptions, i.e., asymptotic normality of $ \hP $ implies condition  (a) of Lemma \ref{lem4}.  Then, from this general result, we deduce consistency of the  Gaussian and Laplace approximations.
 
\subsubsection{Consistency from asymptotic normality} \label{ap7_2_1}
 It is  is well-known that asymptotic normality implies consistency for random variables. To the best of our knowledge,  there is no result available for random probability measures. We require the following two mild Assumptions   \ref{assp16}  and \ref{assp13} to prove this result for  $ \hP $. Assumption \ref{assp16} ensures that we can study the convergence of $ \hP $ on the space of probability measures.

 \begin{assp} \label{assp16} The approximation of the distribution of the generic proxy is  a random probability measure   from $\underline{\Sbf}_\infty$ to $\T$ for $T$ big enough $\P$-a.s., i.e., for $T$ big enough, (a) [measurability condition] for all $B \in \mathcal{E}_\T$, $x_{1:\infty} \mapsto \hP(B)$ is $\underline{\mathcal{S}}_\infty/\mathcal{B} $-measurable; and (b) [probability condition]  $\hP $ is a probability  measure $\P$-a.s.

 \end{assp}
 
Assumption \ref{assp16} requires that we can approximate the distribution of $\theta^\bullet_T $  $\P$-a.s.  in a measurable way. The following  Assumption \ref{assp13}
allows us to consider the Cholesky decomposition of the proxy of the asymptotic variance-covariance matrix, and its inverse.
\begin{assp} \label{assp13}  $\tilde{\Sigma}_T\stackrel{\P}{\rightarrow}\tilde{\Sigma} $, where  $\tilde{\Sigma} $ is a positive-definite  matrix, and   $(\tilde{\Sigma}_T)_{T \in \[1,\infty \[} $   a sequence of $\sigma(X_{1:\infty})/\mathcal{B}(\R^{p^2}) $-measurable matrices that are symmetric w.p.a.1, as $T  \rightarrow \infty $. 
\end{assp}

\begin{rk} $\tilde{\Sigma} $ does not need to be $\Sigma$, which is the asymptotic variance of $\sqrt{T}(\theta^\bullet_T- \theta_0) $: see Assumption \ref{assp9} on p. \pageref{assp9}. \hfill $\diamond $

\end{rk}

The following Lemma is the main result of this subsection. 

\begin{lem}\label{lem8} Denote the c.d.f.
 of $\hP $ with $\hF $, i.e., for all $\theta \in \T $, $\hF(\theta)=\hP(]-\infty, \theta]) $.  Under Assumptions \ref{assp1},  \ref{assp4}, \ref{assp16}, and \ref{assp13} if, as $T \rightarrow \infty $,
\begin{itemize}
\item[(a)]  $\tilde{\theta}_{T} \stackrel{\P}{\rightarrow} \theta_0 $, where $\tilde{\theta}_{T} $ is $\sigma(X_{1:\infty})/\mathcal{E}_\T $-measurable, 
\item[(b)] 
 for all $b \in \R^p $, $\hat{F}_{\theta^\bullet_T} \left(\tilde{\theta}_{T}+\tilde{\Sigma}_T^{\frac{1}{2}}\frac{b}{\sqrt{T}}\right) \stackrel{\P}{\rightarrow} \mathfrak{N}(b;0;\grave{\Sigma}^{\frac{1}{2}}) $,
\end{itemize}
  then, $ $ as $T \rightarrow \infty $, 
$$
\rho_P\left( \hF , \delta_{\theta_0} \right) \stackrel{\P}{\rightarrow} 0, \footnote{ In this statement, we identify the c.d.f. with  its corresponding probability measure. Hereafter, we follow the same abuse of notation.}$$  
which, in turn, implies consistency of $\hP $ by Lemmas \ref{lem4} and \ref{lem5}.
\end{lem}
\begin{proof}  The idea of the proof is to map probability measures to random variables, and then note that   asymptotic normality trivially implies consistency for random variables.  By the upcoming Lemma \ref{lem7}, for  $T  $ big enough, there exists a random variable $Z_T :(\Omegabf \times \T,\mathcal{E}_\Omegabf \otimes \mathcal{E}_\T )\rightarrow (\T, \mathcal{E}_\T) $ s.t. for all $\theta \in \R^p $,
\begin{eqnarray}
\breve{\P}(Z_T \leqslant \theta|X_{1:\infty})=\hat{F}_{\theta^\bullet_T}(\theta),\label{eq13} \end{eqnarray} 
 where $\breve{\P}$ is a probability measure defined by equation (\ref{eq15}) in Lemma  \ref{lem7}.  If a sequence of symmetric matrices   converges to a positive-definite matrix, the matrices  of the sequence  are positive-definite for an index big enough.\footnote{By Corollary III.2.6 in Bhatia (1997), $ \max_{j \in \[1,p \]} |\text{eig}_j^{\downarrow}(\tilde{\Sigma}_T)-\text{eig}_j^{\downarrow}(\tilde{\Sigma})| \leqslant \Vert \tilde{\Sigma}_T-\tilde{\Sigma} \Vert  $ where $ \text{eig}_j^{\downarrow}(\tilde{\Sigma})$ denotes the eigenvalues of $ \tilde{\Sigma}$  in descending order. Now,  all the eigenvalues of $\tilde{\Sigma} $ are strictly positive, so that the eigenvalues of $\tilde{\Sigma}_T $ are strictly positive$ $ w.p.a.1, as $T\rightarrow \infty $. } Thus, by Assumption  \ref{assp13},  $ {\tilde{\Sigma}_T}$ is a positive-definite matrix w.p.a.1 as $T \rightarrow \infty $, which implies that it has a Cholesky decomposition w.p.a.1 as $T \rightarrow \infty $.
Therefore, for  $T  $ big enough, define 
\begin{eqnarray}
Y_T:=\sqrt{T}\tilde{\Sigma}^{-\frac{1}{2}}_T(Z_T-\tilde{\theta}_{T}) , \label{eq14}
\end{eqnarray}
so that, for all $y \in \R^p $, 
\begin{eqnarray*} 
\breve{\P}(Y_T\leqslant y| X_{1:\infty})& = & \breve{\P}(\sqrt{T}\tilde{\Sigma}^{-\frac{1}{2}}_T(Z_T-\tilde{\theta}_{T})\leqslant y| X_{1:\infty})\\
& = & \breve{\P}\left(Z_T \leqslant \tilde{\theta}_{T} +\tilde{\Sigma}^{\frac{1}{2}}_T\frac{y}{\sqrt{T}} | X_{1:\infty} \right)\\
& \stackrel{(a)}{=} & \hat{F}_{\theta^\bullet_T}( \tilde{\theta}_{T} +\tilde{\Sigma}^{\frac{1}{2}}_T\frac{y}{\sqrt{T}}) \text{  }\\ %
& \stackrel{(b)}{\stackrel{\P}{\rightarrow}} &   \mathfrak{N}(y;0;\grave{\Sigma}^{\frac{1}{2}}) \text{ as $T \rightarrow \infty $} 
\end{eqnarray*}
\textit{(a)} By the $\sigma(X_{1:\infty})$-measurability of  $\tilde{\theta}_{T} $ and $\tilde{\Sigma}_T$ (see condition (a), and Assumption \ref{assp16}), disintegration of $\breve{\P}$ w.r.t. $X_{1:\infty} $ allows us to regard  $\tilde{\theta}_{T} $ and $\tilde{\Sigma}_T $ as fixed (e.g., Kallenberg, 1997/2002, Theorem 6.4), and thus to use equation (\ref{eq13}). \textit{(b)} Use condition (b). 

Now, on one hand, convergence in probability is equivalent to convergence in $L^1$ for uniformly integrable sequences (e.g., Kallenberg, 1997/2002, Proposition 4.12).  On the other hand, $\breve{\P}(Y_T\leqslant y| X_{1:\infty})\leqslant 1 $, which implies that it is uniformly integrable. Therefore, as $T \rightarrow \infty $, 
\begin{eqnarray*}
& & \breve{\E}\left|\breve{\P}(Y_T\leqslant y| X_{1:\infty})-\mathfrak{N}(y;0;\grave{\Sigma}^{\frac{1}{2}})\right| \stackrel{\breve{\P}}{\rightarrow }0 \\
& \stackrel{(a)}{\Rightarrow}  & \breve{|\P}(Y_T\leqslant y)-\mathfrak{N}(y;0;\grave{\Sigma}^{\frac{1}{2}})| \stackrel{\breve{\P}}{\rightarrow }0\\
&\stackrel{(b)}{\Rightarrow}  & Y_T \stackrel{d}{\rightarrow} \mathcal{N}(0,\grave{\Sigma}) \text{ , under $\breve{\P} $ }\\
%
& \stackrel{(c)}{\Rightarrow}   & \frac{ Y_T}{\sqrt{T}}=O_{\breve{\P}}\left(\frac{1}{\sqrt{T}}\right)\\
& \stackrel{(d)}{\Rightarrow}  & \frac{ Y_T}{\sqrt{T}}+\tilde{\theta}_{T} \stackrel{\breve{\P}}{\longrightarrow}  \theta_0\\
& \stackrel{(e)}{\Rightarrow}  & Z_T \stackrel{d}{\rightarrow} \theta_0 \text{ , under $\breve{\P} $ }\\
& \stackrel{(f)}{\Rightarrow}  & \breve{\P}(Z_T \leqslant \theta)\stackrel{\breve{\P}}{\rightarrow} \ind_{[\theta_0, \infty[}(\theta) \text{ at any continuity point $\theta $ of $\ind_{[\theta_0, \infty[}(.) $}\\
& \stackrel{(g)}{\Rightarrow}  &\hat{F}_{\theta^\bullet_T}(\theta)= \breve{\P}(Z_T \leqslant \theta|X_{1:\infty})\stackrel{\breve{\P}}{\rightarrow} \ind_{[\theta_0, \infty[}(\theta) \text{ at any continuity point $\theta $  of $\ind_{[\theta_0, \infty[}(.) $}\\
& \stackrel{(h)}{\Rightarrow}  & \rho_P\left( \hF , \delta_{\theta_0} \right) \stackrel{\P}{\rightarrow} 0 
\end{eqnarray*}
\textit{(a)} By iterated conditioning and positivity of probabilities,    $\breve{|\P}(Y_T\leqslant y)-\mathfrak{N}(y;0;\grave{\Sigma}^{\frac{1}{2}})|=\left|\breve{\E}\breve{|\P}(Y_T\leqslant y| X_{1:\infty})|-\breve{\E}|\mathfrak{N}(y;0;\grave{\Sigma}^{\frac{1}{2}})|\right| \leqslant  \breve{\E}\left|\breve{\P}(Y_T\leqslant y| X_{1:\infty})-\mathfrak{N}(y;0;\grave{\Sigma}^{\frac{1}{2}})\right|$, where the inequality follows from the reverse triangle inequality for the $L^1 $ norm. \textit{(b)-(c)} Portmanteau theorem. \textit{(d)} By condition (a),  $\tilde{\theta}_{T} \stackrel{\P}{\rightarrow} \theta_0$,  as $T \rightarrow \infty $, which implies that $\tilde{\theta}_{T} \stackrel{\breve{\P}}{\rightarrow} \theta_0 $, as $T \rightarrow \infty $ by the upcoming Lemma \ref{lem7}(iii). \textit{(e)} Definition of $Y_T $ by equation (\ref{eq14}). \textit{(f)} Apply portmanteau theorem. \textit{(g)} By consistency adequacy  (see Remark \ref{rk6} in Appendix \ref{ap1} on p. \pageref{ap1}, and note that  $\breve{\P} $ and $\tilde{\theta}_T $ can take the place of $\P $ and $\hat{\theta}_T $, respectively, in the statement of Proposition \ref{prop6}), $\breve{\P}(Z_T \leqslant \theta|X_{1:\infty})\stackrel{\breve{\P}}{\rightarrow} \ind_{[\theta_0, \infty[}(\theta)$, where $\breve{\P}(Z_T \leqslant \theta|X_{1:\infty})=\hat{F}_{\theta^\bullet_T}(\theta) $ by  equation (\ref{eq13}).\textit{(h)} Lemma \ref{lem7}(iii),  $\hat{F}_{\theta^\bullet_T}(\theta)\stackrel{\P}{\rightarrow} \ind_{[\theta_0, \infty[}(\theta)$ at continuity point of $\ind_{[\theta_0, \infty[}(.) $. Then, note that the converge w.r.t. the Prokhorov metric corresponds to the convergence in law, which, in turn, is equivalent to the convergence of the c.d.f. at the continuity points of the limiting c.d.f. \end{proof}

\begin{lem}
\label{lem7} 
Let    $\kappa:\underline{\Sbf}_\infty \times  \mathcal{E}_\T \rightarrow [0,1]$  be a  random probability measure   from $\underline{\Sbf}_\infty$ to $\T$.
Under Assumptions \ref{assp1} and \ref{assp4}(a)(b), there exist
\begin{itemize}
\item[i)] a probability measure $\breve{\P} $ on $(\Omegabf \times \T,\mathcal{E}_\Omegabf \otimes \mathcal{E}_\T ) $ s.t.,  for all $A \in  \mathcal{E}_\Omegabf \otimes\ \mathcal{E}_\T$, 
\begin{eqnarray}
\breve{\P}(A)=\E\int_\T \ind_A(.,\theta)\kappa(X_{1:\infty}, \d \theta),\label{eq15}
\end{eqnarray}
\item[ii)] a random vector $Z:(\Omegabf \times \T,\mathcal{E}_\Omegabf \otimes \mathcal{E}_\T )\rightarrow (\T, \mathcal{E}_\T) $ s.t., for all $B \in  \mathcal{E}_\T $,
\begin{eqnarray*}
\breve{\P}(Z \in B|X_{1:\infty} )=\kappa(X_{1:\infty}, B )\quad \breve{\P}-a.s.,
\end{eqnarray*} 
\item[iii)] and, for all random sequences $(W_T)_{T=1}^\infty $,   $W_{T} \stackrel{\P}{\rightarrow} K $,  as $T \rightarrow \infty $, where $K $ is a constant vector, is equivalent to   $W_{T} \stackrel{\breve{\P}}{\longrightarrow} K $,  as $T \rightarrow \infty $.
\end{itemize}

\end{lem}

\begin{proof} i)-ii) It corresponds to Lemma 6.9 in Kallenberg (1997/2002).
iii) For all neighborhoods $N_{K}$ of $K$, 
\begin{eqnarray*}
\breve{\P}( W_T \in N_{K} )& = & \breve{\P}\{(\omega, \theta)\in  \Omegabf \times \T:W_T(\omega) \in N_{K} \}\\
 %
 %
  & \stackrel{(a)}{=} & \E\int_\T \ind_{\{\omega\in  \Omegabf :W_T(\omega) \in N_{K} \} \times \T}(., \theta)\kappa(X_{1:\infty}, \d \theta)\\
& \stackrel{(b)}{=} & \E \ind_{\{\omega\in  \Omegabf :W_T(\omega) \in N_{K} \} }(.)\int_\T\kappa(X_{1:\infty}, \d \theta)\\
%
%
& \stackrel{(c)}{=} & \E \ind_{\{\omega\in  \Omegabf :W_T(\omega) \in N_{K} \} }(.)=\P( W_T \in N_{K})
\end{eqnarray*}
\textit{(a)} Definition of $\breve{\P} $ given by equation (\ref{eq15}). \textit{(b)} For all $\dot \omega \in \Omegabf $ and $\theta \in \T $,  $\ind_{\{\omega\in  \Omegabf :W_T(\omega) \in N_{K} \} \times \T}(\dot \omega, \theta)= \ind_{\{\omega\in  \Omegabf :W_T(\omega) \in N_{K} \} }(\dot \omega)$. \textit{(c)} By, definition of the Lebesgue integral for step functions, $\int_\T\kappa(X_{1:\infty}, \d \theta)= \kappa(X_{1:\infty}, \T)$, and then $\kappa(X_{1:\infty}, \T)=1 $ because, for all  $x_{1:\infty}\in\underline{\Sbf}_\infty   $, $\kappa(x_{1:\infty},.) $ is a probability measure on $(\T , \mathcal{E}_\T) $.  

Therefore,  $W_{T} \stackrel{\P}{\rightarrow} K $,  as $T \rightarrow \infty $, is equivalent to  $W_{T} \stackrel{\breve{\P}}{\longrightarrow} K $,  as $T \rightarrow \infty $.
%
\end{proof}

\subsubsection{Consistency of Gaussian approximations}

\begin{prop}\label{prop5}  Let $\tilde{\Sigma} $ be a positive-definite  matrix, and $(\hat{\Sigma}_T)_{T=1}^{\infty}  $ a sequence of $\sigma(X_{1:\infty})/\mathcal{B}(\R^{p^2}) $-measurable matrices. Under Assumptions \ref{assp1},  \ref{assp4}, and \ref{assp9}, if, as $T \rightarrow \infty $, $\hat{\Sigma}_T \stackrel{\P}{\rightarrow} \Sigma$ where $\Sigma $ is a positive-definite matrix, then,  as $T \rightarrow \infty $, 
$$
\rho_P\left( \mathfrak{N}\left(.;\theta^*_{T,G}; \frac{\diag(\hat{\Sigma}_T)^{\frac{1}{2}}}{\sqrt{T}} \right) , \delta_{\theta_0} \right) \stackrel{\P}{\rightarrow} 0,$$  
which, in turn, implies consistency of the Gaussian approximation by Lemmas \ref{lem4} and \ref{lem5}. \end{prop}
\begin{proof}  Check assumptions of  Lemma \ref{lem8} with $\hF(.)=\mathfrak{N}\left(.;\theta^*_{T,G}; \diag(\hat{\Sigma}_T)^{\frac{1}{2}}/\sqrt{T} \right)  $,  $\tilde{\Sigma}_T=\diag(\hat{\Sigma}_T) $, and   for $\tilde{\Sigma}=\diag(\tilde{\Sigma})  $, and then apply it.  Assumption \ref{assp16} is verified because $\hF(.)=\mathfrak{N}\left(.;\theta^*_{T,G}; \diag(\hat{\Sigma}_T)^{\frac{1}{2}}/\sqrt{T} \right)=\mathfrak{N}\left(\sqrt{T}\diag(\hat{\Sigma}_T)^{-\frac{1}{2}}(.-\theta^*_{T,G});0, I \right)   $  where $\theta^*_T $ and $\hat{\Sigma}_T $ are measurable by definition, and  $\mathfrak{N}\left(.;0, I \right)$ is continuous. Assumption \ref{assp13} is verified  because a diagonal matrix is symmetric, and $\Sigma$ is positive-definite by assumption. Assumption (a) of Lemma \ref{lem8} is verified by Assumption \ref{assp9}.  Assumption (b) of Lemma \ref{lem8} is verified because, by the change of variable $u=\sqrt{T} \diag(\hat{\Sigma}_T)^{-\frac{1}{2}}(\theta-\theta^*_{T,G})$,
for all $b \in \R^p $, $\mathfrak{N}\left(\theta^*_{T,G}+\diag(\hat{\Sigma}_T)^{\frac{1}{2}}\frac{b}{\sqrt{T}};\theta^*_{T,G}; \diag(\hat{\Sigma}_T)^{\frac{1}{2}}/\sqrt{T} \right)\\= \frac{1}{\sqrt{2 \pi T^p |\diag(\hat{\Sigma}_T)|_{\det}}}\int_{-\infty}^{\theta^*_{T,G}+\diag(\hat{\Sigma}_T)^{\frac{1}{2}}\frac{b}{\sqrt{T}}}\exp\left[-\frac{T}{2}(\theta-\theta^*_{T,G})'\diag(\hat{\Sigma}_T)^{-1}(\theta-\theta^*_{T,G}) \right] \lambda(\d \theta) \\= \frac{1}{\sqrt{2 \pi}}\int_{-\infty}^{b}\exp\left[-\frac{1}{2}u'u \right] \lambda(\d u)=\mathfrak{N}(b;0;1)$. 
\end{proof}

\subsubsection{Consistency of  Laplace approximations } \label{ap7_2_3}

In this subsection, under mild technical assumptions, we deduce the consistency  of weighted and criterion-adjusted Laplace approximations from the Bernstein-von Mises theorem, i.e., Assumption \ref{assp14}(c).

\begin{assp} \label{assp14}
\emph{\textbf{ (a)}}For $T $ big enough, $w(.)\geqslant 0 $ $\P$-a.s.  \emph{\textbf{(b)}} For $T $ big enough, $0 < \int_\T\e^{TQ_T(X_{1:T},\dot\theta)}w(\dot \theta)\lambda(\d \dot \theta)<\infty $ $\P$-a.s.   \emph{\textbf{(c)}} [Bernstein-von Mises theorem]  For all $b \in \R^p $, as $T \rightarrow \infty $,
\begin{eqnarray*}
\hat{F}_{\theta^\bullet_T} \left(\theta^*_{T,L}+\frac{b}{\sqrt{T}}\right) \stackrel{\P}{\rightarrow} \mathfrak{N}(b;0;\Sigma^{\frac{1}{2}})
\end{eqnarray*}
where $\hat{F}_{\theta^\bullet_T} (\theta):=\int_{-\infty}^\theta \frac{\e^{TQ_T(X_{1:T},\ddot\theta)}w(\ddot\theta)}{\int_\T\e^{TQ_T(X_{1:T},\dot\theta)}w(\dot \theta)\lambda(\d \dot \theta)} \lambda(\ddot \theta)$, and $ \theta^*_{T,L} \in \arg \max_{\ddot \theta \in \T}\e^{TQ_T(X_{1:T},\ddot\theta)}$.  \emph{\textbf{(d)}}  For all $x_{1:\infty} \in \underline{\Sbf}_\infty $, $\e^{TQ_T(x_{1:T},.)}w(.) $ is  $ \mathcal{E}_{\T}/\mathcal{B}(\R_{+})$-measurable. 
\end{assp}

Assumption \ref{assp14}(c) has been established under general assumptions (e.g., Le Cam, 1953, 1958;   Chen, 1985; Kim, 1998; Chernozhukov and Hong, 2003). Assumption \ref{assp14} (a)(b) ensures that $\hat{F}_{\theta^\bullet_T}(.) $ is a c.d.f. $\P$-a.s. for $T$ big enough. For the consistency of the criterion-adjusted weighted Laplace approximation, we also require the following Assumption \ref{assp15}. Assumption \ref{assp14}(d)
is a weak requirement that ensures the existence of random variables with a distribution specified through $\hat{F}_{\theta^\bullet_T}(.)$. 
\begin{assp} \label{assp15} \emph{\textbf{ (a)}} For all $\theta \in \T$, $\dot \theta \mapsto u(\theta, \dot \theta) $ is continuous and bounded.
\\ \emph{\textbf{(b)}} There exists a function $h:\T \rightarrow  \R $ s.t., for all $\theta \in \T $, for $T $ big enough,  $\int_{\T}u(\theta, \ddot \theta)\e^{TQ_T(X_{1:T},\ddot\theta)}w(\ddot\theta)\lambda(\d \ddot \theta)< h(\theta) $ $\P$-a.s.,  and $\int_T h(\theta) \lambda(\d \theta)< \infty $. 
\end{assp}

Assumption \ref{assp15}(a) typically follows from the continuity of $u(.,.) $, and the compactness of $\T$, while Assumption \ref{assp15}(b) typically follows from the assumptions on $Q_T(X_{1:T},\ddot\theta) $, that are needed to establish Assumption \ref{assp14}(c). 

\begin{prop} 
\label{prop7}
Under Assumptions \ref{assp1}, \ref{assp4}, \ref{assp13}, and \ref{assp14}, if, as $T \rightarrow \infty $, $\theta^*_{T,L} \stackrel{\P}{\rightarrow} \theta_0 $,
then 
 \begin{enumerate}
\item[i)] $\rho_P\left( \hat{F}_{\theta^\bullet_{T,WL}} (.), \delta_{\theta_0} \right) \stackrel{\P}{\rightarrow} 0 $,  as $T \rightarrow \infty $, which, in turn, implies consistency of the weighted Laplace  approximation by Lemmas \ref{lem4} and \ref{lem5};
\item[ii)] under the additional Assumptions  \ref{assp12}(a) and \ref{assp15}, as $T \rightarrow \infty $,

$\rho_P\left(  \frac{\int_{.}\int_{\T}u(\theta, \ddot \theta)\e^{TQ_T(X_{1:T},\ddot\theta)}w(\ddot\theta)\lambda(\d \ddot \theta)\lambda(\d \theta )}{\int_{\T^2} u(\theta, \dot \theta)\e^{TQ_T(X_{1:T},\dot\theta)}w(\dot \theta)\lambda(\d \dot \theta)\lambda(\d \theta )} ,\frac{\int_{.} u(\theta, \theta_0)\lambda(\d \theta)}{\int_\T u(\dot \theta, \theta_0)\lambda(\d \dot \theta)} \right) \stackrel{\P}{\rightarrow} 0 $, which, in turn, implies consistency of the criterion-adjusted weighted  Laplace  approximation by Lemmas \ref{lem4} and \ref{lem5}.

\end{enumerate}

\end{prop}

\begin{proof} i) Check assumptions of  Lemma \ref{lem8} with   $\tilde{\Sigma}_T=I $, and   for $\tilde{\Sigma}=I  $, and then apply it. Assumption \ref{assp16} is verified because, under Assumptions \ref{assp14}(a)(b)(d), by the upcoming Lemma \ref{lem63}, for $T  $ big enough, $\widehat{\P \circ {\theta^{\bullet}_{T, WL}}^{-1}}:=\int_{.} \frac{\e^{TQ_T(X_{1:T},\ddot\theta)}w(\ddot\theta)}{\int_\T\e^{TQ_T(X_{1:T},\dot\theta)}w(\dot \theta)\lambda(\d \dot \theta)} \lambda(\ddot \theta)$ defines a random probability measure from $\underline{\Sbf}_{\infty} $ to $(\T, \mathcal{E}_\T) $ .
Assumptions (a) and (b) of Lemma \ref{lem8} are verified by assumption of Proposition \ref{prop7}  and Assumption \ref{assp14}(c), respectively. 

ii) By i), as $T \rightarrow \infty $,  $\rho_P\left( \hat{F}_{\theta^\bullet_T} (.), \delta_{\theta_0} \right) \stackrel{\P}{\rightarrow} 0$. Thus, by definition of the convergence in law and Assumption \ref{assp15}(a), for all $\theta\in \T $,
as $T \rightarrow \infty $, \begin{eqnarray*}
& \stackrel{}{\Rightarrow} & \int_{\T}u(\theta, \ddot \theta)\frac{\e^{TQ_T(X_{1:T},\ddot\theta)}w(\ddot\theta)}{\int_\T\e^{TQ_T(X_{1:T},\dot\theta)}w(\dot \theta)\lambda(\d \dot \theta)}\lambda(\d \ddot \theta)  \stackrel{\P}{\rightarrow}    u(\theta, \theta_0)\\
 & \stackrel{(a)}{\Rightarrow} &   \forall B \in  \mathcal{E}_{\T},\, \int_B\int_{\T}u(\theta, \ddot \theta)\frac{\e^{TQ_T(X_{1:T},\ddot\theta)}w(\ddot\theta)}{\int_\T\e^{TQ_T(X_{1:T},\dot\theta)}w(\dot \theta)\lambda(\d \dot \theta)}\lambda(\d \ddot \theta)\lambda(\d \theta)  \stackrel{\P}{\rightarrow}    \int_B u(\theta, \theta_0)\lambda(\d \theta)\\
& \stackrel{(b)}{\Rightarrow} & \forall B \in  \mathcal{E}_{\T}, \frac{\int_{B}\int_{\T}u(\theta, \ddot \theta)\e^{TQ_T(X_{1:T},\ddot\theta)}w(\ddot\theta)\lambda(\d \ddot \theta)\lambda(\d \theta )}{\int_{\T^2} u(\theta, \dot \theta)\e^{TQ_T(X_{1:T},\dot\theta)}w(\dot \theta)\lambda(\d \dot \theta)\lambda(\d \theta )}   \stackrel{\P}{\rightarrow}\frac{\int_{B} u(\theta, \theta_0)\lambda(\d \theta)}{\int_\T u(\dot \theta, \theta_0)\lambda(\d \dot \theta)}\\
& \stackrel{(c)}{\Rightarrow} & \rho_P\left(  \frac{\int_{.}\int_{\T}u(\theta, \ddot \theta)\e^{TQ_T(X_{1:T},\ddot\theta)}w(\ddot\theta)\lambda(\d \ddot \theta)\lambda(\d \theta )}{\int_{\T^2} u(\theta, \dot \theta)\e^{TQ_T(X_{1:T},\dot\theta)}w(\dot \theta)\lambda(\d \dot \theta)\lambda(\d \theta )} ,\frac{\int_{.} u(\theta, \theta_0)\lambda(\d \theta)}{\int_\T u(\dot \theta, \theta_0)\lambda(\d \dot \theta)} \right) \stackrel{\P}{\rightarrow} 0
\end{eqnarray*}
\textit{(a) }
By Assumption \ref{assp15}(b), apply
 Lebesgue dominated convergence theorem.
\textit{(b) } Note that simplification of the denominators yields $\frac{\int_{B}\int_{\T}u(\theta, \ddot \theta)\e^{TQ_T(X_{1:T},\ddot\theta)}w(\ddot\theta)\lambda(\d \ddot \theta)\lambda(\d \theta )}{\int_{\T^2} u(\theta, \dot \theta)\e^{TQ_T(X_{1:T},\dot\theta)}w(\dot \theta)\lambda(\d \dot \theta)\lambda(\d \theta )} \\=\frac{\int_{B}\int_{\T}u(\theta, \ddot \theta) \frac{\e^{TQ_T(X_{1:T},\ddot\theta)}w(\ddot\theta)}{\int_\T\e^{TQ_T(X_{1:T},\dot\theta)}w(\dot \theta)\lambda(\d \dot \theta)}\lambda(\d \ddot \theta)\lambda(\d \theta )}{\int_{\T^2} u(\theta, \ddot \theta)\frac{\e^{TQ_T(X_{1:T},\ddot\theta)}w(\ddot\theta)}{\int_\T\e^{TQ_T(X_{1:T},\dot\theta)}w(\dot \theta)\lambda(\d \dot \theta)}\lambda(\d \ddot \theta)\lambda(\d \theta )} $, and then apply twice the previous line, with $B=B $ and $B =\T $.
\textit{(c)} Apply Portmanteau theorem.  
\end{proof}
%

\begin{lem} \label{lem63} If, for all $x_{1:\infty} \in \underline{\Sbf}_\infty $, a function $g:\underline{\Sbf}_\infty \times \T \rightarrow \R $ is  $ \mathcal{E}_{\T}/\mathcal{B}(\R)$-measurable   on $\T $,  then, for all $A\in  \mathcal{E}_{\T}$,  $x_{1:\infty}\mapsto \int_{A}g(x_{1:\infty},\theta)  \lambda( \d\theta ) $ is $\underline{\mathcal{S}}_\infty/\mathcal{B}(\R) $-measurable. 

\end{lem}
\begin{proof}
 Let $A\in  \mathcal{E}_{\T}$. Define for this proof\begin{eqnarray*}
\mathcal{H}_{A}:= \left\{ h(.,.): \begin{matrix} h:\underline{\Sbf}_\infty \times \T \rightarrow \R  \\
\forall \theta \in \T, h(.,\theta)\text{ is  $\underline{\mathcal{S}}_\infty/\mathcal{B}(\R)$-measurable   }  \\
x_{1:\infty}\mapsto \int_{A}h(x_{1:\infty},\theta)  \lambda( \d\theta ) \text{ is } \underline{\mathcal{S}}_\infty/\mathcal{B}(\R)\text{-measurable} \\
\end{matrix}\right\} 
\end{eqnarray*}
Check the assumptions of a functional form of Sierspinki monotone class theorem(e.g.,   Florens, Mouchart and Rolin, 1990, Theorem 0.2.21). First, $\mathcal{H}_A$ is a $\R$-vector space because measurability is preserved by  linear combinations, and because the integral of a linear combination  of functions is the linear combination  of the integrals of the respective functions.  Second, $\mathcal{H}_A$ contains the constant function $1$. Third, if $(h_n(.) )_{n\geqslant 1}$ is a sequence of non-negative functions in $\mathcal{H}_A$ such that $h_n(.)\uparrow h(.) $ where $h(.) $ is a bounded function on $\T $, then  $h(.)\in \mathcal{H}_A$ by  preservation of  measurability  under limit  and the Lebesgue monotone convergence theorem. Fourth, $\mathcal{H_{A}} $ contains the indicator function of every set in the $\pi$-system consisting of measurable rectangles, $\mathcal{I}:= \left\{R\cap\T: R:= \tilde{R}\cap \T \text{ with }  \right.$   $\left. \tilde{R}=\prod_{i=1}^p[a_i,b_i] \text{ where }(a_i,b_i)\in \R^2 \wedge a_i \leqslant b_i \right\} $ (an intersection of two measurable rectangles is a measurable rectangle)  because, for all $R \in \mathcal{I} $, (i) $\forall \theta \in \T,  $ $x_{1:\infty}\mapsto\ind_R(\theta) $ is   $\underline{\mathcal{S}}_\infty/\mathcal{B}(\R)$-measurable, (ii) and for all $A \in  \mathcal{E}_{\T} $, $x_{1:\infty}\mapsto \int_{A}\ind_R(\theta)  \lambda( \d\theta ) =\lambda(A \cap R)< \lambda(\T)<\infty$  is  $\underline{\mathcal{S}}_\infty/\mathcal{B}(\R)$-measurable.

As $\sigma(\mathcal{I})= \mathcal{E}_{\T}$, by  a functional form of Sierspinki monotone class,  if a function $g(.,.)$  is  $ \mathcal{E}_{\T}/\mathcal{B}(\R)$-measurable,  $g\in\mathcal{H}_A$, and thus $x_{1:\infty}\mapsto \int_{A}g(x_{1:\infty},\theta)  \lambda( \d\theta ) $ is $\underline{\mathcal{S}}_\infty/\mathcal{B}(\R) $-measurable. \end{proof}

\section{Notations and overview of some practices}\label{ap9}
\newpage

\begin{landscape}

\begin{table} \caption{\textbf{Overview of some calibration and econometric practices (Section \ref{sec5_2})}}
\begin{tiny}
 \label{tab1}
\begin{tabular}{lllll}\hline \hline
Approximation &$\theta^*_T $ & $\propto \hat{f}_{\theta^\bullet_T}(\theta)$ $[\mu] $& $\check{\theta}_T$ m.a.e.  & $\propto \hat{f}_{\theta^\bullet_\infty}(\theta)$ $[\mu] $   \\\hline
Plain calibration& $\theta^*_{T,C} $ & $\ind_{\{ \theta^*_{T,C}\}}(\theta) $ $[\nu] $& $\theta^*_{T,C} $ & $\ind_{\{\theta_0\}} (\theta)$ $[\nu] $  \\
Criterium-adj. calibration &  $F^{-1}_{\theta^*_{T,CC}|\theta^{*}_{T,C}}(U^{\bullet}_{CC}) $ 
& $u(\theta, \theta^*_{T,C}) $ $[\lambda]$& $\displaystyle \arg \max_{\theta \in \T} u(\theta, \theta^*_{T,C})$ & $u(\theta, \theta_0) $ $[ \lambda] $ \\
Gaussian & $\theta^*_{T,G} $ & $ \exp\left[-\frac{T}{2}(\theta-\theta^*_{T,G}(\omega))'\diag(\Sigma^*_T(\omega))^{-1}(\theta-\theta^*_{T,G}(\omega)) \right]  $ $[\lambda] $& $\theta^*_{T, G} $ & $\ind_{\{\theta_0\}} (\theta)$ $[\nu] $\\
Plain Laplace & $\displaystyle\arg \max_{\theta \in \T} \e^{TQ_T(X_{1:T},\theta)}$ & $\e^{TQ_T(X_{1:T},\theta)} $ $[ \lambda] $ & $\theta^*_{T,L} $ & $\ind_{\{\theta_0\}} (\theta)$ $[\nu] $ \\
Weighted Laplace & $\displaystyle\arg \max_{\theta \in \T} \e^{TQ_T(X_{1:T},\theta)}w(\theta)$ & $\e^{TQ_T(X_{1:T},\theta)}w(\theta) $ $[ \lambda] $ & $\theta^*_{T,WL}$ & $\ind_{\{\theta_0\}} (\theta)$ $[\nu] $ \\
Criterium-adj. weighted Laplace& $F^{-1}_{\theta^*_{T,CWL}|\theta^{*}_{T,WL}}(U^{\bullet}_{CWL}) $ & $\int_\T u(\theta, \dot \theta)\e^{TQ_T(X_{1:T},\dot\theta)}w(\dot\theta) \lambda(\d \dot \theta) $ $[ \lambda] $  & $\displaystyle\arg \max _{\theta \in \T}\int_\T u(\theta, \dot \theta)\e^{TQ_T(X_{1:T},\dot\theta)}w(\dot\theta) \lambda(\d \dot \theta) $ & $u(\theta, \theta_0) $ $[ \lambda] $ \\\hline
\end{tabular}

\end{tiny}
\end{table}

\begin{table} \caption{\textbf{Main notations}}
\begin{tiny}
 \label{tab2}
\begin{tabular}{lll}
\hline
\hline
Notation & Description  \\ 
\hline
$\theta^*_T $  & Proxy of the unknown parameter $\theta_0 $ \\ 
$\theta^\bullet_T $  & Generic proxy of the unknown parameter $\theta_0 $  \\ 

$U^\bullet $  & Random variable uniformly distributed on $[0,1] $ that is independent from the data $X_{1:T} $ \\ 
$ \hat{f}_{\theta^\bullet_T}(.)$ & Approximation of the probability density function  of $\theta^\bullet_T $, i.e., $\frac{\d \hP}{\d \mu}(.) $ \\ 
$\lambda$ & Lebesgue measure  \\ 
$\nu$ & Counting measure  \\ 
$\check{\theta}_T$ & Neoclassical estimate  \\
$F_{\theta^*_T}$ & Cumulative distribution function of $\theta^*_T $  \\
$R_{1-\alpha,T} $& Neoclassical confidence region m.a.e. \\
$u(.,.)$ & Criterium function from $\T^2$ to $\R_+ $ \\
$w(.)$ & Weighting function from $\T$ to $\R_+ $, i.e., $\frac{\d (\widehat{\P\circ {\theta^\bullet_{T, WL}}^{-1}})}{\d (\widehat{\P\circ {\theta^\bullet_{T, L}}^{-1}})}(.)$ \\
$\[a,b \]$ & $\{c\in [a,b] \cap \N \} $ \\

$\mathcal{N}(\bar m, s) $ & Gaussian distribution with mean $\bar m$ and standard deviation $s $. \\

$\mathfrak{N}\left(\theta; \tau ;  s\right) $ & C.d.f. of $\mathcal{N}(\tau, s) $, i.e,  $\int_{-\infty}^\theta \mathfrak{n}\left(\dot \theta; \tau ;  s\right) \lambda(\d \dot \theta)$\\
\hline
\end{tabular}
\end{tiny}
\end{table}

\end{landscape}

\end{document}